\documentclass{amsart}

\usepackage[utf8]{inputenc}
\usepackage[a4paper, margin=1.1in]{geometry}    
\usepackage[dvipsnames]{xcolor}                 
\usepackage{graphicx}                           
\usepackage{verbatim}                           
\usepackage{todonotes}                          
\usepackage[colorlinks=true, linkcolor=blue, citecolor=blue, urlcolor=blue]{hyperref}  
\usepackage{upgreek}                            
\usepackage[linesnumbered,ruled,vlined]{algorithm2e}
\usepackage[foot]{amsaddr}                      
\usepackage{tpsmacros} 
\usepackage[capitalise,noabbrev]{cleveref}
\usepackage{bbm}
\usepackage{subfigure}
\usepackage[numbers]{natbib}

\theoremstyle{plain}
\newtheorem{theorem}{Theorem}[section]

\newtheorem{proposition}[theorem]{Proposition}

\newtheorem{assumption}[theorem]{Assumption}

\theoremstyle{definition}

\theoremstyle{remark}
\newtheorem{remark}[theorem]{Remark}

\DeclareMathOperator{\tr}{tr}
\DeclareMathOperator{\im}{im}
\DeclareMathOperator{\spn}{span}
\DeclareMathOperator{\ESS}{ESS}
\DeclareMathOperator{\Unif}{Unif}

\title{Guided filtering and smoothing for infinite-dimensional diffusions}

\author{Thorben Pieper-Sethmacher$^{1,*}$} 
\email{\texttt{thorbenpieper@tudelft.nl}}

\author{
Daniele Avitabile$^{2,3}$}
\email{\texttt{d.avitabile@vu.nl}}

\author{Frank van der Meulen$^2$}
\email{\texttt{f.h.van.der.meulen@vu.nl}}
\address{$^1$ Delft Institute for Applied Mathematics \\
  Delft University of Technology \\
    Mekelweg 4, 2628 CD Delft \\
    The Netherlands.}
\address{$^2$ Department of Mathematics \\ Vrije Universiteit Amsterdam \\ De Boelelaan 1111, 1081 HV Amsterdam \\ The Netherlands.}
\address{$^3$ MathNeuro Team \\ Inria branch of the University of Montpellier \\ 860
rue Saint-Priest 34095 Montpellier Cedex 5 \\ France.}

\thanks{$^*$Corresponding author: Thorben Pieper-Sethmacher,
\texttt{thorbenpieper@tudelft.nl}.}

\begin{document}

\begin{abstract}
We consider the filtering and smoothing problems for an infinite-dimensional
diffusion process $X$, observed through a finite-dimensional representation at
discrete points in time. At the heart of our proposed methodology lies the
construction of a path measure, termed the guided distribution of $X$, that is
absolutely continuous with respect to the law of $X$, conditioned on the
observations. We show that this distribution can be incorporated as a potent proposal
measure for both sequential Monte Carlo as well as Markov Chain Monte Carlo schemes
to tackle the filtering and smoothing problems respectively. In the offline setting,
we extend our approach to incorporate parameter estimation of unknown model
parameters. The proposed methodology is numerically illustrated in a case study for a stochastic neural field equation. 
\end{abstract}

\maketitle

\section{Introduction}
Enriching model predictions of a stochastic dynamical system using
partial and noisy observations is an attractive idea, with numerous applications in
the physical, engineering, and biological sciences
\cite{smithUncertaintyQuantificationTheory2014, 
 Law2015,
nakamuraInverseModelingIntroduction2015,
Reich2015,
Sanz-alonsoInverseProblemsData2023}. From a mathematical and statistical viewpoint, such \textit{data assimilation} methods constitute inverse problems in which the unknown may be the state of the system, its model parameters, or a combination thereof. If the underlying system is spatially independent, it is governed by a \textit{stochastic differential equation} (SDE) in a finite-dimensional state space, with observations pertaining the state variables at selected observation times. 
The mathematical foundations for these problems have been developed over the past decade, and are currently available in several texts \cite{bain2009fundamentals, Law2015, Reich2015, Sanz-alonsoInverseProblemsData2023}.

The present paper concerns data assimilation for systems governed by \textit{infinite-dimensional SDEs}. Such systems include spatiotemporal processes modelled by \textit{stochastic partial differential equations (SPDEs)} as well as spatially extended SDEs in which spatial coupling is induced from nonlocal interactions. This setting poses technical challenges in view of the dimensionality of the state space, and fewer studies in this area are currently available. We introduce a unified framework for two classical data assimilation tasks -- \textit{filtering} and \textit{smoothing} -- as well as \textit{parameter estimation} for SDEs in Hilbert spaces. Our framework differs from existing methods by treating filtering and smoothing in a rigorous abstract form, taking an infinite-dimensional dynamical system approach to the data assimilation problem. 
This aligns with the ``discretisation-free'' formulation of inverse problems as advocated in \cite{Stuart2010Inverse} and implies that, contrary to approaches where the SDE is first discretised, our methods are well defined under increasing mesh refinements.
 We demonstrate our methodology in a case study for the stochastic \textit{Wilson--Cowan--Amari} equation, estimating parameters that were previously assumed to be known in the literature.

\subsection{Setup}\label{subsec:setup}
Consider a system governed by an infinite-dimensional diffusion equation of the form
\begin{equation}\label{eq: dXt} 
\begin{cases}
\df X_t &= \left[ A X_t + F(t, X_t) \right] \df t + Q^{\frac12} \df W_t, \quad  t \geq 0,\\ 
 \quad X_0 &\sim \mu_0.
\end{cases}
\end{equation}
Here, $A$ denotes the generator of a strongly continuous semigroup $(S_t)_{t \geq 0}$ on a Hilbert space $H$, whereas $F$ is a continuous nonlinearity and $Q$ a symmetric, positive operator of trace class on $H$. The process $(W_t)_{t \geq 0}$ is a cylindrical Wiener process on $H$, defined on a stochastic basis $(\Omega, \calF, (\calF_t)_{t \geq 0},\P)$ and $\mu_0$ denotes a Borel measure on $H$.

We assume that \cref{eq: dXt} admits a unique mild solution $X$ satisfying
\begin{align}
\label{eq: Xt}
X_t &= S_t X_0 + \int_0^t S_{t-s} F(s, X_s) \df s + \int_0^t S_{t-s} Q^{\frac12} \df W_s.
\end{align}
The process $X$ is an $H$-valued, predictable Markov process with transition kernel $(\mu_{s,t})_{s \leq t}$ defined by $\mu_{s,t}(x,B) := \P(X_t \in B \mid X_s = x)$ for all Borel sets $B \subset H$ and $x \in H$.

Suppose at discrete observation times $0 < t_1 < t_2 < \ldots < t_n = T$, we are given \textit{finite-dimensional observations} $y_i$ through realisations $Y_i = y_i$ of 
\begin{align}
\label{eq: def_Y}
    Y_i \mid X_{\ti} &\sim k_i(X_{\ti}, \cdot), \quad i = 1, \ldots, n.
\end{align}
Here, $k_i(X_{\ti}, \cdot)$ denotes a conditional density on $\R^{m_i}$, where $m_i$ is the dimension of the observation at time $t_i$. 
For the sake of convenience, we assume from now on that
\begin{align}
\label{eq: def_Y2}
    Y_i \mid X_{\ti} \sim f(\cdot; L X_{\ti}, \Sigma), \quad i = 1, \ldots, n,
\end{align}
where $L: H \to \R^m$ is a bounded linear operator, termed the \textit{observation operator}, $\Sigma$ is a positive-definite matrix and $f(\cdot; Lx, \Sigma)$ is the Gaussian density with mean $Lx$ and covariance $\Sigma$.
However, our methods generalise to non-Gaussian observation densities and we will discuss how to adapt to such cases towards the end of the paper.

Based on the described setup, we address three estimation problems. 
Firstly, we are concerned with state estimation of the unobserved, latent path $X$. 
Here, we consider both the online problem of estimating the filtering distribution 
\begin{align}
\label{eq: def_filtering}
    \P(X_{\ti} \in B \mid Y_1, \ldots, Y_i), \quad i = 1, \ldots, n,
\end{align}
and as well as the offline problem of estimating the smoothing distribution
\begin{align}
\label{eq: def_smoothing}
    \P(X_{t} \in B \mid Y_1, \ldots, Y_n), \quad t \in [0,T],
\end{align}
for any Borel set $B \subset H$.
Moreover, if \cref{eq: dXt} is parametrised by some unknown model parameter $\theta \in \R^p$, we address parameter estimation of $\theta$ in the offline setting.

\subsection{Related Work}
In the context of SPDEs, there has been a growing interest in statistical inference for infinite-dimensional SDEs. As the survey \cite{cialenco2018statistical} shows, the literature mostly concerns models admitting closed-form parameter estimators studied under
increasing-information asymptotics. 
Additionally, results for nonlinear SPDEs are sparse. In the frequentist framework, inference of a drift parameter for the two-dimensional stochastic Navier-Stokes
equation has been considered in \cite{Cialenco11Parameter} and for semi-linear SPDEs
in \cite{Stannat20Drift}, \cite{Altmeyer2023Parameter}, and
\cite{Cialenco2024Statistical}. 
Estimation concerning the nonlinearity has been treated in the parametric case in 
 \cite{Gaudlitz2023Estimation} and in the non-parametric case in \cite{Gaudlitz2023Non, HildebrandtTrabs2023Nonparametric}.
 
In contrast to the frequentist statistical perspective, ``results on filtering in the context of SPDEs are rather scarce" and ``Bayesian inference for SPDEs is another area with few existing results" (\cite{cialenco2018statistical}, page 326). 
Both of these problems are commonly found in the literature for spatiotemporal statistics. Here, formulating so-called ‘physically inspired spatiotemporal models’ as solutions to linear SPDEs has been proposed as a more computationally efficient approach compared to conventional methods based on Gaussian process regression, see \cite{Lindgren2011Explicit}, \cite{Sarkka2013Spatiotemporal}, and \cite{sigrist2015spde}. A recent introduction to this framework can be found in Chapter 5 of \cite{Wikle2019Spatio}.
However, the results in this direction are limited to linear SPDEs and do not generalise to nonlinear and non-Gaussian signals. 

The few works that address data assimilation for nonlinear SPDEs mostly concern the filtering problem. 
In \cite{Llopis2018Particle}, a guided particle filter, including tempering and resample-move steps, has been introduced to estimate the filtering distribution of a two-dimensional stochastic Navier-Stokes signal. Moreover, a similar approach was used in \cite{Lang2022Bayesian} in the context of filtering for the stochastic rotating shallow water (SRSW) model and in \cite{Cotter2025camassaholm} for the Camassa-Holm equations. Likewise, our approach to the filtering problem expands on the methodology proposed in \cite{Llopis2018Particle}. 

Smoothing and Bayesian parameter inference for
infinite-dimensional diffusions in the generality of our setup has, to the best of
our knowledge, not been considered previously in the literature.
In the special case of a single, noiseless observation and known model parameters, smoothing is reduced to sampling of an infinite-dimensional diffusion bridge. This has been investigated for strong solutions to infinite-dimensional SDEs in \cite{Yang2024Simulating} and for mild solutions in \cite{Piepersethmacher2025Simulation}.

As stated above, our paper addresses also processes that evolve according to spatially extended SDEs; these arise naturally when the dynamical system of interest is inherently nonlocal,
as in models coming from neuroscience
\cite{izhikevichDynamicalSystemsNeuroscience2006,
ermentroutMathematicalFoundationsNeuroscience2010,
bressloffWavesNeuralMedia2014, coombesNeurodynamicsAppliedMathematics2023},
ecology~\cite{volpertEllipticPartialDifferential2014},
and solid
mechanics~\cite{duNonlocalModelingAnalysis2019,deliaNonlocalIntegralEquation2024}. 
The infinite-dimensional SDE that we use as an application of our theory is a stochastic version of the
Wilson--Cowan--Amari equation~\cite{wilson1973mathematical,amari1977dynamics}, also
known as \textit{neural field equation}. In this model the cortex is represented as a
continuum, and nonlocality manifests itself through a nonlinear, nonlocal operator
collecting global neuronal activity. There are
multiple reasons to target this system: (i) it is a tractable example of a 
parabolic, nonlocal, and \textit{non-diffusive} system, which sets it apart from the
other applications mentioned above, modelling nonlocal diffusion or nonlocal
reaction; (ii) in spite of its simplicity, it shares the salient mathematical features
of more sophisticated neuroscience models; (iii) it is capable of supporting many
spatiotemporal patterns observed in healthy and pathological brains, and crucially
(iv) the mathematical neuroscience community has begun studying data assimilation
problems for these types of systems, prompted by available brain datasets
collected in the past decades.

Wilson--Cowan--Amari models subject to stochastic forcing have been introduced more
recently than their PDE counterparts, with literature focusing on their analysis
\cite{
kilpatrick2013wandering,
kilpatrickPulseBifurcationsStochastic2014,
kuehnLargeDeviationsNonlocal2014,
faugerasStochasticNeuralField2015,
MacLaurin.2020}, and numerical simulation
\cite{limaNumericalInvestigationStochastic2019,limaNumericalSolutionStochastic2022a}.
Further, inverse
problems for neural fields have been studied in the context of kernel reconstructions
\cite{potthastInverseProblemsNeural2009, nakamuraInverseModelingIntroduction2015,
alswaihliKernelReconstructionDelayed2018}, and more specifically in data assimilation
using Kalman filters with in-vitro
\cite{Schiff2007Kalman,sauerDataAssimilationHeterogeneous2009} and synthetic 
\cite{kulikovaDatadrivenParameterEstimation2023} data. 
In contrast to the existing literature, we treat the data assimilation problem in continuous space and time, rather than relying on a spatiotemporal discretisation of the neural field model.

 \vspace{\baselineskip}

\subsection{Approach}
\label{sec: approach}
Our approach relies on the construction of \textit{guided processes}, a tractable class of processes obtained by steering the solution $X$ based on the observed data. For SDEs with Euclidean state-spaces, the idea of guided processes goes to back to \cite{Clarke} and \cite{DelyonHu}, the terminology being first introduced in \cite{PapaspiliopoulosRoberts2012}. A related technique is known as \textit{nudging} in the data assimilation literature. We comment on this connection in Appendix \ref{app: C}.

Here, we build specifically upon the work on guided processes as defined in \cite{schauer2017guided}, \cite{mider2021continuous} and \cite{vdMeulenSchauerSommer2025}. To lift this approach to the infinite-dimensional setting, we rely heavily on the exponential changes of measure studied in \cite{Piepersethmacher2025Class}.

Consider for now the case that $n=1$, i.e. we
observe one realisation $y$ of the random variable $Y \mid X_T \sim k(X_T, \cdot)$.
In this simplified case, estimating the conditional distribution of $X$ given $Y = y$
solves both the filtering as well as the smoothing problem. 

It is well-known that this conditional distribution can be derived by a change of measure known as \textit{Doob's h-transform} as follows. Defining for any fixed $y \in \R^{m}$ the likelihood function 
\begin{align*}
h(t,x) &= \int_H k(z,y) \mu_{t,T}(x, \df z), \quad t \in [0,T], x \in H,
\end{align*}
one shows that the process $(h(t,X_t))_{t \in [0,T]}$ is a non-negative $\P$-martingale and hence defines a unique change of measure $\P^h$ on $\calF_T$ by 
\begin{align*}
\df \P^h_{\mid \calF_T} &= \frac{h(T,X_T)}{C^h} \df \P_{\mid \calF_T},
\end{align*}
with $C^h := \E[h(0,X_0)]$ acting as a normalising constant. Under the new measure $\P^h$, the law $\calL^h(X)$ of $X$ is exactly the conditional distribution of $X$ given $Y = y$ in the sense that 
\begin{align*}
	\E^h[\varphi(X)]= \E[\varphi(X) \mid Y = y]
\end{align*} for any bounded and measurable function $\varphi$.
Moreover, under certain regularity conditions on $h$, it was shown in \cite{Piepersethmacher2025Class} that the change of measure $\P^h$ is of Girsanov type. The process $X$ under $\P^h$ is then a mild solution to 
\begin{align}
\label{eq: dXh}
\df X^h_t &= \left[ A X^h_t + F(t, X^h_t) + Q \Df_x \log h (t,X^h_t) \right] \df t + Q^{\frac12} \df W^h_t,
\end{align}
where $W^h$ is a $\P^h$-cylindrical Wiener process. 
\cref{eq: dXh} differs from the original dynamics in \cref{eq: dXt} only by an additional drift term involving the \textit{score} $\Df_x \log h (t,x)$ of the likelihood function $x \mapsto h(t,x)$ that steers the process $X^h$ into regions of high probability given the observation $Y = y$. 
With slight abuse of notation, we write $X^h$ for the process $X$ defined on the stochastic basis $(\Omega, \calF, (\calF_t)_{t \in [0,T]},\P^h)$ and call it the \textit{conditioned process}. Note that the existence of $\P^h$ and $X^h$ does not rely on the existence of a mild solution to Equation \eqref{eq: dXh}.

\vspace{\baselineskip}

In most cases, the transition kernel $\mu$ of $X$, and hence the function $h$, is intractable. This renders directly sampling the conditioned process infeasible, even if conditions for the well-posedness of Equation \eqref{eq: dXh} are satisfied. 
A natural approach to overcome this problem is to substitute the intractable function $h$ by a tractable function $g$ that `approximates' $h$ in some sense.
Besides being tractable, the function $g$ should
\begin{itemize}
    \item[(i)] be informed by the observed data in a way that resembles the original function $h$,
    \item[(ii)] be such that there exists a mild solution $X^g$ to the equation
  \begin{align}
  \label{eq: dX^g}
  	    \df X^g_t &= \left[ A X^g_t + F(t, X^g_t) + Q \Df_x \log g (t,X^g_t) \right] \df t + Q^{\frac12} \df W^g_t,
  \end{align}
    \item[(iii)] be such that the laws of $X^h$ and $X^g$ are absolutely continuous in path space with a tractable Radon-Nikodym (R.N.) derivative $\Phi$.
\end{itemize}
One can then obtain \textit{weighted samples} of the conditioned process $X^h$ by sampling the process $X^g$ and associating with each sample the weight $\Phi(X^g)$. 
Since $X^g$ is `guided' by the observation $Y = y$, we refer to it as the \textit{guided process} and its law as the \textit{guided distribution}.

\vspace{\baselineskip}

In our approach, we define $g$ by substituting the transition kernel $\mu$ in the definition of $h$ with a tractable transition kernel $\nu$ of an auxiliary process $Z$. If $Z$ is an Ornstein-Uhlenbeck process, we show that the function $g$ defines a change of measure on $\calF_T$ by 
\begin{align*}
\df \P^g_{\mid \calF_T} &= \frac{g(T,X_T)}{C^g} \exp \left(- \int_0^T \dfrac{\calK g}{g}(s,X_s) \, \df s \right) \df \P_{\mid \calF_T}
\end{align*}
such that $X$ under $\P^g$ satisfies Equation \eqref{eq: dX^g} with respect to a $\P^g$-cylindrical Wiener process $W^g$.
Here, $\calK$ denotes the infinitesimal generator of the original, unconditioned process $X$ and $C^g$ is another normalising constant. 

 The function $x \mapsto g(t,x)$ is to be interpreted as the likelihood function of $Z_t = x$ under the simplified dynamics of $Z$ and therefore satisfies condition (i). Moreover, condition (ii) is ensured by the Gaussian nature of $g(t,x)$ implying a global Lipschitz condition of the score function $\Df_x \log g(t,x)$.
 Lastly, by construction of $\P^h$ and $\P^g$, the laws of $X^g$ and $X^h$ are absolutely continuous in path space, hence satisfying the third condition imposed on $g$. 
 In \cite{Piepersethmacher2025Simulation} it was shown that this absolute continuity persists if the observation kernel degenerates to a Dirac measure in $Lx$, i.e. if $Y \mid X_T \sim \delta_{Lx}(\cdot)$. In other words, the guided distribution remains a valid importance sampling distribution for $X^h$ even under highly informative observations. 

\vspace{\baselineskip}

In this work we apply this basic idea of constructing the conditioned process $X^h$ as well as the guided process $X^g$ by a change of measure to derive solutions to both the filtering as well as the smoothing problem. 
For the filtering problem, we propose the law of $X^g$ as a natural candidate for the proposal distribution in a 
\textit{particle filter} (or \textit{sequential Monte Carlo}) approach.

On the other hand, to solve the smoothing problem we extend the changes of measure $\P^h$ and $\P^g$ to account for the complete set of observations $(Y_1,\ldots Y_n)$. 
This enables drawing weighted samples of $X^h$ which can then be passed onto \text{Markov chain Monte Carlo} or \textit{importance sampling} schemes to sample from the smoothing distribution of $X$.

\subsection{Outline}
In Section \ref{sec: changes_measure} we derive the two changes of measure that define the smoothing and guided distribution of $X$. In particular, we present two ways of efficiently computing the score $\Df_x \log g$ needed to sample from the guiding distribution and evaluate the corresponding weights.
The implementation of the guided distribution into a particle filtering scheme is discussed in Section \ref{sec: guided_particle_filter}. Section \ref{sec: guided_MCMC} derives a Markov Chain Monte Carlo algorithm that targets the smoothing distribution of $X$ and posterior of model parameters $\theta$.
Performance of the proposed methodology is illustrated in Section \ref{sec: amari} through an example based on the stochastic Amari equation.

\subsection{Frequently used notation}
We denote by $H$ a Hilbert space with inner product $\langle x,y \rangle$ and norm $|x| = \sqrt{ \langle x, x \rangle }$. 
The process $X$ denotes the mild solution to Equation \eqref{eq: dXt}, defined on the stochastic basis $(\Omega, \calF, (\calF_t)_{t \geq 0},\P)$, with Markov transition kernel $\mu = (\mu_{s,t})_{s \leq t}$. If $F$ is independent of $X$ and instead $F(t,X_t) = a_t$ for some predictable, $H$-valued process $(a_t)_t$, we write $Z$ in place of $X$ for the Ornstein-Uhlenbeck process that satisfies \eqref{eq: dXt} and $\nu = (\nu_{s,t})_{s \leq t}$ for its transition kernel.
\vspace{\baselineskip}

The kernel $k_i(x,\cdot)$ defines a conditional density on $\R^{m_i}$ for any $x \in H$. For most of the paper, we assume $k_i(x, \cdot) = k(x,\cdot)$ to be a Gaussian density on $\R^m$ with mean $Lx$ and covariance matrix $\Sigma$, where $L:H \to \R^m$ is a bounded linear observation operator. 
By $Y_i \in \R^{m_i}$ we denote the observations of $X_{\ti}$ with density $k_i(X_{\ti}, \cdot)$ at observation time $t_i$, with $Y = (Y_1,\ldots Y_n)$ defined as the joint vector of all observations. We refer to individual components of one observation vector $Y_i \in \R^{m_i}$ as measurements.
For realisations of $Y_i$, we write $y_i \in \R^{m_i}$. We shall also write $y_i^+ = (y_i,\ldots,y_n)$ for the vector of all non-past observations at time $t_i$. Moreover, $m_i^+$ denotes the dimension of $y_i^+$, i.e. $m_i^+ = \sum_{j=i}^n m_j$. For a function $f$, $f^+(t) := \lim_{\substack{s \downarrow t}} f(s)$ is defined as the right-sided limit of $f$ at $t$ whenever existent.

\vspace{\baselineskip}
By $C([0,T];H)$ we denote the space of continuous, $H$-valued functions endowed with the supremum norm and its Borel algebra. 
For the laws of $X$ and $Z$ on the Borel algebra of $C([0,T];H)$ we write $\calL(X)$ and $\calL(Z)$ respectively. 
With slight abuse of notation, we shall also write $X_{[\tim, \ti]}$ for the process $(X_t)_{t \in [\tim, \ti]}$ and $\calL(X_{[\tim,\ti]})$ for its law on $C([\tim,\ti];H)$.

The realised observations $y$ together with the transition kernels $\mu$ and $\nu$ define the functions $h$ and $g$. 
These define the changes of measure $\P^h$ and $\P^g$ under which the law of $X$ is, respectively, the smoothing and guided distribution given the observations $y$. We denote these by $\calL^h(X)$ and $\calL^g(X)$.

\section{Changes of measure}
\label{sec: changes_measure}

The aim of this section is to derive the two changes of measure $\P^h$ and $\P^g$ that respectively define the intractable smoothing distribution $\calL^h(X)$ and the tractable guided distribution $\calL^g(X)$ of $X$. In particular, the construction will show that 
\begin{itemize}
    \item[(i)] $\calL^h(X)$ and $\calL^g(X)$ are absolutely continuous with Radon-Nikodym derivative
\begin{align*}
    \Phi(X) := \dfrac{\df \calL^h(X)}{\df \calL^g(X)}(X) = \frac{C^g}{C^h} \exp \left(  \int_0^T \langle F(s,X_s), G(s,X_s) \rangle \df s \right),
\end{align*}
where $C^g$ and $C^h$ are normalising constants and $G(t,x) = \Df_x \log g(t,x )$ is a tractable function,
    \item[(ii)] the distribution $\calL^g(X)$ can be sampled by numerically solving 
    \begin{align*}
        \df X^g_t &= \left[ A X^g_t + F(t, X^g_t) + Q G (t,X^g_t) \right] \df t + Q^{\frac12} \df W^g_t.
    \end{align*}
\end{itemize}
Moreover, we derive two ways of efficiently computing $G$ needed to sample from $X^g$ and evaluate the corresponding weights. 
The proofs of the results in this section can be found in Appendix \ref{app: A}.

\subsection{The smoothing distribution}
\label{sec: smoothing_measure}
We start by constructing the smoothing distribution of $X$ given $Y_1 = y_1, \ldots, Y_n = y_n$. We do so by defining a measure $\P^h$ on $\calF_T$ such that 
\begin{align}
\label{eq: smoothingEh}
    \E^h \left[ \varphi(X) \right] &=  \E[ \varphi(X) \mid Y_1 = y_1, \ldots, Y_n = y_n] 
\end{align}
for all bounded and measurable functionals $\varphi: C([0,T],H) \to \R$.
Note that for $n = 1$, this includes the locally optimal proposal distribution in the guided particle filter as introduced in \cref{sec: approach}.

To this end, define the function $h: [0,T] \times H \to \R_+$ by
\begin{align}
\label{eq: def_h}
h(t,x) := \int k(x_{\ti}, y_i) \left( \prod_{j=i}^{n-1} k(x_{\tjp},y_{j+1}) \, \mu_{\tj,\tjp}(x_{\tj}, \df x_{\tjp}) \right) \mu_{t, \ti}(x, \df x_{\ti}),  \quad t \in (\tim,\ti], 
\end{align}
with $h(0,x) := h^+(0,x)$.
For the sake of notational clarity, we drop the dependence of $h$ on the observations $(y_1,\ldots,y_n)$. 
However, it is worth pointing out that on any time interval $(\tim,\ti]$, $h(t,x)$ has a natural interpretation as the likelihood function of $X_t = x$ given all non-past observations $y_i^+ = (y_i,\ldots,y_n)$, i.e. in Bayesian notation
\begin{align*}
    h(t,x_t) = p(y_i^+ \mid x_t), \quad t \in (\tim,\ti].
\end{align*}

A direct consequence of the definition of $h$ is that it satisfies the following recursive relation known as the \textit{backwards information filter}.
\begin{proposition}
\label{prop: h_properties}
The function $h(t,x)$ defined in \eqref{eq: def_h} satisfies 
\begin{align*}
h(t,x) = \int h(t_i,x_{\ti}) \, \mu_{t,\ti}(x, \df x_{\ti}), \quad t \in (\tim, \ti].
\end{align*}
In particular, $h$ is continuous on any interval $(\tim, \ti)$ and left continuous at any $\ti, i = 1,\dots,n$.
Moreover, at the observation times $\ti$, it holds that
\begin{align*}
h(\ti, x) = k(x,y_{j}) \int h(\tip,x_{\tip}) \, \mu_{\ti,\tip}(x, \df x_{\ti}) = k(x,y_{j}) \, h^+(\ti,x).
\end{align*}
\end{proposition}

With the choice of $h$ established, define now the process $(E^h_t)_{t \in [0,T]}$ by
\begin{align}
\label{eq: def_Eht}
E^h_t := \frac{1}{C^{h}} \left( \prod_{j=1}^{i-1} k(X_{t_j},y_j) \right) h(t,X_t), \quad t \in (\tim,\ti],
\end{align}
with $E^{h}_0 := h(0,X_0)/C^{h}$ and $C^h := \E[h(0,X_0)]$ acting as a normalising constant.

\begin{remark}
\label{rem: prodkh_likelihood}
At any time $t \in (\tim,\ti]$, the product $\prod_{j=1}^{i-1} k(x_{t_j},y_j)$ appearing on the right-hand side of \eqref{eq: def_Eht} is the likelihood of previous states of $X$ at observation times given past observations, whereas $h(t,x_t)$ is the likelihood of $X_t = x_t$ given all non-past observations. 
Hence, in total, the term 
\begin{align*}
    \left( \prod_{j=1}^{i-1} k(x_{\tj},y_j) \right) h(t,x_t,y), \quad t \in (\tim,\ti],
\end{align*} 
is the likelihood of $(X_{t_1} = x_{t_1}, \ldots , X_{t_{i-1}} = x_{t_{i-1}}, X_t = x_t)$ given all observations $y_j$, $j = 1,\dots,n$.
\end{remark}

\begin{proposition}
\label{prop: Eht_martingale}
The process $(E^h_t)_{t \in [0,T]}$ defined in \eqref{eq: def_Eht} is a non-negative $\P$-martingale with respect to the filtration $(\calF_t)_{t \in [0,T]}$.
\end{proposition}

Given its martingale property and noting that $E^h_0 = 1$, the process $E^h$ defines a measure $\P^h$ on $\calF_T$ by
\begin{align}
\label{eq: def_dPh}
\df \P^h_{\mid \calF_T} &= E^h_T \df \P_{\mid \calF_T}.
\end{align}
The following theorem establishes that this is indeed the change of measure under which the law of $X$ is the smoothing distribution. 

\begin{theorem}
\label{thm: dPh} \,
\begin{itemize} 
    \item[(i)] The measure $\P^h$ satisfies \eqref{eq: smoothingEh}.
    \item[(ii)] In addition, if $h$ is Fréchet differentiable in $x$ such that $\Df_x h \in C_m((\tim,\ti);H)$ for all $i = 1 \ldots n$, then the process $X$ under $\P^h$ is a mild solution to
    \begin{align}
    \label{eq: dXh2}
    \begin{cases}
        \df X^h_t &= \left[ A X^h_t + F(t, X^h_t) + Q \Df_x \log h (t,X^h_t) \right] \df t + Q^{\frac12} \df W^h_t, \\
        X^h_0 &\sim \mu_0^h,
    \end{cases}
    \end{align}
    where $W^h$ is a $\P^h$-cylindrical Wiener process and $\mu_0^h(B) := \P^h(X_0 \in B)$.
\end{itemize}
\end{theorem}

\begin{remark}
    The assumption in Theorem \ref{thm: dPh} (ii) is in general hard to verify. However, to sample from the law of $X$ under $
    \P^h$, the differential form of $X^h$ in \cref{eq: dXh2} is not necessary. 
    Instead, weighted samples of the smoothing distribution are obtained by sampling $X$ under the substitute measure $\P^g$ defined in the next subsection.
\end{remark}

\subsection{The guided distribution}
\label{sec: guided_measure}

To construct the guided distribution, we follow similar steps as in the previous section, replacing $\mu$ in the definition of $h$ with a tractable transition kernel $\nu$ of an auxiliary process $Z$.

For this, let $Z$ be the Ornstein-Uhlenbeck process, defined as the unique mild solution to
\begin{align}
\label{eq: dZt}
\begin{cases}   
\df Z_t &= \left[ A Z_t + a_t \right] \df t+ Q^{\frac12} \df W_t, \quad  t \geq 0, \\
Z_0 &\sim \nu_0.
\end{cases}
\end{align}
The $H$-valued process $(a_t)_{t \geq 0}$ is assumed to be predictable with integrable trajectories and $\nu_0$ denotes a Gaussian measure on $H$. 
Let $(\nu_{s,t})_{s \leq t}$ be the transition kernel of $Z$ defined by $\nu_{s,t}(x, B) := \P(Z_t \in B \mid Z_s =x)$ for any $0 \leq s \leq t$, $x \in H$ and Borel set $B \subset H$. It is well-known that $\nu_{s,t}(x, \cdot)$ is a Gaussian measure on $H$ with mean $S_{t-s} x + \int_s^t S_{t-u} a_u \df u$ and covariance operator
\begin{align}
\label{eq: def_Qt}
    Q_{t-s} := \int_0^{t-s} S_{u} Q S_{u}^* \df u.
\end{align} 

Following the definition of $h$ in Equation \eqref{eq: def_h}, we define the mapping $g$ by 
\begin{align}
\label{eq: def_g}
g(t,x) := \int k(x_{\ti}, y_i) \left( \prod_{j=i}^{n-1} k(x_{\tjp},y_{j+1}) \, \nu_{\tj,\tjp}(x_{\tj}, \df x_{\tjp}) \right) \nu_{t, \ti}(x, \df x_{\ti}),  \quad t \in (\tim,\ti],
\end{align}
with the convention that $g(0,x) := g^+(0,x)$.

The function $g(t,x)$ on $(\tim, \ti]$ has a natural interpretation as the likelihood of $Z_t = x$ given the vector $(y_i,\dots,y_n)$ of non-past observations $y_i \sim k(Z_{\ti}, \cdot)$ under the simplified dynamics of \cref{eq: dZt}.
Moreover, just like the $h$-function, $g$ satisfies the recursion of the backwards information filter 
\begin{align}
\begin{split}
\label{eq: g_BIF2}
    g(t,x) &= \int g(t_i,x_{\ti}) \, \nu_{t,\ti}(x, \df x_{\ti}), \quad t \in (\tim, \ti], \\
    g(\ti, x) &= k(x,y_{i}) \, g^+(\ti,x), \quad i \in \{1,\dots,n\},
\end{split}
\end{align}
with $g^+(\ti,x) = \int g(\tip,x_{\tip}) \, \nu_{\ti,\tip}(x, \df x_{\ti})$.
    
Contrary to the generally intractable $h$, the Gaussian nature of $\nu_{s,t}(x, \cdot)$ and $k(x, \cdot)$ enables us to compute the function $g$ using the recursion in \eqref{eq: g_BIF2} as shown in the next theorem.
\begin{theorem}
\label{thm: g_LtRtalphat}
On any interval $(\tim, \ti]$, $i =1, \dots,n,$ let $L_t \in L(H, \R^{m^+_i})$ be given by
\begin{align}
\label{eq: def_Lt}
L_t := \begin{bmatrix}  L S_{\ti-t} \\  L S_{\tip -t} \\ \vdots \\ L S_{\tn-t} \end{bmatrix} 
\end{align}
and let $R_t \in L(\R^{m^+_i})$ and $\alpha_t \in \R^{m^+_i}$ be defined by the backwards recursions
\begin{align}
\label{eq: def_Rt_alphat}
\begin{split}
    R_t &:= \begin{cases} \Sigma + L Q_{\tn-t} L^*,  & i =  n,\\
        \begin{bmatrix} 
\Sigma & 0 \\
0 & R^+_{\ti} 
\end{bmatrix} + L_{\ti} Q_{\ti-t} L^*_{\ti},  &\text{else},
    \end{cases}   \\
\alpha_t &:= \begin{cases}
    L \left( \int_t^{\tn} S_{\tn-s} a_s \, \df s \right), & i = n, \\
\begin{bmatrix}  0 \\ \alpha^+_{\ti} \end{bmatrix} + L_{\ti} \left( \int_t^{\ti} S_{\ti -s} a_s \, \df s\right), & \text{else}.
\end{cases}
\end{split}
\end{align}

Then, the function $g(t,x)$ is given by
\begin{align}
\label{eq: g_BIF_computed}
g(t,x) = f(y_i^+ ; L_tx + \alpha_t, R_t), \quad t \in (\tim, \ti].
\end{align}
\end{theorem}

\vspace{\baselineskip}

Following Theorem \ref{thm: g_LtRtalphat}, the function $g$ can be expressed on any interval $(\tim,\ti]$ as the Gaussian density in Equation \eqref{eq: g_BIF_computed}. This lets us define a tractable change of measure $\P^g$ on $\calF_T$ as follows. 
Denote by $G$ the score function of the likelihood $x \mapsto g(t,x)$, i.e.
\begin{align}
\label{eq: def_G}
G(t,x) := \Df_x \log g(t,x) = L_t^* R_t^{-1} (y_i^+ - L_tx - \alpha_t), \quad t \in (\tim,\ti], x \in H.
\end{align}

Moreover, define the process $(E^g_t)_{t \in [0,T]}$ by setting
\begin{align}
\label{eq: def_Egt} 
E^g_t := \frac{1}{C^{g}} \left( \prod_{j=1}^{i-1} k(X_{t_j},y_j) \right) g(t,X_t) \exp \left(- \int_0^t \langle F(s,X_s), G(s,X_s) \rangle \, \df s \right), \quad t \in (\tim,\ti],
\end{align}
with $E^g_0 := g(0,X_0)/C^{g}$ and $C^g := \E[g(0,X_0)]$ acting as a normalizing constant.

The theorem presented below shows that $E^g$ is a martingale that defines a measure $\P^g$ on $\calF_T$ under which the law of $X$ is exactly the guided distribution as introduced in \cref{sec: approach}.
\begin{theorem}
\label{thm: dPg} \,
\begin{itemize}
    \item[(i)] The process $E^g$ defined in \eqref{eq: def_Egt} is a non-negative $\P$-martingale with respect to the filtration $(\calF_t)_{t \in [0,T]}$.
    \item[(ii)] Under the measure $\P^g$ defined on $\calF_T$ by $\df \P^g_{\mid \calF_T} = E^g_T \df \P_{\mid \calF_T}$, the process $X$ is the unique mild solution to 
    \begin{align}
    \begin{split}
    \begin{cases}
        \label{eq: dXg}
        \df X^g_t &= \left[ A X^g_t + F(t, X^g_t) + Q G (t,X^g_t) \right] \df t + Q^{\frac12} \df W^g_t,
        \quad t \in [0,T], \\
        X^g_0 &\sim \mu_0^g,
    \end{cases}
    \end{split}
    \end{align}
    where $W^{g}$ is a cylindrical Wiener process under $\P^g$. 
\end{itemize}
\end{theorem}

From the construction of $\P^h$ and $\P^g$ it follows that the smoothing and guided distribution of $X$ are absolutely continuous with Radon-Nikodym derivative given by 
\begin{align}
\label{eq: dPh_dPg}
\Phi(X) = \frac{C^g}{C^h} \Psi(X) 
\end{align}
where $\Psi(X)$ is defined as 
\begin{align}
\label{eq: def_Psi}
	\Psi(X) := \exp \left(  \int_0^T \langle F(s,X_s), G(s,X_s) \rangle \df s \right).
\end{align}
Hence, weighted samples of the smoothing distribution $\calL^h(X)$ can be obtained by drawing samples from $\calL^g(X)$ and evaluating the weights $\Phi(X)$. As a consequence of Theorem \ref{thm: dPg}, the sampling step corresponds to numerically solving Equation \eqref{eq: dXg}.

\begin{remark}
Drawing samples from $\calL^g(X)$ as well as evaluating the weights $\Phi(X)$ requires computation of the function $G$ defined in Equation \eqref{eq: def_G}.
In practical applications, this is carried out on a gridded domain in both space and time. 

In that case, the infinite-dimensional operators $A$, $S_t$ and $Q$ are typically approximated by $M{\times}M$-dimensional matrices on the given spatial - or spectral - grid. Likewise, the observation operator $L$ is represented by a $m{\times}M$-dimensional matrix. 
Evaluating $G$ based on Theorem \ref{thm: g_LtRtalphat} on a temporal grid of $(\tim,\ti]$ with $N$ grid points then requires the following operations:
\begin{itemize}
    \item[1.] The operators $L_t$ are given in closed form and are a simple concatenation of linear compositions of $L$ and the semigroup $S$. For each observation time $\ti$, the matrix multiplication $L_t = L_{\ti} S_{t_i -t}$ needs to be carried out for $N$ grid points of $(\tim,\ti]$. This results in $O(N m_i^+ M^2)$ operations.
    \item[2.] To compute $R_t$, the matrix approximations of $Q_{\ti-t} = \int_0^{\ti-t} S_s Q S_s^* \df s$ need to evaluated. If this integral cannot be solved directly, it needs to be numerically integrated with computational complexity of $O(N M^3).$ Additionally, the matrix multiplications $L_{\ti} Q_{\ti-t} L^*_{\ti}$ will need to be carried out, leading to an additional total cost of $O(N m_i^+ M^2)$.
    Moreover, evaluation of $G$ will require the matrix inversion of $R_t$. Due to the block matrix structure of $R_t$, this is of cost $O((n-i+1)m^3) = O(m_i^+ m^2)$ for each inversion. 
    \item[3.] Lastly, to compute $\alpha_t$, we need to numerically backwards integrate the term $\int_t^{\ti} S_{\ti -s} a_s \, \df s$. This is of complexity $O(N M^2)$ after discretisation, on top of the cost $O(N m^+_i M)$ of evaluating $L_{\ti} \int_t^{\ti} S_{\ti -s} a_s \, \df s$. 
\end{itemize}

Hence, in total, the cost of the operations required to compute $G$ using Theorem \ref{thm: g_LtRtalphat} on a spatiotemporal grid with $(N,M)$ grid points is of complexity
\begin{align*}
    O( N \max\{M^3, m^+_i M^2, m_i^+ m^2\}).
\end{align*}
In principle, the spatial grid size $M$ chosen by the user is fixed but can be arbitrarily large.
However, with growing number of observations $n$, the dimension $m^+_i = (n-i+1)m$ quickly grows larger than $M$, leading to a computational costs of complexity $O(N m^+_i M^2)$.
In that case, it is beneficial to parametrise $g$ in a different manner as we will explore in the next section.
\end{remark}

\subsection{Computationally efficient guiding}
\label{sec: comp_efficient_guiding}

We derive an alternative to Theorem \ref{thm: g_LtRtalphat} of computing $G$ to sample from the guided process $X^g$. 
Define the functions $(V_t)_{t \in [0,T]}$ and $(U_t)_{t \in [0,T]}$ by
\begin{align}
\begin{split}
\label{eq: def_Ut_Vt}
    U_t &:= L^*_t R^{-1}_t L_t \\
    V_t &:= L^*_t R^{-1}_t (y_i^+ - \alpha_t), \quad t \in (\tim,\ti]
\end{split}   
\end{align}
for each $i = 1,\ldots,n$ with $U_0 := U^+_0$ and $V_0 := V^+_0$.
Following Theorem \ref{thm: g_LtRtalphat}, it holds that 
\begin{align}
\label{eq: G}
    G(t,x) = V_t - U_t x.
\end{align}
As the upcoming theorem shows, $U_t$ and $V_t$ can be obtained by solving a set of infinite-dimensional backwards differential equations.

\begin{theorem}
    \label{thm: G_UtVt}
    On each interval\footnote{If $i=n$, $U$ needs to be initialised with $U_{\tn} = L^* \Sigma L $ and $V$ with $V_{\tn} = L^* \Sigma^{-1} y_n$.} $(\tim, \ti]$, $i = 1,\ldots,n,$  $U$ is the unique mild solution to the backwards Ricatti equation 
    \begin{align}
    \begin{split}
    \label{eq: dU_t}
    \begin{cases}
    \df U_t &=  \left[ -A^* U_t - U_t A + U_t Q U_t \right] \df t\\
    U_{\ti} &= L^* \Sigma^{-1} L + U^+_{\ti},
    \end{cases}
    \end{split}
    \end{align}
    whereas $V$ is the unique mild solution to the backwards equation
    \begin{align}
    \begin{split}
    \label{eq: dV_t}
    \begin{cases}
        \df V_t &= \left[ -A^* V_t + U_t Q V_t + U_t a_t \right] \df t,\\
        V_{\ti} &= L^* \Sigma^{-1} y_i + V^+_{\ti}. 
    \end{cases}
    \end{split}
    \end{align}
\end{theorem}

\vspace{\baselineskip}

\begin{remark}
    Without going into too much detail, it is worth pointing out that there exists extensive literature on the approximation of the infinite-dimensional Ricatti equation. We refer to \cite{Gibson1983Linear}, \cite{Burns2015Solutions}, \cite{Cheung2025Approximation} and references within for an introduction into the relevant literature. 

    For numerical purposes, the infinite-dimensional operators $A$ and $Q$ are typically approximated by $M{\times}M$-dimensional matrices on a spatial - or spectral - grid.
    Given such approximations of $A$ and $Q$, the backwards Ricatti equation \eqref{eq: dU_t} can be numerically solved using standard solvers for ODEs with total costs of $O(N M^3)$ on a temporal grid of $N$ points. 
    In a similar fashion, the backwards evolution equation \eqref{eq: dV_t} can be solved using standard ODE solvers with total costs of $O(N M^3)$.
    Note that, contrary to the computation of $G$ based on Theorem \ref{thm: g_LtRtalphat}, this cost is independent of the number of observations. Moreover, in many cases, structural properties of the matrix approximations of $A$ and $Q$, such as sparsity, can be taken advantage of to further reduce the computational costs. 
\end{remark}

\begin{remark}
    In certain cases, the backwards Ricatti equation \eqref{eq: dU_t} can be solved in closed form. 
    This includes the highly relevant case that $A$, $Q$ and $L$ are diagonalisable, i.e. if there exists an orthonormal basis $(e_j)_{j=1}^{\infty}$ with 
    \begin{align*}
        A e_j &= - a_j e_j, \\
        Q e_j &= q_j e_j 
    \end{align*}
    such that $0 < a_j \to \infty$ and $0 < q_j < \sup_j q_j <\infty$ and $\im(L^*) \subset \spn\{e_j : j = 1, \ldots m\}$.

    Equation \eqref{eq: dU_t} then decouples into a sequence $(u^j)_{j=1}^m$ of scalar-valued Ricatti differential equations
    \begin{align*}
        d u^j_t = 2 a_j u^j_t + q_j (u^j_t)^2 \df t, \quad t \in (\tim, \ti), 
    \end{align*}
    which are solved in closed form for all $j = 1 \ldots m$ by (cf. \cite{Zaitsev2002Handbook}, Chapter 1.1)
    \begin{align*}
        u^j_t 
        &= \frac{\exp(- 2 a_j (t_i - t))}{(u_{\ti}^j)^{-1} + \dfrac{q_j}{2 a_j} \left[ 1 - \exp( - 2 a_j (\ti-t)) \right]}, \quad \quad t \in (\tim, \ti).
    \end{align*}
\end{remark}

\vspace{\baselineskip}

When computing $G$ based on Theorem \ref{thm: G_UtVt}, the following proposition is useful to derive the full expression of $g$. This is required, for example, when inferring not only unobserved states of $X$ but also unknown model parameters $\theta$.   
\begin{proposition}
\label{prop: dct}
In addition to $U$ and $V$ as defined in Theorem \ref{thm: G_UtVt}, let $c$ be defined by
\begin{align*}
c_t = - \dfrac{1}{2} \left[ \log((2 \pi)^{m^+_i}) + \log(\det(R_t))  + \langle y^+_i - \alpha_t, R^{-1}_t (y^+_i - \alpha_t)\rangle \right], \quad t \in (\tim, \ti].
\end{align*}
Then $\log g(t,x) = c_t + \langle V_t,x \rangle - \dfrac{1}{2} \langle x, U_t x \rangle.$
Moreover, on any interval $(\tim, \ti]$, $i = 1, \ldots, n$, $c_t$ solves the backwards differential equation
\begin{align}
\begin{split}
\begin{cases}
\label{eq: dct}
\df c_t &= \left[ \dfrac{1}{2} \tr \left[ U_t Q  \right] - \left \langle   a_t ,  V_t \right \rangle -\dfrac{1}{2} \left \langle V_t, Q V_t \right \rangle\right] \df t,  \\
c_{\ti} &= \log f(y_i; 0, \Sigma) - c^+_{\ti}.
\end{cases}
\end{split}
\end{align}
\end{proposition}

\section{Guided particle filtering}
\label{sec: guided_particle_filter}

In this section, we show how a special case of the guided distribution can be used as proposal distribution in a particle filtering scheme. 
For the convenience of the reader, we give a brief outline of the particle filter (PF).
We refer to \cite{Doucet2001Sequential} or \cite{Chopin2020Introduction} for book-length treatments of this field. 

The particle filter uses a set of particles $\{X^{(j)}_{\ti} : j = 1,\ldots, J\}$ with associated weights $\{w^{(j)}_{i} : j = 1,\ldots, J\}$ to approximate the filtering distribution \eqref{eq: def_filtering} at any observation time $\ti$ by a Monte Carlo approximation
\begin{align*}
    \E[ \varphi(X_{\ti}) \mid Y_1 = y_1, \ldots, Y_n = y_i] &\approx \sum_j w^{(j)}_{\ti} \varphi(X^{(j)}_{\ti})
\end{align*}
for any bounded and measurable function $\varphi: H \to \R$.

The particles and their weights are sequentially updated with incoming observations as per the following recursion. 
Given an estimate $\{ (X^{(j)}_{\tim},w^{(j)}_{i-1})  : j = 1,\ldots, J\}$ of the filtering distribution at time $\tim$ and given a Markov kernel $\Q(\cdot \mid x)$ from $H$ to $C([\tim, \ti];H)$, the particles evolve forward to the next observation time $\ti$ according to
\begin{align*}
    X^{(j)}_{[\tim,\ti]} \sim \Q( \cdot \mid  X^{(j)}_{\tim}), \quad j = 1 , \ldots J.
\end{align*}
The weights are then updated (and normalised) via
\begin{align*}
   w_{i}^{(j)} \propto w_{i-1}^{(j)} \, k(X^{(j)}_{\ti}, y_{i}) \frac{\df \calL(X_{[\tim,\ti]} \mid X^{(j)}_{\tim})}{\df \Q( \cdot \mid X^{(j)}_{\tim})} \left(X^{(j)}_{[\tim,\ti]}\right), \quad \text{s.t. } \sum_{j=1}^J w_{i}^{(j)} = 1.
\end{align*}
Here, $\calL(X_{[\tim,\ti]} \mid X^{(j)}_{\tim})$ denotes the law of $X_{[\tim,\ti]}$ satisfying \eqref{eq: dXt} with initial state $X_{\tim} = X^{(j)}_{\tim}$.

Typically, to prevent \textit{weight degeneracy}, i.e. one or few weights dominating the others, a resampling step of the particles based on the normalized weights is introduced. This can be done every fixed number of steps or in an adaptive number based on maintaining the \textit{effective sample size} $ESS_i := 1/\left( \sum_{j=1}^J(w_{i}^{(j)})^2 \right)$ above a predefined threshold.

\vspace{\baselineskip}
\begin{algorithm}[b]
    \LinesNotNumbered  
    \label{alg: guided_particle_filter}
    \caption{Guided Particle Filter}
    \KwIn{Observations $\{y_i\}_{i=1}^n$, number of particles $J$, threshold $J_0$}
    \KwOut{Particles and weights $\Bigl\{ \left(X_{\ti}^{(j)}, w_i^{(j)} \right) : j = 1 \ldots J\Bigr\}$ for each $i = 1, \dots, n$}
    Initialise $\{X_0^{(j)}\}_{j=1}^J \sim \mu_0$ and set $w_0^{(j)} = \frac{1}{J}$\;
    \For{$i = 1$ \KwTo $n$}{
        \textbf{1. Evolve particles:} \\
        \For{$j = 1$ \KwTo $J$}{
            (i) Sample guided proposal: \[ X^{(j)}_{[\tim,\ti]} \sim \calL^{\gi}( X_{[\tim,\ti]} \mid X^{(j)}_{\tim});\] \\
            (ii) Compute unnormalized weight:
            \[
            \tilde{w}_{i}^{(j)} = w_{i-1}^{(j)} \, g_i(\tim, X^{(j)}_{\tim}) \exp \left( \int_{\tim}^{\ti} \langle F(s, X^{(j)}_s), G_i(s, X^{(j)}_s)\rangle \df s\right);
            \]
        }
        \textbf{2. Normalize Weights:} 
        \[w_{i}^{(j)} = \frac{\tilde{w}_{i}^{(j)}}{\sum_{k=1}^J \tilde{w}_{i}^{(k)}};\]
        Compute effective sample size:
        $\text{ESS} = 1/\sum_{j=1}^J \left(w_{i}^{(j)}\right)^2$\;
        \textbf{3. Resample:}
        \If{\textnormal{ESS}   $< J_0$}{
            Resample $\{X_{\ti}^{(j)}\}_{j=1}^J$ according to weights $w_{i}^{(j)}$\;
            Set $w_{i}^{(j)} = \frac{1}{J}$ for all $j$\;
        }
    }
\end{algorithm}
The proposal kernel $\Q$ is a parameter of the particle filter that is to be chosen by the user. 
If $\Q( \cdot \mid X^{(j)}_{\tim}) = \calL(X_{[\tim,\ti]} \mid X^{(j)}_{\tim})$, the algorithm is known as the \textit{bootstrap particle filter (BPF)}. Since proposals in the BPF are `uninformed' by the observations, it is known to typically struggle with highly degenerate weights. 

On the other hand, the locally optimal proposal kernel is given by the conditional distribution 
\begin{align*}
    \Q( \cdot \mid X^{(j)}_{\tim}) = \P(X_{[\tim,\ti]} \in \cdot \mid X_{\tim} = X^{(j)}_{\tim}, Y_{i} = y_{i}).
\end{align*}
The proposal $\Q$ minimises the variance of the weights $w_{i}^{(j)}$, see for example \cite{Chopin2020Introduction}, Theorem 10.1 for details.
However, as noted in the introduction and previous section, this proposal kernel is typically intractable. 

\vspace{\baselineskip}

We propose instead to steer the particles on the interval $[\tim,\ti]$ into regions of high probability given $Y_{i} = y_{i}$ by the guided distribution introduced in Section \ref{sec: guided_measure}.
For this, define 
\begin{align*}
    \gi(t,x) := \int k(x_{\ti}, y_{i})\nu_{t, \ti}(x, \df x_{\ti}), \quad t \in [\tim, \ti],
\end{align*}
and let $\P^{\gi}$ be the change of measure on $\calF_{\ti}$ as defined in Theorem \ref{thm: dPg}.
We adopt as proposal kernel the guided distribution
\begin{align}
\label{eq: guided_filtering_proposal}
    \Q( \cdot \mid X^{(j)}_{\tim}) = \calL^{\gi}( X_{[\tim,\ti]} \mid X^{(j)}_{\tim}),
\end{align}
i.e. the law of $X_{[\tim,\ti]}$ under $\P^{\gi}$ initiated at $X_{\tim} = X^{(j)}_{\tim}$. 
Particles are then evolved forward in time by simulating the guided process
\begin{align}
\label{eq: dXg2}
\df X^{\gi}_t &= \left[ A X^{\gi}_t + F(t, X^{\gi}_t) + Q G_i (t,X^{\gi}_t) \right] \df t + Q^{\frac12} \df W^{\gi}_t,
        \quad t \in [\tim,\ti], 
\end{align}
with respective initial conditions $X^{(j)}_{\tim}, j = 1, \ldots J$ and $G_i(t,x) = \Df_x \log g_i(t,x)$. Moreover, Theorem \ref{thm: dPg} implies that the update step in the corresponding weights is given by
\begin{align}
\begin{split}
w_{i}^{(j)} &\propto w_{i-1}^{(j)} \, k(X^{(j)}_{\ti}, y_{i}) \frac{\df \calL(X_{[\tim,\ti]} \mid X^{(j)}_{\tim})}{\df \calL^{\gi}(X_{[\tim,\ti]} \mid X^{(j)}_{\tim})} \left(X^{(j)}_{[\tim,\ti]}\right) \\ 
&= w_{i-1}^{(j)} \, k(X^{(j)}_{\ti}, y_{i}) \, \dfrac{g_i(\tim, X^{(j)}_{\tim})}{g_i(\ti, X^{(j)}_{\ti})} \exp \left( \int_{\tim}^{\ti} \langle F(s, X^{(j)}_s), G_i(s, X^{(j)}_s)\rangle \df s\right) \\
&= w_{i-1}^{(j)} \, g_i(\tim, X^{(j)}_{\tim}) \exp \left( \int_{\tim}^{\ti} \langle F(s, X^{(j)}_s), G_i(s, X^{(j)}_s)\rangle \df s\right).
\end{split}
\end{align}

The resulting \textit{guided particle filter} is summarised in Algorithm \ref{alg: guided_particle_filter}.

\begin{remark}
\label{rem: filtering_G}
The measure $\calL^{\gi}( X_{[\tim,\ti]})$ is the `one-step-ahead' guided distribution on $[\tim, \ti]$ corresponding to the next observation $y_i$. Hence, the functions $g_i$ and $G_i$ only need to be computed as functions of $y_i$ and the parametrisation of Theorem \ref{thm: g_LtRtalphat} remains a computationally valid choice. 
Specifically, on any $[\tim, \ti]$, $g_i(t,x)$ is the Gaussian likelihood $g_i(t,x) = f(y_i ; L_tx + \alpha_t, R_t)$ 
with $L_t = L S_{\ti-t}$, $\alpha_t = L \left( \int_t^{\ti} S_{\ti-s} a_s \, \df s \right)$ and $R_t = \Sigma + L Q_{\ti-t} L^*$, whereas $G_i$ is the score
\begin{align*}
	G_i(t,x) = L_t^* R_t^{-1} \left( y_i - L_t x - \alpha_t\right), \quad t \in [\tim,\ti].
\end{align*}
\end{remark}

\subsection{Tempering and moving steps}
Even though the proposal kernel introduced in Equation \eqref{eq: guided_filtering_proposal} defines a suitable, likelihood-informed importance sampling distribution, the guided particle filter might still face weight degeneracy due to the high dimensionality of the state space. 
Since this is a common issue in particle-based inference, many solutions to overcome it have been proposed and studied in the literature. For the purpose of our numerical experiments in Section \ref{sec: amari}, we introduce two such solutions - \textit{tempering} and \textit{moving} (or \textit{particle rejuvenation}) steps - into the guided particle filter. This follows the suggestions made in \cite{Llopis2018Particle}, where the filtering problem was considered in a similar setup to ours. 
The following brief exposition is based on that work. 

Denoting by $\Pi_i$ the conditioned law $\Pi_i := \calL(X_{[0,t_i]} \mid Y_{1} \dots Y_{i})$, an application of the Bayes theorem and the change of measure $\eqref{eq: guided_filtering_proposal}$ gives the recursion 
\begin{align}
\label{eq: Lambda}
\begin{split}
	\dfrac{\df \Pi_i}{\df ( \Pi_{i-1} \otimes \calL^{g_i})} (X_{[0,t_i]}) &\propto  g_i(\tim, X_{\tim}) \exp \left( \int_{\tim}^{\ti} \langle F(s, X_s), G_i(s, X_s)\rangle \df s\right) \\
	&=: \Lambda_i(X_{[\tim,\ti]}).
\end{split}
\end{align}

Suppose that, at observation time $t_i$, we have an evenly weighted set of samples of the distribution $\Pi_{i-1}$. 
The filtering step at $t_i$ then corresponds to targeting the measure $\Pi_i$ by an importance sampling scheme with proposal distribution $ \Q_i := \Pi_{i-1} \otimes \calL^{g_i}$. Samples from this proposal are drawn by forward propagating the given particles of $\Pi_{i-1}$ following the dynamics in \eqref{eq: dXg2} with IS weights given by $\Lambda_i$ in \eqref{eq: Lambda}.

\vspace{\baselineskip}
Now fix the observation time $t_i$ and drop the subscript $i$ to simplify notation. 
In state spaces of high dimension and/or the situation of highly informative observations, the proposal distribution $\Q$ might not be a sufficiently good approximation for the target $\Pi$.  

Tempering aims to overcome the distance between $\Pi$ and $\Q$ by introducing a sequence of intermediate distributions $\Pi^l$ between $\Pi$ and $\Q$ defined by 
\begin{align}
\label{eq: dPi^l}
	\dfrac{\df \Pi^l}{\df \Q}(X) \propto \Lambda(X)^{\psi_l}
\end{align}
for a set of inverse temperatures $0 = \psi_0 < \ldots < \psi_L = 1$.
Notably, we have absolute continuity of $\Pi^{l+1}$ with respect to $\Pi^l$ with R.N. derivative proportional to $ \Lambda(X)^{\psi_{l+1} - \psi_l}$. 
Instead of targeting $\Pi$ directly by $\Q$, one progressively steps through this sequence of artificial distributions, starting from $\Q$ to target the flattened $\Pi^1$ and ending up at targeting $\Pi$ with proposals drawn from $\Pi^{L-1}$.

Crucially, the parameters $\{ \psi_l \}_l$ can be chosen adaptively as follows. 
Given a set of evenly weighted particles of $\Pi^l$ at current inverse temperature $\psi_l$, the parameter $\psi_{l+1}$ is set by 
\begin{align*}
	\psi_{l+1} := \inf \{ \psi \in (\psi_l, 1] : \ESS_{l,l+1} \leq \alpha J \}
\end{align*}
with the definition of $\inf \O := 1$.
The hyperparameter $\alpha \in (0,1)$ is pre-specified by the user and $\ESS_{l,l+1}$ denotes the essential sample size of the weights $\Lambda(X)^{\psi_{l+1} - \psi_l}$ between $\Lambda^{l+1}$ and $\Lambda^l$. 

After $\Pi^{l+1}$ has been determined, particles are resampled based on the weights $\Lambda(X)^{\psi_{l+1} - \psi_l}$.
Subsequently, to increase particle diversity, each particle is moved by a few steps of an MCMC scheme that targets $\Pi^{l+1}$. 
This can be done by proposing independent samples from $\Q$ and accepting/rejecting the proposals in a Metropolis-Hastings step based on the likelihood in Equation \eqref{eq: dPi^l}.
Alternatively, proposals can be localised based on the preconditioned Crank-Nicolson scheme, see Section \ref{sec: guided_MCMC} for details. 
The complete \textit{tempering and move} procedure is summarised in Algorithm \ref{alg: tempering_steps}. 
It is to be carried out at each observation time $t_i$ in place of Steps $2.$ and $3.$ in Algorithm \ref{alg: guided_particle_filter}.

\vspace{\baselineskip}
\begin{algorithm}[b]
    \LinesNotNumbered  
    \label{alg: tempering_steps}
    \caption{Tempering and move steps at fixed observation time $t_i$}
    \KwIn{Samples $X^{(j)}$ of $\Q = \Pi_{i-1} \otimes \calL^{g_i}$ with weights $\Lambda(X^{(j)})$, hyperparameters $\alpha$ and $N$}
    \KwOut{Evenly weighted samples $X^{(j)}$ of $\Pi = \Pi_i$}
    Initialise $l = 0$ and $\psi_l = 0$\;
    \While{$\psi < 1$}{
        \textbf{1. Compute next $\psi$:} set $l \gets l + 1$ and
        \[ \psi_{l} = \inf \{ \psi \in (\psi_{l-1}, 1] : \ESS_{l-1, l} \leq \alpha J \}; \]
        \textbf{2. Specify $\Pi^l$ and resample:} 
        \begin{itemize}
        	\item[(i)]	Compute weights and normalise 
        	 \[ w_l^{(j)} \propto \Lambda(X^{(j)})^{\psi_{l} - \psi_{l-1}} \quad \text{s.t.} \sum_{j=1}^J w_l^{(j)} = 1 \]

        	\item[(ii)] Resample $\{X^{(j)}\}_{j=1}^J$ according to weights $w_{l}^{(j)}$\;

        \end{itemize}

        \textbf{3. Move particles:} \\
        \For{$j = 1$ \KwTo $J$}{
           \For{$n = 1$ \KwTo $N$}{
            (i) Sample guided proposal: $ X' \sim \Q $\;
            (ii) Compute $M = \min\left(1, \left( \dfrac{\Lambda(X')}{\Lambda(X^{(j)})} \right)^{\psi_l} \right)$\;
            
            (iii) Draw $U \sim \text{Unif}(0,1)$\; 
        	\quad \quad \If{$U < M $}{
      		\quad \quad Set $X^{(j)} = X'$\;}
        }
        }
                
   }
\end{algorithm}

\begin{remark}
For ease of exposition, the MCMC move step in Algorithm \ref{alg: tempering_steps} is presented using independent proposals drawn from $X' \sim \Q$. 
These proposals can be localised in Wiener space by using the preconditioned Crank-Nicolson scheme as is further explained in Section \ref{sec: guided_MCMC}.
For practical purposes, this implies having to store at the observation time $t_i$ for each particle $X^{(j)}$ the sample of the Wiener process $W$ that corresponds to simulating the sample path $X^{(j)}_{[\tim, \ti]}$ following the dynamics in \eqref{eq: dXg2}.
Note that this is still carried out in an online setting since the stored Wiener paths relate to the evolution on $[\tim, \ti]$ only and may hence be discarded after each complete filtering step.
\end{remark}

\section{Guided smoothing via Metropolis-Hastings sampling}
\label{sec: guided_MCMC}

The following section introduces a solution to the smoothing problem via a Metropolis-Hastings (MH) algorithm that targets the smoothing distribution $\calL^h(X)$ based on proposals drawn from the guided distribution $\calL^g(X)$.

\subsection{The case of known $X_0$.}
Assume first that the initial distribution $\mu_0$ is given by a Dirac measure $\delta_{x_0}$ for some $x_0 \in H$. In that case, $X_0 = x_0$ remains unchanged under the changes of measure $\P^h$ and $\P^g$ and the Radon-Nikodym derivative in \eqref{eq: dPh_dPg} reduces to 
\begin{align*}
	\Phi(X) = \frac{g(0,x_0)}{h(0,x_0)} \Psi(X) 
\end{align*}
with $\Psi(X)$ as defined in Equation \eqref{eq: def_Psi}. A basic MH sampler that targets $\calL^h(X)$ can then be constructed by proposing samples from $\calL^g(X)$ and accepting or rejecting the proposals with an acceptance probability based on $\Phi(X)$. Note that $\Phi(X)$ includes the intractable term $h(0,x_0)$ which acts as a proportionality constant and cancels out in an MH acceptance step for a known $x_0$.

Sampling from $\calL^g(X)$ and evaluating $\Psi(X)$ relies on the function $G$. This function needs to be computed \textit{once} as a function of the complete set of observations on the basis of the backwards information filter in Equation \eqref{eq: g_BIF2}. Either of the parametrisations in Theorem \ref{thm: g_LtRtalphat} or Theorem \ref{thm: G_UtVt} can be used for this, though in practice the latter is usually the computationally favourable one. 

\vspace{\baselineskip}

To localise the proposals drawn from $\calL^g(X)$ we introduce random-walk-type proposals in the Wiener space by adapting a \textit{preconditioned Crank-Nicolson (pCN)} scheme (see \cite{Neal1999Regression}, \cite{Beskos08MCMC}, \cite{Cotter2013MCMC}) as follows.  

By Theorem \ref{thm: dPg}, drawing samples of $\calL^g(X)$ is equivalent to sampling the mild solution $X^g$ to Equation \eqref{eq: dXg}. 
The well-posedness of Equation \eqref{eq: dXg} implies the existence of a measurable map $\Gamma$ such that 
\begin{align}
\label{eq: solution_map}
	X^g = \Gamma(x_0, W), \quad \P^g\text{-a.s.,}
\end{align}
where $W$ is a $\P^g$-cylindrical Wiener process.
Assume now that the current value $X$ of the MH sampler is such that $X = \Gamma(x_0, V)$ for some process $V$. 
The pCN proposal is then given by:
\begin{itemize}
    \item[(i)] Draw a Wiener process $W$, independent of $V$;
    \item[(ii)] Set $V'= \sqrt{1 - \beta^2} V + \beta W$;
    \item[(iii)] Propose $X' = \Gamma(x_0, V')$.
\end{itemize}
The parameter $\beta \in (0,1]$ determines the size of the pCN step. For $\beta = 1$, independent proposals of $X^g$ are drawn.
The resulting MH sampler is given in Algorithm \ref{alg: MH_sampler1}.

\begin{algorithm}[t]
\label{alg: MH_sampler1}
\caption{MH Sampler of $\calL^h(X)$}
\LinesNotNumbered  
    \KwIn{Parameters $A,~F, ~Q$ and $x_0$, observations $y$, iterations $N, N_0$, step size $\beta$}
	\KwOut{Samples $\{X_i \}$ of $\calL^h(X)$ (after burn-in period $N_0$)}
	\textbf{Initialize:} \\
	Precompute $G$ by solving the backwards ODEs $(U_t)_{t \in [0,T]}$ and $(V_t)_{t \in [0,T]}$ in \eqref{eq: dU_t} and \eqref{eq: dV_t}\;
	
	Draw a Wiener process $V$ and set $X = \Gamma(x_0,V)$\;
	\For{$i = 0... ~N-1$}{
	 	\textbf{1. MH step in $X$} \\
                \quad(i) ~Draw a Wiener process $W$ and set $V'= \sqrt{1 - \beta^2} V + \beta W$\; 
            \quad(ii) Compute $X' = \Gamma(x_0,V')$ and $M = \min\left(1, \dfrac{\Psi(X')}{\Psi(X)} \right)$\;
            \quad(iii) Draw $U \sim \text{Unif}(0,1)$\; 
        \quad \quad \If{$U < M $}{
      \quad \quad Set $X = X'$ and $V = V'$\;
    }
 	\textbf{2. Save current state $X$} \\
    \quad \If{$i > N_0 $}{
      \quad \quad Set  $X_i = X$\;

    } 
  	}
\end{algorithm}

\begin{remark}
\label{rem: solution_map}
	Since $W$ is a cylindrical Wiener process on $H$, it cannot be treated directly as an $H$-valued process. However, $W$ can be defined as a Wiener process taking values in a larger Hilbert space $H'$ such that $H \hookrightarrow H'$ is embedded in a Hilbert-Schmidt way. Note that such an embedding always exists, see \cite{Hairer2009Introduction}, Section 4.4 for details.
	The solution map $\Gamma$ in \eqref{eq: solution_map} is then to be understood as a measurable mapping $\Gamma: H \times C([0,T];H') \to C([0,T];H)$. 
	
\end{remark}

\begin{remark}
\label{rem: bbV^x0}
	In the spirit of Remark \ref{rem: solution_map}, the MH sampler in Algorithm \ref{alg: MH_sampler1} can be rephrased as targeting a measure $\mathbb{V}^{x_0}$  that is absolutely continuous with respect to a Wiener measure $\mathbb{W}$ on $C([0,T];H')$. For this, let $V$ be a $\P^g$-cylindrical Wiener process on $H$, taking values in the enlargement $H'$ of $H$.
	 Note that the measures $\P^g$ and $\P^h$ depend on the fixed initial state $x_0$.
	Define $\mathbb{V}^{x_0} = \calL^h(V)$ as the law of $V$ under $\P^h$ on $C([0,T];H')$. It immediately follows that
	\begin{align}
	\label{eq: dV^x0}
		\dfrac{\df \mathbb{V}^{x_0}}{\df \mathbb{W}}(V) = \dfrac{\df \calL^h(V)}{\df \calL^g(V)}(V) = \Phi(\Gamma(x_0,V)) = \frac{g(0,x_0)}{h(0,x_0)} \Psi(\Gamma(x_0,V)). 
	\end{align}
	On the other hand, it holds for any bounded and measurable functional $f$ that 
	\begin{align*}
		\E^h\left[ f(\Gamma(x_0,V)) \right] &= \E^g[ f(\Gamma(x_0,V)) \Phi(\Gamma(x_0,V))] =  \E^g[ f(X) \Phi(X)] = \E^h[f(X)].
	\end{align*}
	Hence, $\mathbb{V}^{x_0}$ is the unique measure on $C([0,T];H')$ that is absolutely continuous with respect to $\mathbb{W}$ such that, if $V \sim \mathbb{V}^{x_0}$, then $\Gamma(x_0,V)$ equals in law the smoothing distribution of $X$ given $x_0$.
\end{remark}

\subsection{The case of unknown $X_0$.}
Let now $X_0$ be unknown with prior distribution $\mu_0$ and suppose the following assumption is satisfied.
\begin{assumption}
\label{ass: nu_0}
	There exists a Gaussian measure $\nu_0$ on $(H, \calB(H))$ such that $\mu_0 \ll \nu_0$ with density $\rho(x_0) := \tfrac{\df \mu_0}{\df \nu_0}(x_0)$.
\end{assumption}
A simple calculation shows that the conditional law $\mu_0^h = \calL^h(X_0)$ of $X_0$ given $Y=y$ remains absolutely continuous with respect to $\nu_0$ with density 
\begin{align}
\label{eq: dmu_0^h/dnu_0^h}
	\frac{\df \mu_0^h}{\df \nu_0}(x_0) = \frac{h(0,x_0) \rho(x_0)}{C^h}.
\end{align}
Alternatively, using Bayesian notation and recalling that $h(0,x_0) = p(y \mid x_0)$ and $C^h = \E[h(0,X_0)] = p(y)$, Equation \eqref{eq: dmu_0^h/dnu_0^h} is simply a restatement of the Bayes theorem.

The decomposition $X = \Gamma(X_0,V)$ introduced in the previous section suggests to sample from the full smoothing distribution of $X$ by targeting the  joint measure $\mathbb{V}\otimes \mu_0^h$ defined through the disintegration $\mathbb{V}\otimes \mu_0^h := \mathbb{V}^{x_0}(\df V) \mu_0^h(\df x_0)$ on $C([0,T];H') \otimes H$. Sampling from the disintegrated measure corresponds to first sampling an initial state $X_0 = x_0$ given the observations $y$ and subsequently drawing a sample $V \sim \mathbb{V}^{x_0}$ of the measure $\mathbb{V}^{x_0}$ introduced in Remark \ref{rem: bbV^x0}. The sample $\Gamma(x_0,V)$ then represents a realisation from the smoothing distribution of $X$.

Plugging in the densities in \eqref{eq: dV^x0} and \eqref{eq: dmu_0^h/dnu_0^h} shows that $\mathbb{V}\otimes \mu_0^h$ is absolutely continuous with respect to $\mathbb{W} \otimes \nu_0$ with Radon–Nikodym derivative
\begin{align}
\label{eq: dVmu0dWnu0}
\dfrac{\df (\mathbb{V}\otimes \mu_0^h)}{\df (\mathbb{W} \otimes \nu_0) }(V, x_0) &= \dfrac{g(0,x_0) \rho(x_0)}{C^h} \Psi(\Gamma(x_0,V)).
\end{align}
This density can be used to construct a Gibbs sampler targeting $\mathbb{V}\otimes \mu_0^h$ as summarised in Algorithm \ref{alg: Gibbs_sampler1}.

\begin{algorithm}[b]
\label{alg: Gibbs_sampler1}
\caption{Gibbs Sampler of $\mathbb{V}\otimes \mu_0^h$}
\LinesNotNumbered  
    \KwIn{Parameters $A,~F, ~Q$ and $\rho$, observations $y$, iterations $N, N_0$, step sizes $\beta, \beta_0$}
	\KwOut{Samples $\{(V,X_{0})_i \}$ of $\mathbb{V}\otimes \mu_0^h$ (after burn-in period $N_0$)}
	\textbf{Initialise:} \\
	Precompute $G$ by solving the backwards ODEs $(U_t)_{t \in [0,T]}$ and $(V_t)_{t \in [0,T]}$ in \eqref{eq: dU_t} and \eqref{eq: dV_t}\;
	
	Draw a Wiener process $V$ and $X_0 \sim \nu_0$\;
	\For{$i = 0... ~N-1$}{
 	\textbf{1. Update $V \mid X_0$} \\
                \quad(i) ~Draw a Wiener process $W$ and set $V'= \sqrt{1 - \beta^2} V + \beta W$ \; 
  			\quad(ii) ~Compute $M = \min\left(1, \dfrac{\Psi(\Gamma(X_0,V'))}{\Psi(\Gamma(X_0,V))} \right)$\;
			\quad(iii) Draw $U \sim \text{Unif}(0,1)$\; 
        \quad \quad \If{$U < M $}{
      \quad \quad Set  $V = V'$\;
    }
      \textbf{2. Update $X_0 \mid V$} \\
                \quad(i) ~Draw $z \sim \nu_0$ and set $X_0'= \sqrt{1 - \beta_0^2} X_0 + \beta_0 z$ \; 
  			\quad(ii) ~Compute $M = \min\left(1, \dfrac{\rho(X_0') g(0,X_0')  \Psi(\Gamma(X_0',V))}{\rho(X_0) g(0,X_0)  \Psi(\Gamma(X_0,V))} \right)$\;
			\quad(iii) Draw $U \sim \text{Unif}(0,1)$\; 
        \quad \quad \If{$U < M $}{
      \quad \quad Set  $X_0 = X_0'$\;
    }  
    \textbf{3. Save current state $(V, X_0)$} \\
    \quad \If{$i > N_0 $}{
      \quad \quad Set  $(V, X_0)_i = (V, X_0)$\;

    }  

  	}
\end{algorithm}

\subsection{Parameter estimation}
Suppose the parameters $A, F, Q$ and $\mu_0$ of the dynamics in Equation \eqref{eq: dXt} depend on some unknown parameter $\theta \in \R^p$. 
This dependence carries over to the majority of the processes and functions that we have discussed so far. 
However, for notational clarity, we only make this dependence explicit in certain cases by adding a subscript $\theta$ onto the dependent variables. 
In particular, let us highlight the dependence of the functions $h_{\theta}$, $g_{\theta}$ and $\Gamma_{\theta}$ as well as the measures $\mathbb{V}^{x_0}_{\theta}$ and $\mu_{0,\theta}$ on $\theta$.

The unknown parameter $\theta$ is to be inferred jointly with the unobserved states of $X$, based on the observations $y$. Our methodology makes it natural to do so in a fully Bayesian approach. 
For this, assume that $\theta$ is the realisation of some $\calF_0$-measurable random variable $\Theta$ with Lebesgue density $\pi(\theta)$. With slight abuse of notation, denote by $\pi^y$ the conditional distribution of $\theta$ given $y$. 
We then aim to sample from the joint posterior measure defined by $\mathbb{V}\otimes \mu_0^h \otimes \pi^y := \mathbb{V}_{\theta}^{x_0}(\df V) \mu_{0,\theta}^{h_{\theta}}(\df x_0) \pi^y(\df \theta)$.
Following Equation \eqref{eq: dVmu0dWnu0} and the Bayes theorem, $\mathbb{V}\otimes \mu_0^h \otimes \pi^y$ has density with respect to $\mathbb{W}\otimes \nu_0 \otimes \text{Leb}(\R^p)$ given by 
\begin{align}
\begin{split}
\label{eq: full_posterior_density}
\frac{\df (\mathbb{V}\otimes \mu_0^h \otimes \pi^y)}{\df (\mathbb{W}\otimes \nu_0 \otimes \text{Leb}(\R^p))}(V, x_0, \theta) &= \dfrac{g_{\theta}(0,x_0) \rho_{\theta}(x_0)}{C^{h_{\theta}}} \Psi_{\theta}(\Gamma_{\theta}(x_0,V)) \dfrac{p(y \mid \theta) \pi(\theta)}{\int p(y \mid \theta) \pi(\theta) \df \theta} \\ 
&\propto g_{\theta}(0,x_0) \rho_{\theta}(x_0) \pi(\theta) \Psi_{\theta}(\Gamma_{\theta}(x_0,V)),
\end{split}
\end{align}
where we used that $C^{h_{\theta}} = \E[h_{\theta}(0,X_0)] = p(y \mid \theta)$.
An MCMC algorithm sampler that targets $\mathbb{V}\otimes \mu_0^h \otimes \pi^y$ can hence be defined by adding a Gibbs sampling step $\theta \mid V, x_0$ to Algorithm \ref{alg: Gibbs_sampler1}. The details are summarised in Algorithm \ref{alg: Gibbs_sampler2}.

\begin{algorithm}[t]
\label{alg: Gibbs_sampler2}
\caption{Gibbs Sampler of $\mathbb{V}\otimes \mu_0^h \otimes \pi^y$}
\LinesNotNumbered  
    \KwIn{Parameters $A_{\theta},~F_{\theta}, ~Q_{\theta}$ and $\rho_{\theta}$, observations $y$, prior $\pi$, proposal kernel $q$, \\ iterations $N, N_0$, step sizes $\beta, \beta_0$}
	\KwOut{Samples $\{(V,X_{0}, \theta)_i \}$ of $\mathbb{V}\otimes \mu_0^h \otimes \pi^y$ (after burn-in period $N_0$)}
	\textbf{Initialize:} \\
	Draw $\theta \sim \pi$, a Wiener process $V$ and $X_0 \sim \rho_{\theta} \df \nu_0$;\\
	Compute $G_{\theta}$ by solving the backwards ODEs $(U_t)_{t \in [0,T]}$ and $(V_t)_{t \in [0,T]}$ in \eqref{eq: dU_t} and \eqref{eq: dV_t}\;
	
	\For{$i = 0... ~N-1$}{
 	\textbf{1. Update $V \mid X_0, \theta$} \\
 	See Algorithm \ref{alg: Gibbs_sampler1}, Step 1 with $(\Psi, \Gamma) = (\Psi_{\theta}, \Gamma_{\theta})$;\\
    \textbf{2. Update $X_0 \mid V, \theta$} \\  
     See Algorithm \ref{alg: Gibbs_sampler1}, Step 2 with $(\rho_{\theta}, g_{\theta}, \Psi, \Gamma) = (\rho_{\theta}, g_{\theta}, \Psi_{\theta}, \Gamma_{\theta})$; \\
    \textbf{3. Update $\theta \mid X_0, V$} \\  
     \quad(i) ~Draw $\theta' \sim q(\cdot \mid \theta)$ \; 
    \quad(ii) ~If $A, Q$ depend on $\theta$: Compute $G_{\theta'}$ by solving Eq. \eqref{eq: dU_t} and Eq. \eqref{eq: dV_t};\\
    \quad(iii) ~Compute $M = \min\left(1, \dfrac{q(\theta \mid \theta') \pi(\theta') \rho_{\theta'}(X_0) g_{\theta'}(0,X_0)  \Psi_{\theta'}(\Gamma_{\theta'}(X_0,V))}{q(\theta' \mid \theta) \pi(\theta)  \rho_{\theta}(X_0) g_{\theta}(0,X_0) \Psi_{\theta}(\Gamma_{\theta}(X_0,V))} \right)$\;
			\quad(iii) Draw $U \sim \text{Unif}(0,1)$\; 
        \quad \quad \If{$U < M $}{
      \quad \quad Set  $\theta = \theta'$ and $G_{\theta} = G_{\theta'}$\;
    }
    \textbf{4. Save current state $(V, X_0, \theta)$} \\
    \quad \If{$i > N_0 $}{
      \quad \quad Set  $(V, X_0, \theta)_i = (V, X_0, \theta)$\;

    }  

  	}
\end{algorithm}

\begin{remark}
	If $A$ and $Q$ are independent of the unknown parameter $\theta$, one only needs to solve the backwards ODEs $(U_t)_{t \in [0,T]}$ and $(V_t)_{t \in [0,T]}$ in \eqref{eq: dU_t} and \eqref{eq: dV_t} once. 
	This severely eases the computational burden of Algorithm \ref{alg: Gibbs_sampler2} for the highly relevant case that the parameter $\theta$ of interest parametrises only the nonlinearity $F = F_{\theta}$.
\end{remark}

\subsection{Non-Gaussian observations}
So far we have assumed the observations $Y_i$ given $X_{\ti}$ to be Gaussian with conditional density $k(X_{\ti}, \cdot) = f(\cdot; L X_{\ti}, \Sigma)$.
It is straightforward to let the kernel $k$ depend on the observation time $t_i$ by setting $k_i(X_{\ti}, \cdot) = f(\cdot; L_i X_{\ti}, \Sigma_i)$ for time varying observation operators $L_i: H \to \R^{m_i}$ and extrinsic noise matrices $\Sigma_i \in \R^{m_i \times m_i}$.

More crucially, our proposed methodology can be extended to the case of non-Gaussian observations in the following way. 
Let $l_i$ be some Markov kernel from $H$ to $\R^{m_i}$ and assume observations $Y_i$ are distributed following 
\begin{align}
\label{eq: ap192873}
	Y_i \mid X_{\ti} \sim l_i(X_{\ti}, \cdot).
\end{align}

Substituting $l_i(x_{\ti}, y_i)$ for $k(x_{\ti}, y_i)$ in the definition of the $h$-function in Equation \eqref{eq: def_h} preserves the properties of $h(t,x)$ established in Propositions \ref{prop: h_properties} and \ref{prop: Eht_martingale} as well as Theorem \ref{thm: dPh}.
Hence, the smoothing distribution of $X$ under the observation scheme in \eqref{eq: ap192873} is the law of $X$ under the change of measure 
\begin{align*}
	\dfrac{\df \P^h}{\df \P}(X) = E^h_T = \frac{1}{C^{h}} \left( \prod_{i=1}^{n} l_i(X_{\ti},y_i) \right)
\end{align*}
on $\calF_T$. 
For the guided distribution of $X$, the Markov kernel $k_i$ needs to remain Gaussian in order to preserve the tractable expressions for $g(t,x)$ obtained in Section \ref{sec: guided_measure}. It is therefore defined under the change of measure
\begin{align*}
	\dfrac{\df \P^g}{\df \P}(X) = E^g_T = \frac{1}{C^{g}} \left( \prod_{i=1}^{n} k_i(X_{\ti},y_i) \right) \Psi(X).
\end{align*}
In total, the R.N. derivative of the smoothing distribution with respect to the guided distribution then takes on the form 
\begin{align*}
	\Phi(X) = \frac{C^g}{C^h} \left( \prod_{i=1}^n \dfrac{l_i(X_{\ti}, y_i)}{k_i(X_{\ti}, y_i)} \right) \Psi(X).
\end{align*}

\section{Case study: stochastic Amari equation}
\label{sec: amari}

\subsection{Stochastic Amari equation}
In the following section, we evaluate our proposed methodology for filtering, smoothing, and parameter estimation in the context of a partially observed stochastic Amari model. 
For this, let $H = L^2(D)$ for some bounded spatial domain $D \subset \R^d$. For ease of exposition, we assume the domain $D$ to be rectangular and impose periodic boundary conditions on it.
Consider then the mild solution $X$ to the stochastic Amari equation
\begin{align}
\begin{split}
\label{eq: dX_Amari}
\df X_t &= \left[ - X_t + F(X_t) \right] \df t + Q^{\frac12} \df W_t, \quad  t \geq 0,\\ 
 \quad X_0 &= x_0,
\end{split}
\end{align}
with nonlinearity $F: L^2(D) \to L^2(D)$ defined by 
\begin{align}
\label{eq: F_Amari}
F(X_t) &= \int_D k_F(\xi, \xi') f(X_t(\xi')) \df \xi'.
\end{align}
The process $X_t$ models a stochastic neural field on $D$, with $X_t(\xi)$ representing the average neural activity at location $\xi \in D$ and time $t \geq 0$. 
 
The kernel $k_F(\xi,\xi')$ captures neural connectivity between locations $\xi$ and $\xi'$, whereas $f$ is an activation function, typically a sigmoid-like function or similar. 
In our experiments, we consider 
\begin{align}
\label{eq: def_kf}
\begin{split}
		k_F(\xi, \xi') &= \frac{A}{\sqrt{\pi}} \exp\left( -(|\xi - \xi'| - \delta)^2 \right) - \frac{A}{\sqrt{\pi} B} \exp\left( -\left( \frac{|\xi - \xi'| - \delta}{B} \right)^2 \right) \\
		f(x) &= \frac{1}{1 + \exp\big(-\eta \, x + \zeta \big)} - \frac{1}{1 + \exp(\zeta)},
\end{split}
\end{align}
for a set of parameters $A >0, B >0, \eta >0, \zeta >0$ and $\delta \in \R$.

For the mild solution $X$ to exist, one may consider any trace-class operator $Q$ on $L^2(D)$. We let $Q$ be a diagonal operator defined by its eigenvalues 
\begin{align}
\label{eq: def_qj}
	q_l = \sigma_0^2 (\rho_0^{-2} + (2 \pi l)^2)^{-(d/2 + \eta_0)}
\end{align}
with respect to the Fourier basis functions $(e_l)_{l=1}^{\infty}$ on $L^2(D)$. The operator $Q$ is then the Hilbert-Schmidt integral operator defined by the Matérn kernel on $D$, see for example \cite{Borovitskiy2020Matern}, Theorem 5.

\vspace{\baselineskip}

\paragraph{\bfseries Observation and guiding scheme}
We consider observations $Y_i \in \R^m$ of the latent process $X$ at observation times $(t_i)_{i=1}^n$ defined as in Eq. \eqref{eq: def_Y2}. The observation operator $L: L^2(D) \to \R^m$ we set as 
\begin{align}
\label{eq: def_L}
L x = \left( \int_D x(\xi) \om_j(\xi) \, \df \xi \right)_{j =1}^m = \left( \langle x, \om_j \rangle  \right)_{j =1}^m \in \R^m
\end{align}
for a set of positive, square-integrable functions $\om_j \in L^2(D), j = 1, \ldots m$.
The adjoint operator $L^*$ then takes the form 
\begin{align*}
	L^* y = \sum_{j=1}^m y_j \om_j \in L^2(D), \quad y \in \R^m.
\end{align*}
Specifically, we let $\om_j = \tfrac{1}{|D_j|}\mathbbm{1}_{D_j}$ for a series of mutually disjoint subsets $D_j \subset D$, $j =1 \ldots m$. The operator $L$ then represents observing localised spatial averages of $X_{t_i}(\xi)$ over the domains $D_j$ at observation time $t_i$.

\begin{figure}[b]
    \centering
    \includegraphics[width=0.9\textwidth]{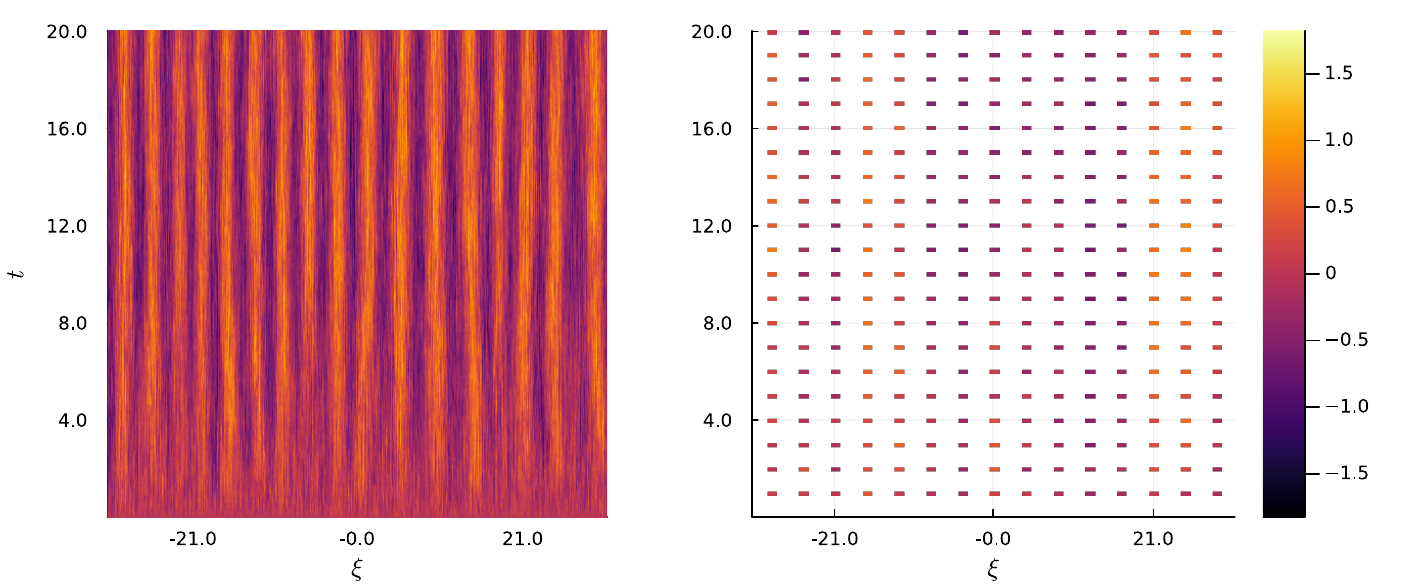}
    \caption{Data generating process of Experiment 1. Colour intensity represents the value of the process and observations, whereas time $t$ is plotted along the y-axis and spatial coordinate $\xi$ along the x-axis. Left: Heatmap of the sample path $X(t,\xi)$ of Eq. \eqref{eq: dX_Amari} with parameters as in \eqref{eq: params} and $\delta_0 = 0$. Right: Observations $Y_i$ at $n = 20$ observation times with $m = 15$ local averages of $X(\ti,\xi)$ per observation.}
    \label{fig: obs_forward_filtering_A}
\end{figure}

\begin{figure}[b]
    \centering
    \includegraphics[width=0.9\textwidth]{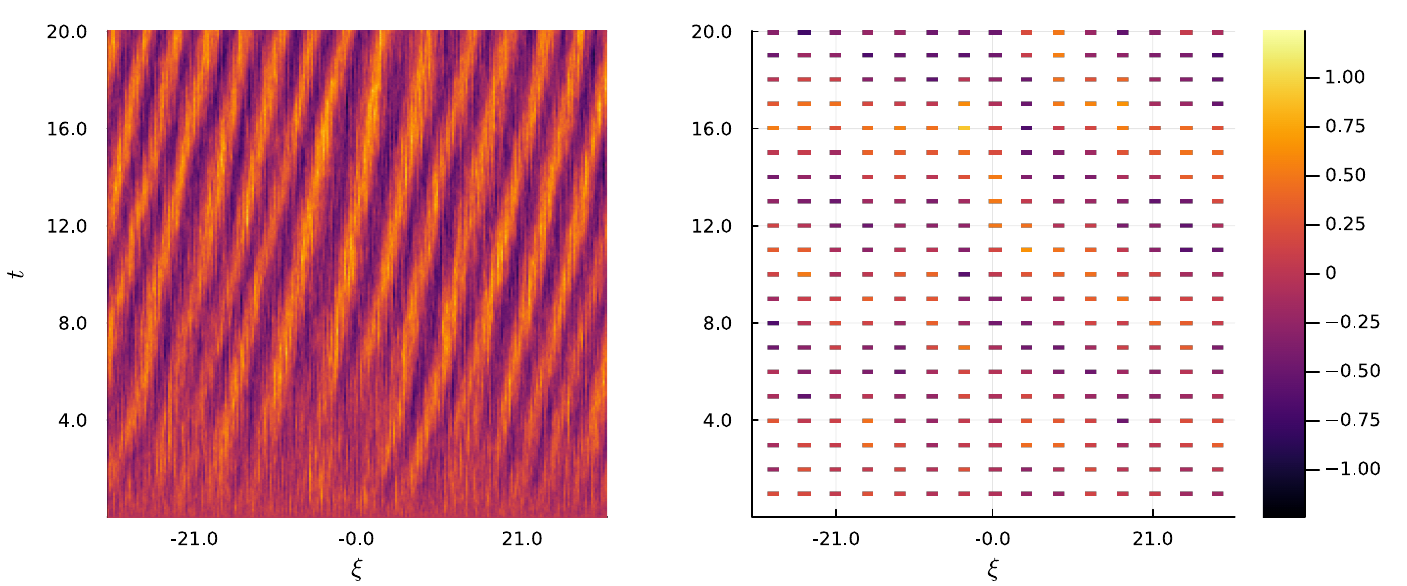}
    \caption{Data generating process of Experiment 2. Colour intensity represents the value of the process and observations, whereas time $t$ is plotted along the y-axis and spatial coordinate $\xi$ along the x-axis. Left: Heatmap of the sample path $X(t,\xi)$ of Eq. \eqref{eq: dX_Amari} with parameters as in \eqref{eq: params} and $\delta_0 = 0.5$. Right: Observations $Y_i$ at $n = 20$ observation times with $m = 15$ local averages of $X(\ti,\xi)$ per observation.}
    \label{fig: obs_forward_filtering_B}
\end{figure}

\vspace{\baselineskip}
To construct the guiding distribution, we set $(a_t)_{t \geq 0}$ in \eqref{eq: dZt} as $a_t \equiv 0$.
For the one-step ahead guiding distribution we receive the following closed-term expression. The computations can be found in Appendix \ref{app: B}.
\begin{proposition}
\label{prop: amari_G}
In the setup above, the guiding term $G_i$ of Eq. \eqref{eq: dXg2} at observation time $t_i$ is given by 
\begin{align*}
	G_i(t,x) &= \exp(- (t - t_i)) L^* \left( \Sigma + \frac{1 - \exp(- 2 (t - t_i))}{2} L Q L^* \right)^{-1} \left( y_i - \exp(-(t - t_i)) L x \right)\end{align*}
where $L Q L^*$ is the symmetric, positive definite $m \times m$-matrix with entries
\begin{align*}
	(L Q L^*)[i,j] = \sum_{l=1}^{\infty} q_l \langle \om_i, e_l \rangle \langle \om_j, e_l \rangle.
\end{align*}
\end{proposition}

\subsection{Filtering}

\paragraph{\bfseries Model setup}
For the filtering problem, we run two experiments with different choices $\delta_0 = 0$ and $\delta_1 = 0.5$ of the model parameter $\delta$. 
The remaining model parameters in \eqref{eq: def_kf} and \eqref{eq: def_qj} are set in both experiments as
\begin{align}
\label{eq: params}
	[A, B, \eta, \zeta, \sigma_0, \rho_0, \eta_0] = [4.0, 1.5, 10.0, 0.5, 3 \cdot 10^5, 5 \cdot 10^{-5}, 1.0]
\end{align}
and the initial state is set to be $X_0 \equiv 0$. Under this parametrisation, the stochastic Amari equation exhibits either steady state Turing patterns ($\delta_0 = 0$) or travelling waves ($\delta_1 = 0.5$) emerging from the initial state $X_0$. 

Experiments are carried out on data generated by forward simulating the solution $X$ to Equation \eqref{eq: dX_Amari} on the spatiotemporal domain $[0,T] \times D = [0,20] \times [-10 \pi, 10\pi]$ with step sizes $\Delta_t = 0.02$ and $\Delta_x = 20 \pi / 2^8$.
\vspace{\baselineskip}

\begin{figure}[htbp]
    \centering
    \subfigure[All methods, $n=20$.]{%
        \includegraphics[width=0.45\textwidth]{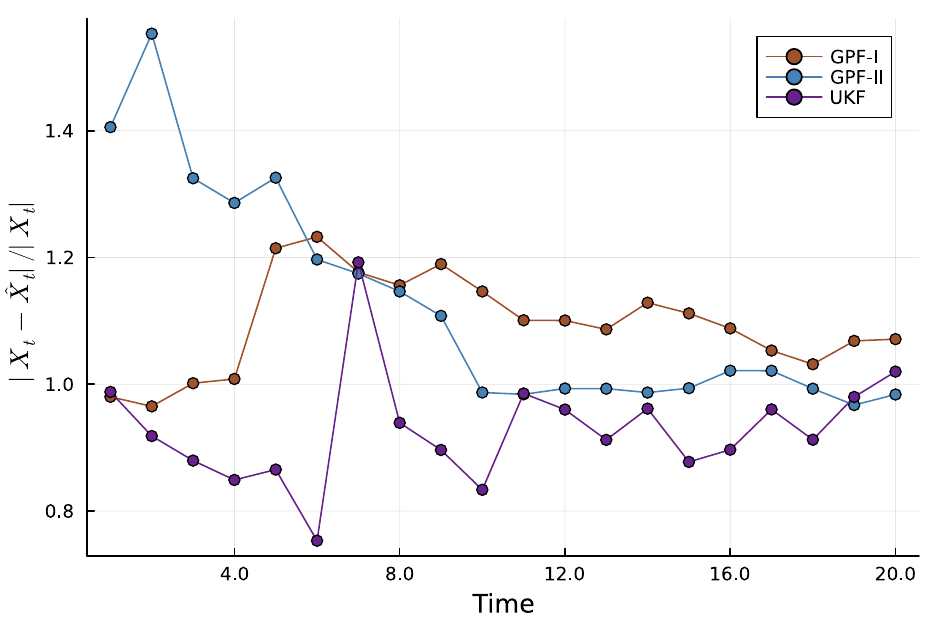}
    }
    \hspace{0.25cm}
    \subfigure[GPF-I, $n = 20$, $10$, $5$.]{%
        \includegraphics[width=0.45\textwidth]{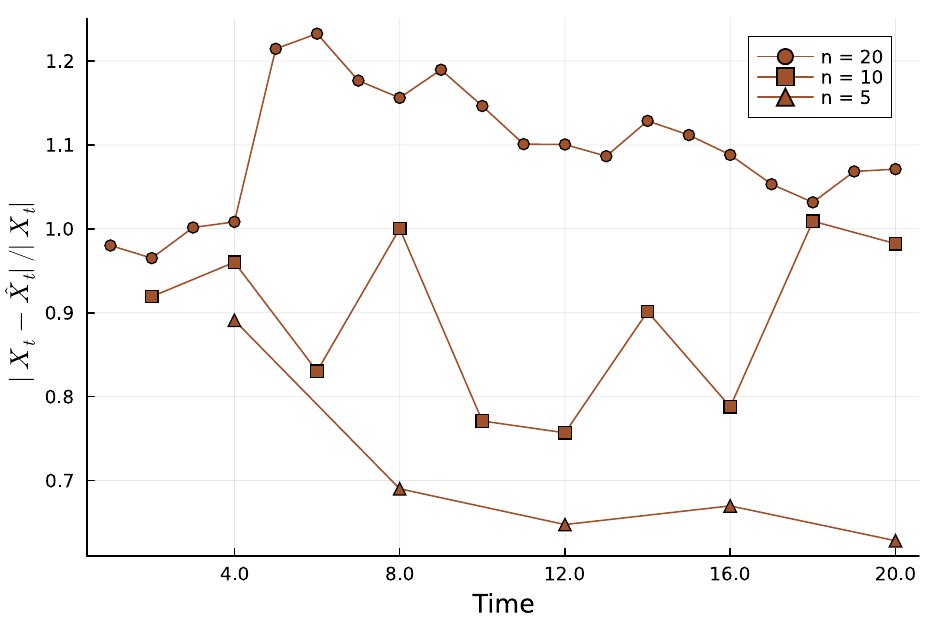}
    }
    \vspace{0.25cm}
    
    \subfigure[GPF-II, $n = 20$, $10$, $5$.]{%
        \includegraphics[width=0.45\textwidth]{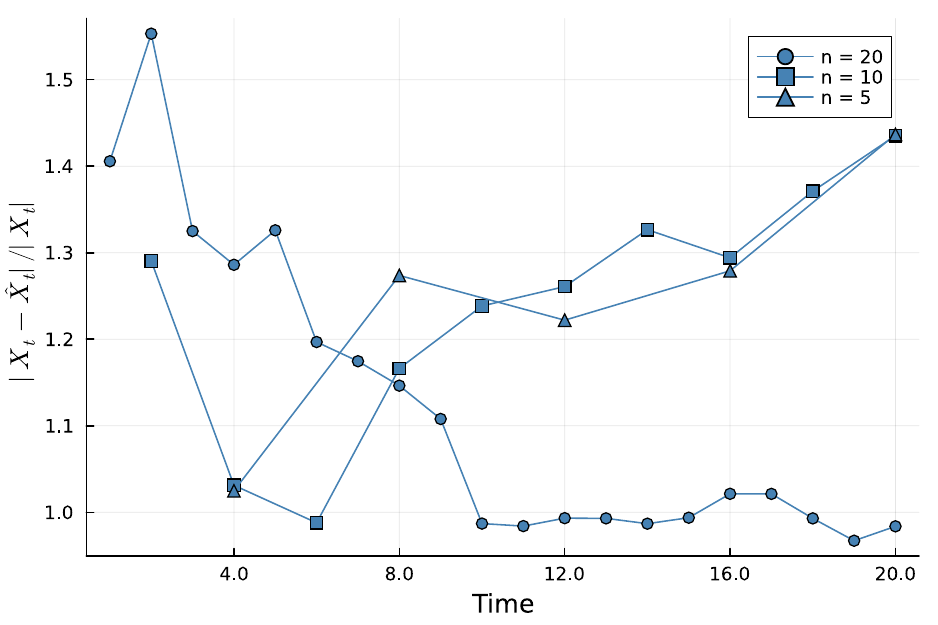}
    }
    \hspace{0.25cm}
    \subfigure[UKF, $n = 20$, $10$, $5$.]{%
        \includegraphics[width=0.45\textwidth]{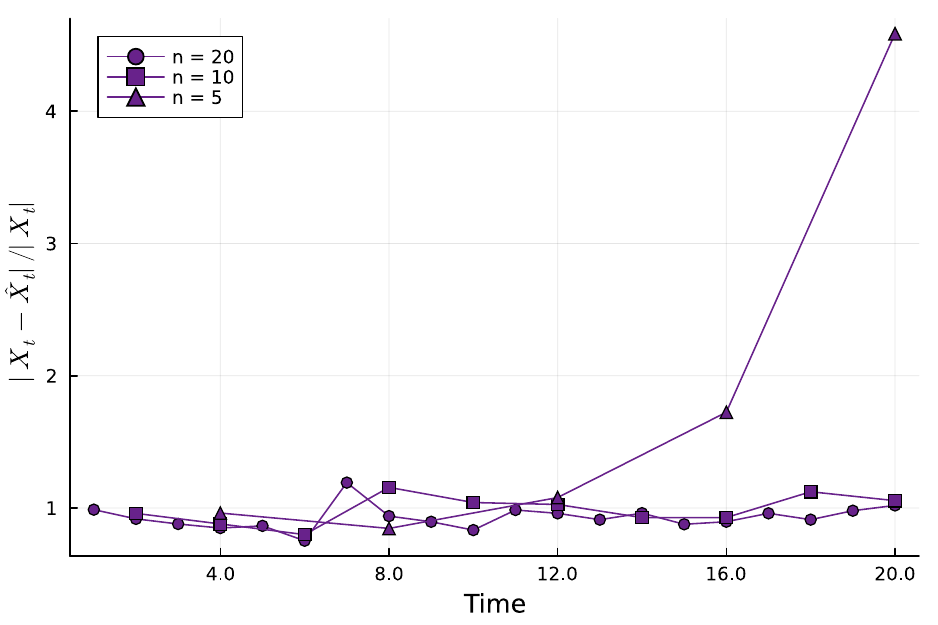}
    }

    \caption{Relative errors $|\hat{X}_{\ti} - X_{\ti}|/|X_{\ti}|$ at observation times $\ti$ for the filtering estimates $\hat{X}_{\ti}$ based on the data set of Figure \ref{fig: obs_forward_filtering_A} and the three methods GPF-I (our proposed method), GPF-II (method with guiding term in \eqref{eq: Gtld}) and UKF. Subfigure (a): a comparison of estimation errors for all three methods in the dense observation case.
    Subfigures (b) - (d): estimation errors of individual methods across subsampled datasets.}
    \label{fig: filtering_errors_A}
\end{figure}

\begin{figure}[htbp]
    \centering
    \subfigure[GPF-I, $n=20$.]{%
        \includegraphics[width=0.45\textwidth]{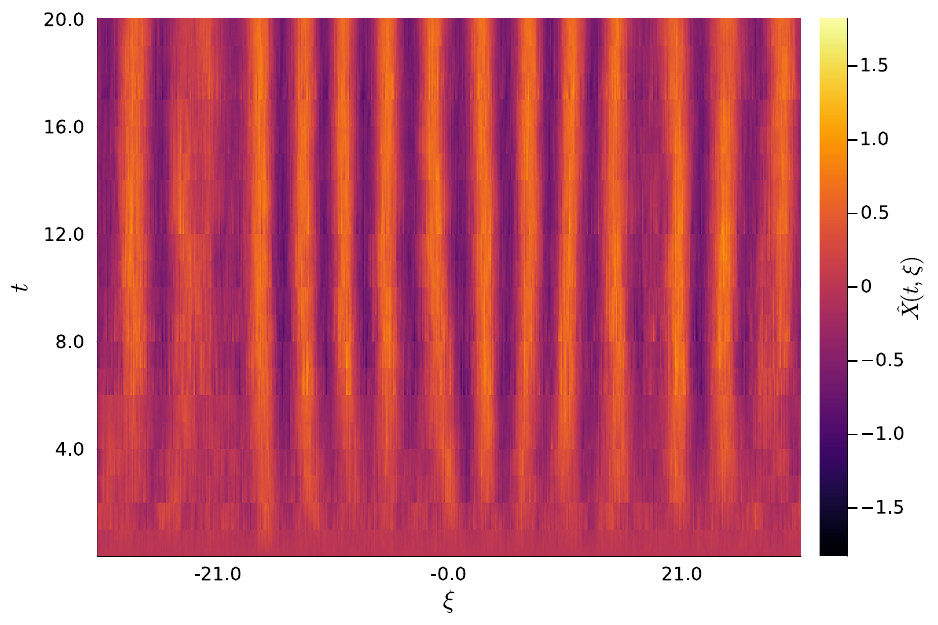}
    }
    \hspace{0.25cm}
    \subfigure[GPF-I, $n=20$.]{%
        \includegraphics[width=0.45\textwidth]{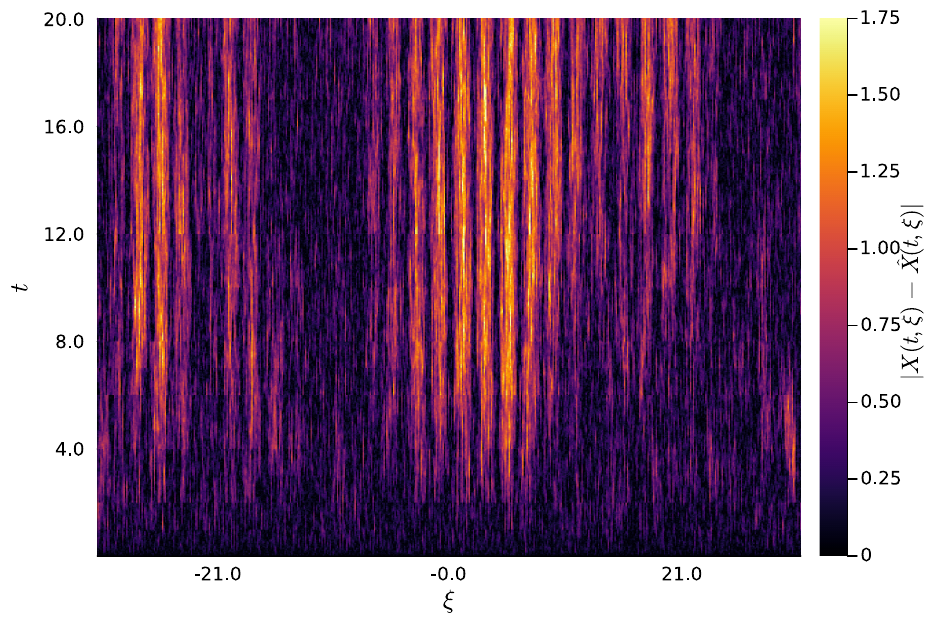}
    }
    \vspace{0.15cm}
        \subfigure[GPF-II, $n = 20$.]{%
        \includegraphics[width=0.45\textwidth]{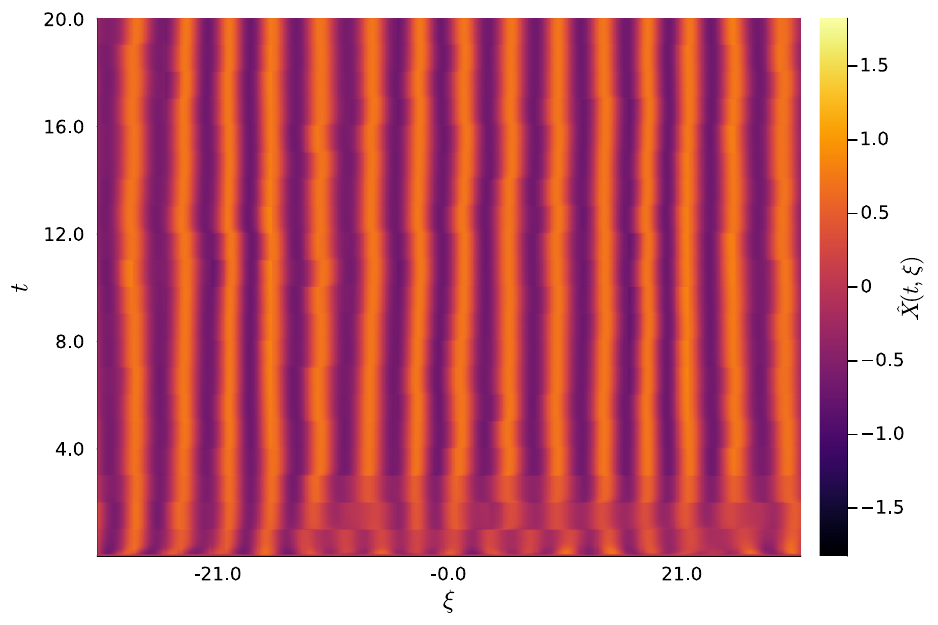}
    }
    \hspace{0.25cm}
    \subfigure[GPF-II, $n = 20$.]{%
        \includegraphics[width=0.45\textwidth]{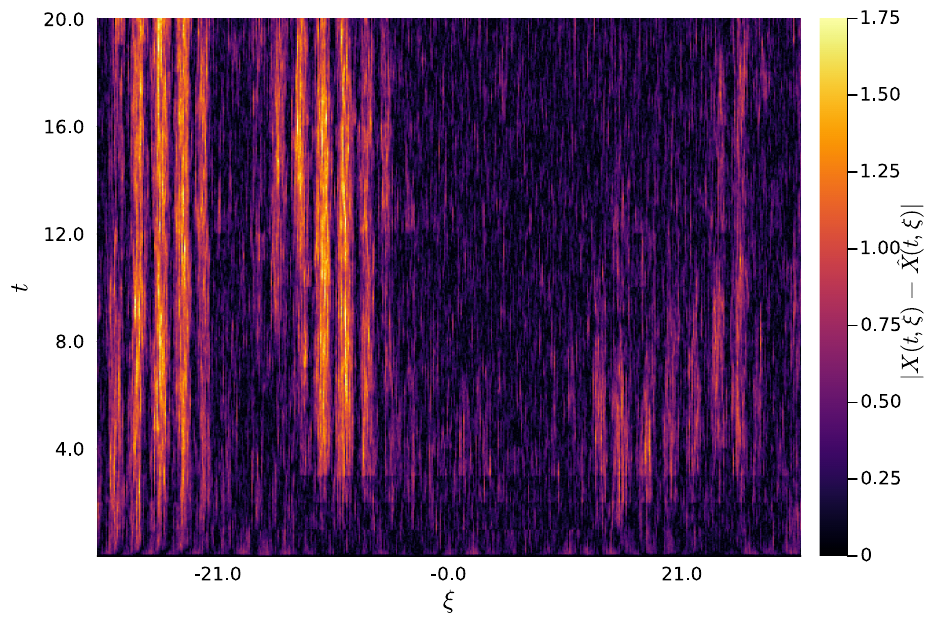}
    }
    \vspace{0.15cm}
        \subfigure[GPF-I, $n = 5$.]{%
        \includegraphics[width=0.45\textwidth]{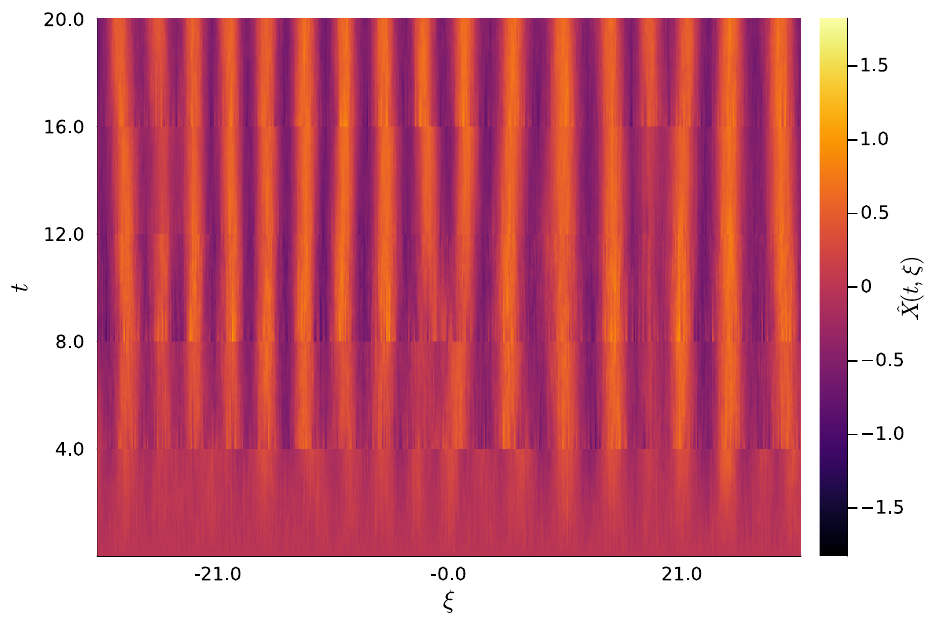}
    }
    \hspace{0.25cm}
    \subfigure[GPF-I, $n = 5$.]{%
        \includegraphics[width=0.45\textwidth]{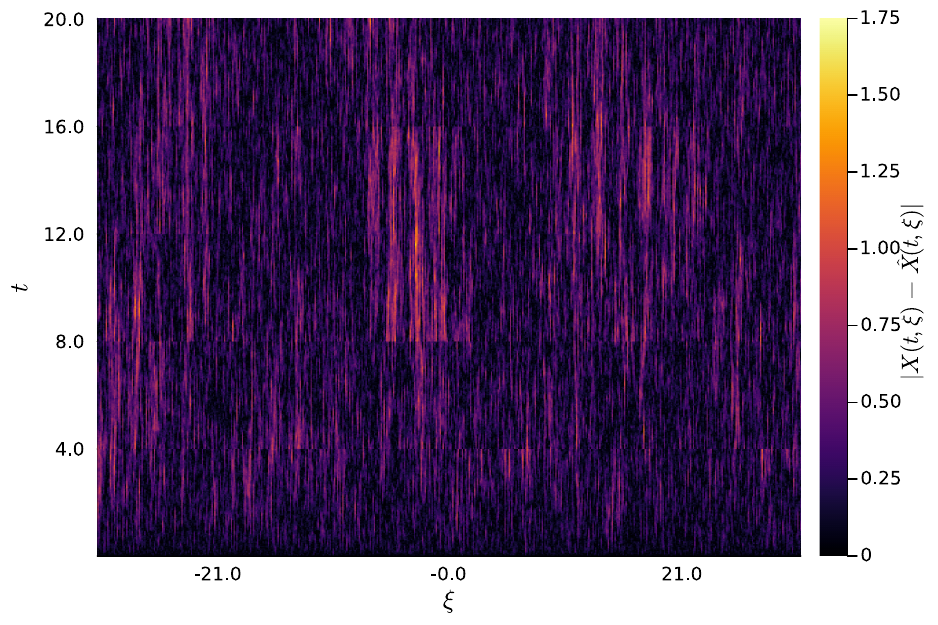}
    }
    \vspace{0.15cm}
        \subfigure[GPF-II, $n = 5$.]{%
        \includegraphics[width=0.45\textwidth]{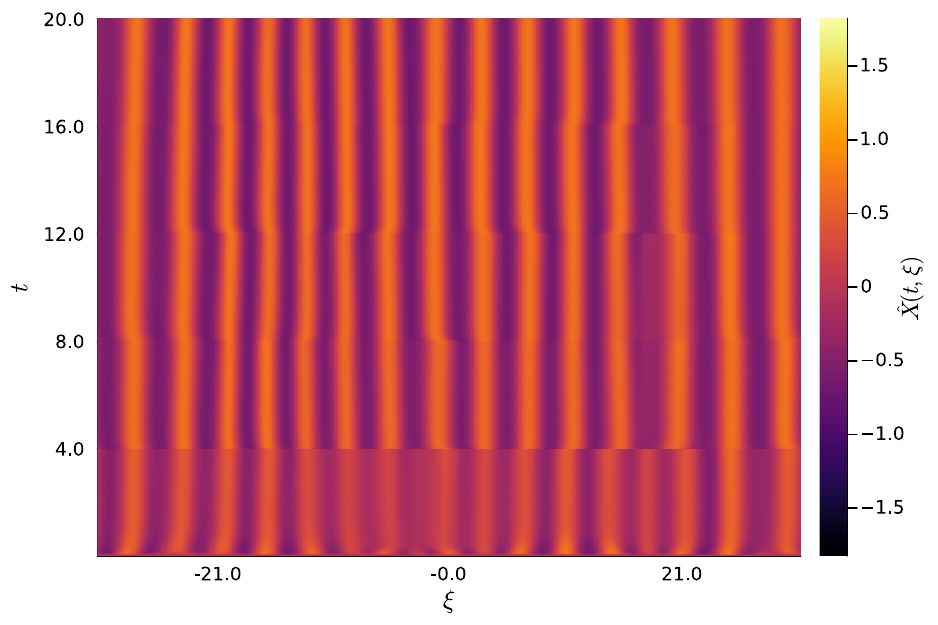}
    }
    \hspace{0.25cm}
    \subfigure[GPF-II, $n = 5$.]{%
        \includegraphics[width=0.45\textwidth]{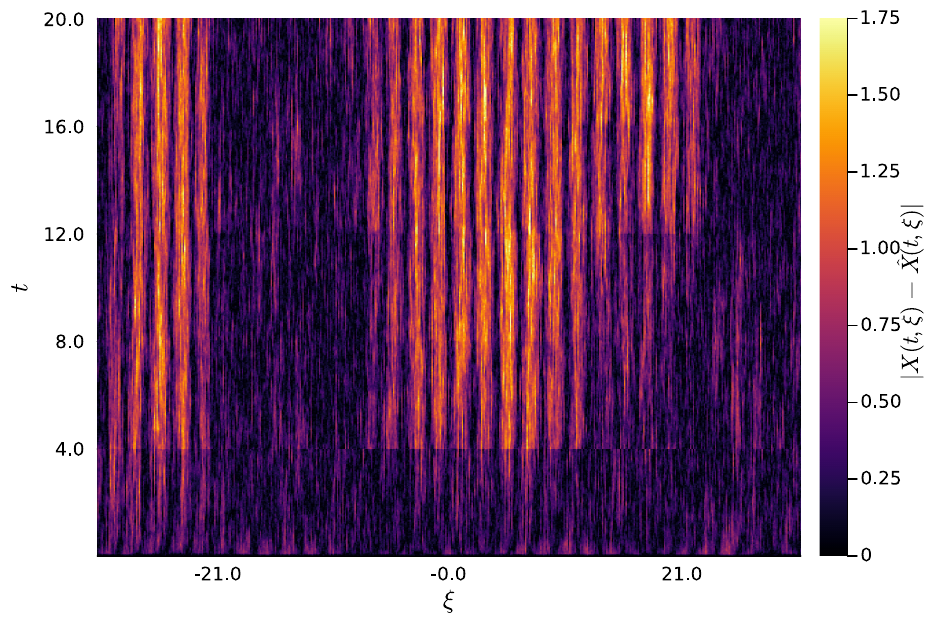}
    }
    \vspace{0.15cm}
	 \caption{Filtering estimates for the two guided particle filter methods based on the dataset in \ref{fig: obs_forward_filtering_A}. Left column: Heatmaps of the filtered path estimate. Right column: Heatmaps of the absolute estimation error. The first and third rows show GPF-I (our proposed method), while the second and fourth rows show GPF-II (method with guiding term in \eqref{eq: Gtld}) under the dense and sparse observation schemes, respectively.}
    \label{fig: filtering_heatmaps_A}
\end{figure}

\paragraph{\bfseries Data}
A total number of $n = 20$ observations are taken at observation times $1.0, 2.0, \ldots, 20.0$.
Each observation $Y_i, i = 1, \ldots, n,$ is defined as in Equation \eqref{eq: def_L} with $\om_j =  \tfrac{1}{|D_j|}\mathbbm{1}_{D_j}$ and measurement domains $D_j$ defined at $m = 15$ equal spatial intervals and length $|D_j| = 1$. As measurement noise we set $\Sigma_i = 0.01 \cdot I$.

The data generating processes $X$ and resulting observations are shown in Figure \ref{fig: obs_forward_filtering_A} and \ref{fig: obs_forward_filtering_B} for the two choices of $\delta$ respectively. 

Additionally, for each experiment, we investigate performance of the proposed methods when downsampling the data from $n =20$ to $n=10$ and $n=5$ observations. 

\vspace{\baselineskip}

\begin{figure}[htbp]
    \centering
    \subfigure[All methods, $n=20$.]{%
        \includegraphics[width=0.45\textwidth]{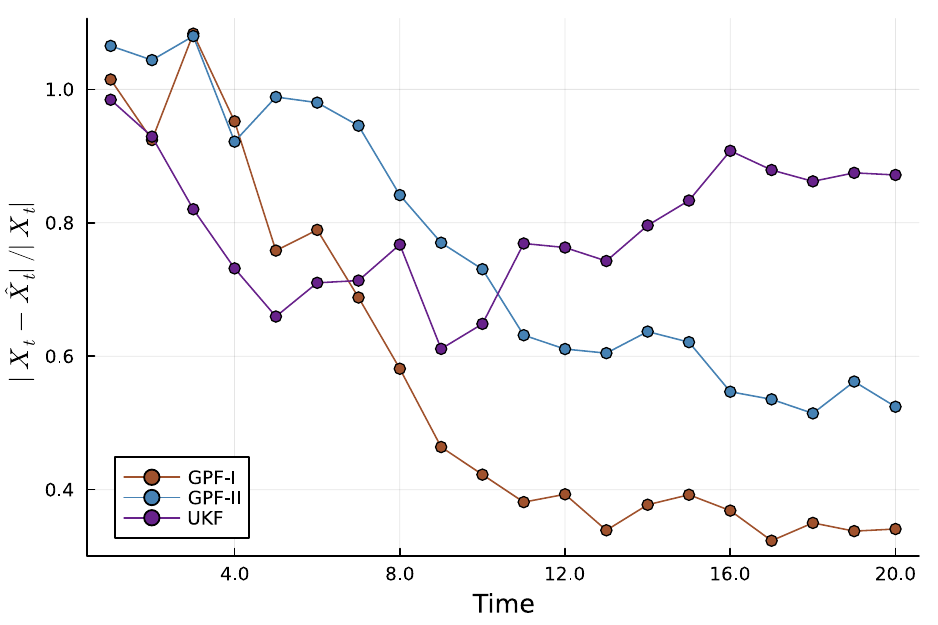}
    }
    \hspace{0.25cm}
    \subfigure[GPF-I, $n = 20$, $10$, $5$.]{%
        \includegraphics[width=0.45\textwidth]{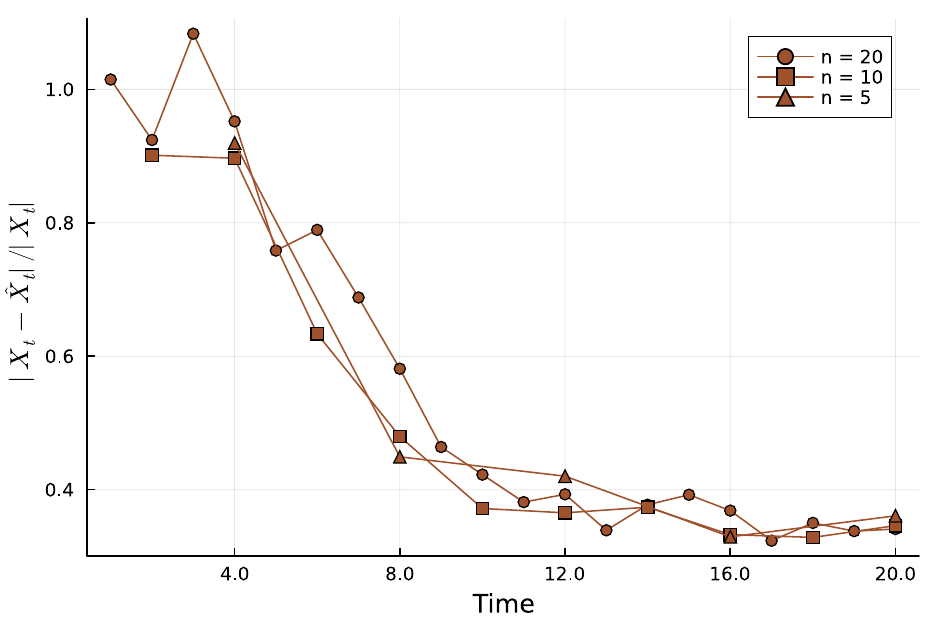}
    }
    \vspace{0.25cm}
    
    \subfigure[GPF-II, $n = 20$, $10$, $5$.]{%
        \includegraphics[width=0.45\textwidth]{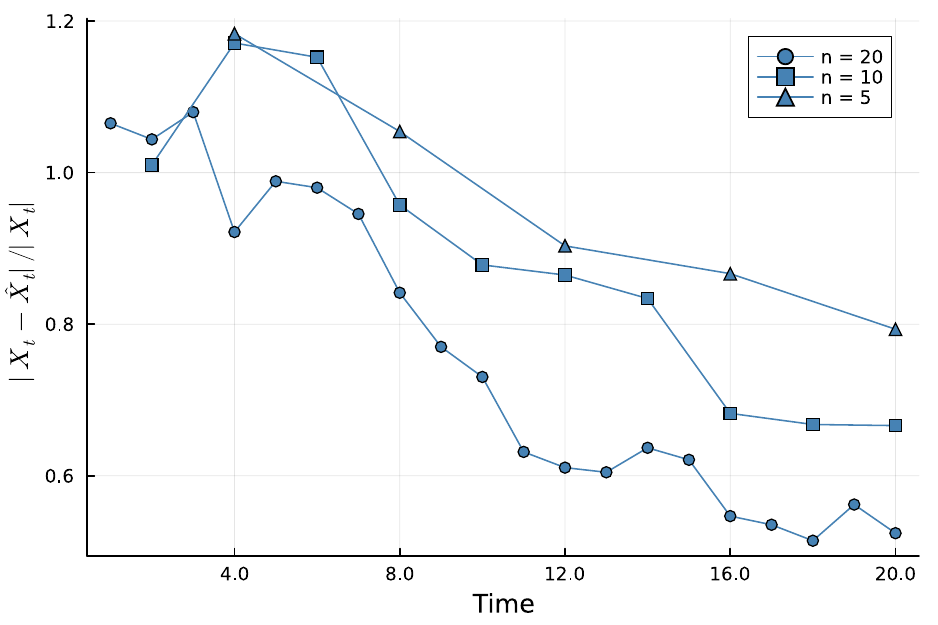}
    }
    \hspace{0.25cm}
    \subfigure[UKF, $n = 20$, $10$, $5$.]{%
        \includegraphics[width=0.45\textwidth]{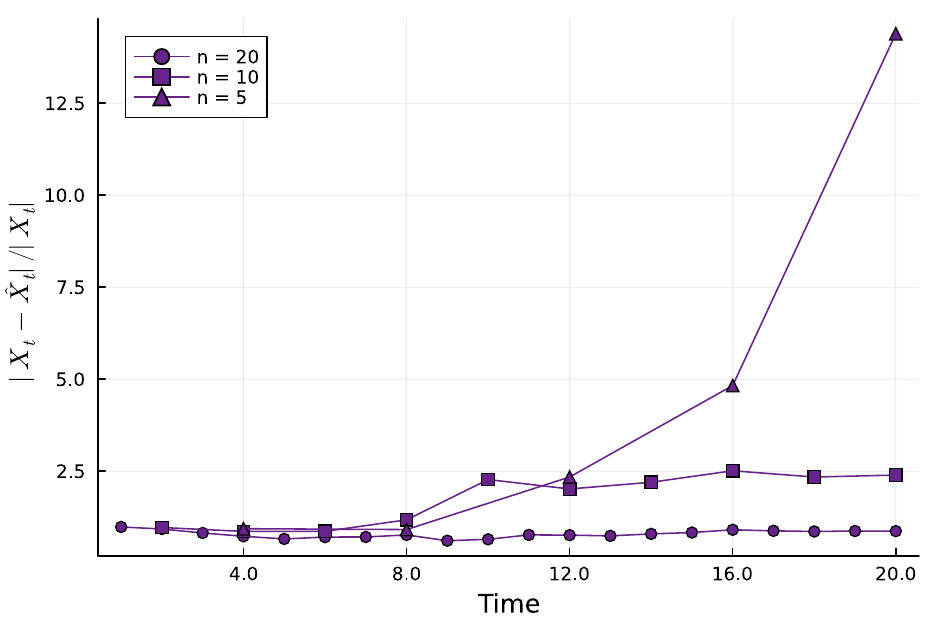}
    }

    \caption{Relative errors $|\hat{X}_{\ti} - X_{\ti}|/|X_{\ti}|$ at observation times $\ti$ for the filtering estimates $\hat{X}_{\ti}$ based on the data set of Figure \ref{fig: obs_forward_filtering_B} and the three methods GPF-I (our proposed method), GPF-II (method with guiding term in \eqref{eq: Gtld}) and UKF. Subfigure (a): a comparison of estimation errors for all three methods in the dense observation case.
    Subfigures (b) - (d): estimation errors of individual methods across subsampled datasets.}
    \label{fig: filtering_errors_B}
\end{figure}

\begin{figure}[htbp]
    \centering
    \subfigure[GPF-I, $n=20$.]{%
        \includegraphics[width=0.45\textwidth]{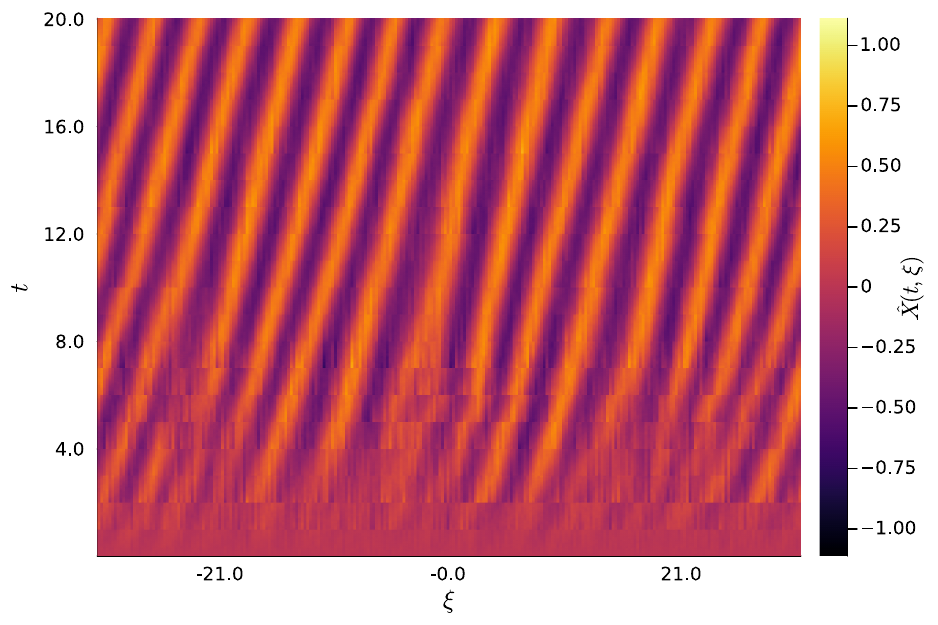}
    }
    \hspace{0.25cm}
    \subfigure[GPF-I, $n=20$.]{%
        \includegraphics[width=0.45\textwidth]{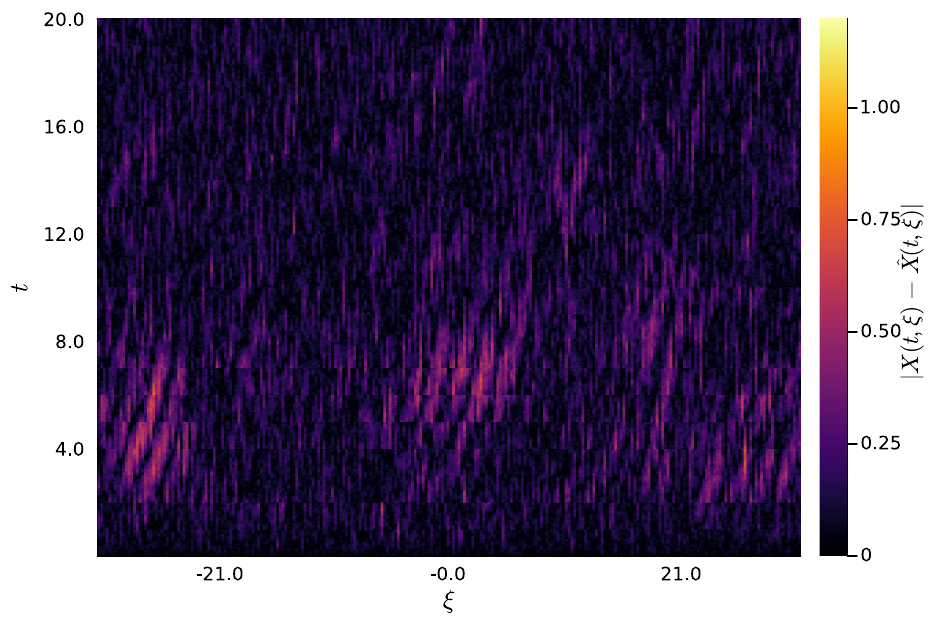}
    }
    \vspace{0.15cm}
        \subfigure[GPF-II, $n = 20$.]{%
        \includegraphics[width=0.45\textwidth]{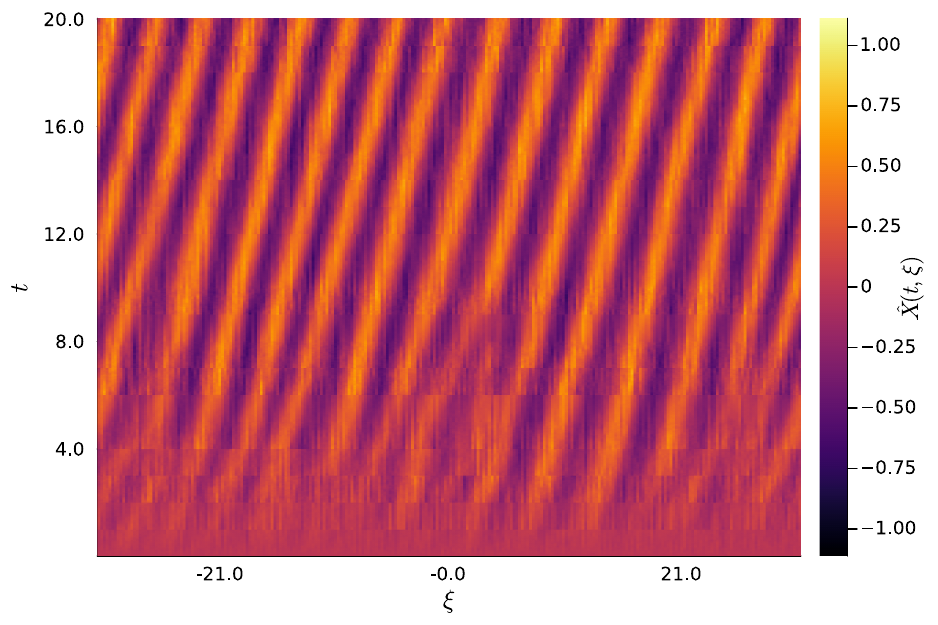}
    }
    \hspace{0.25cm}
    \subfigure[GPF-II, $n = 20$.]{%
        \includegraphics[width=0.45\textwidth]{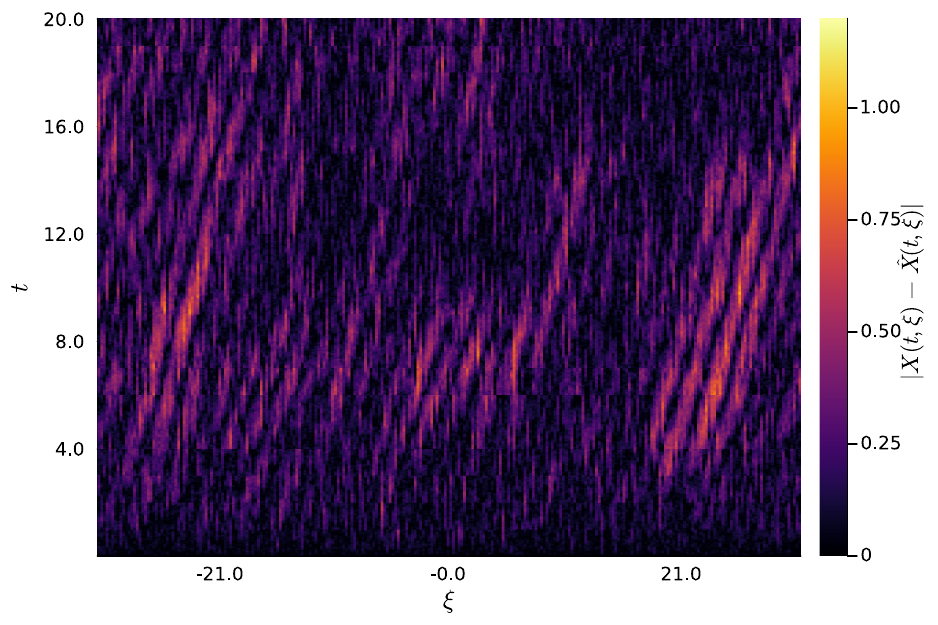}
    }
    \vspace{0.15cm}
        \subfigure[GPF-I, $n = 5$.]{%
        \includegraphics[width=0.45\textwidth]{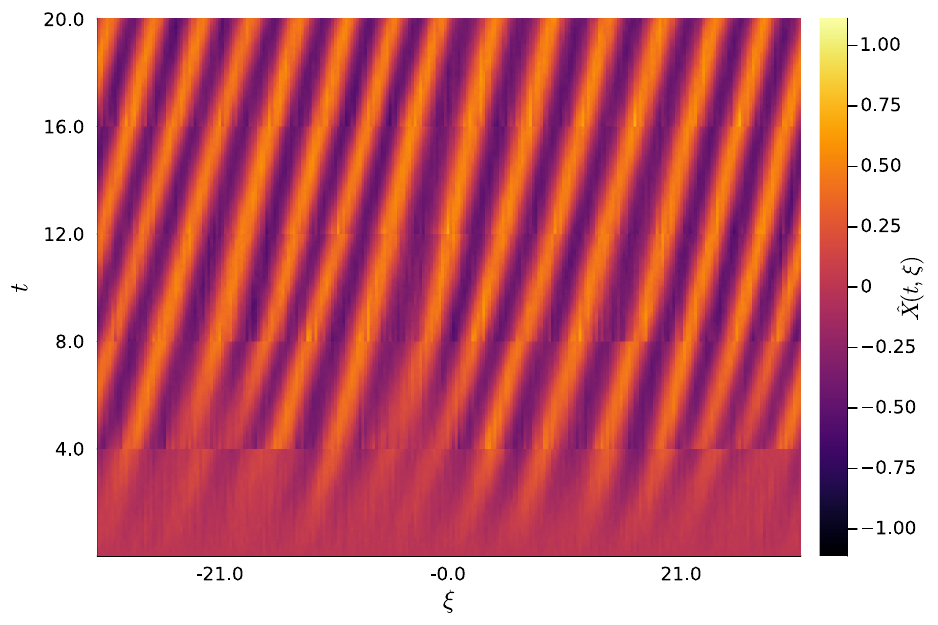}
    }
    \hspace{0.25cm}
    \subfigure[GPF-I, $n = 5$.]{%
        \includegraphics[width=0.45\textwidth]{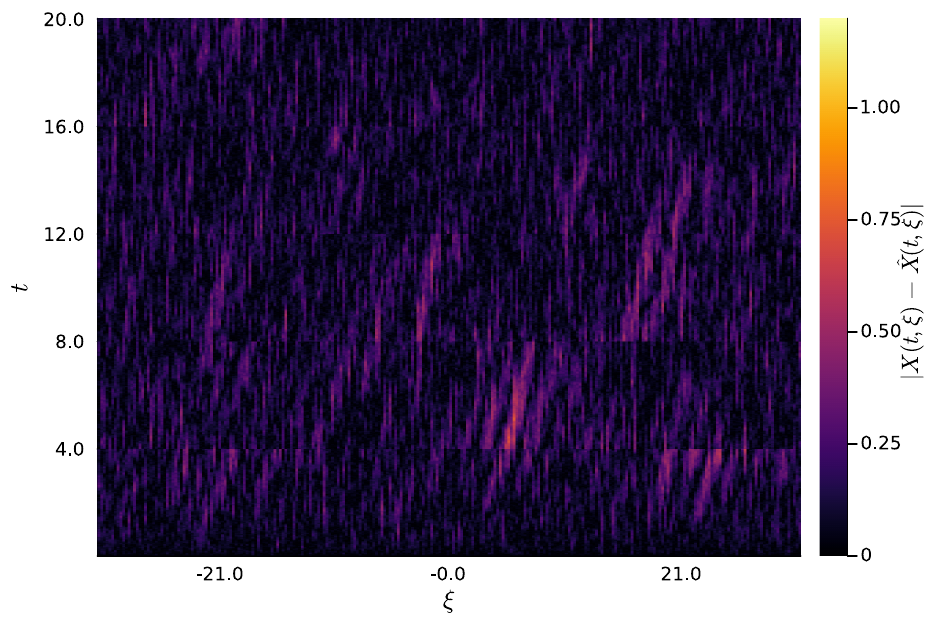}
    }
    \vspace{0.15cm}
        \subfigure[GPF-II, $n = 5$.]{%
        \includegraphics[width=0.45\textwidth]{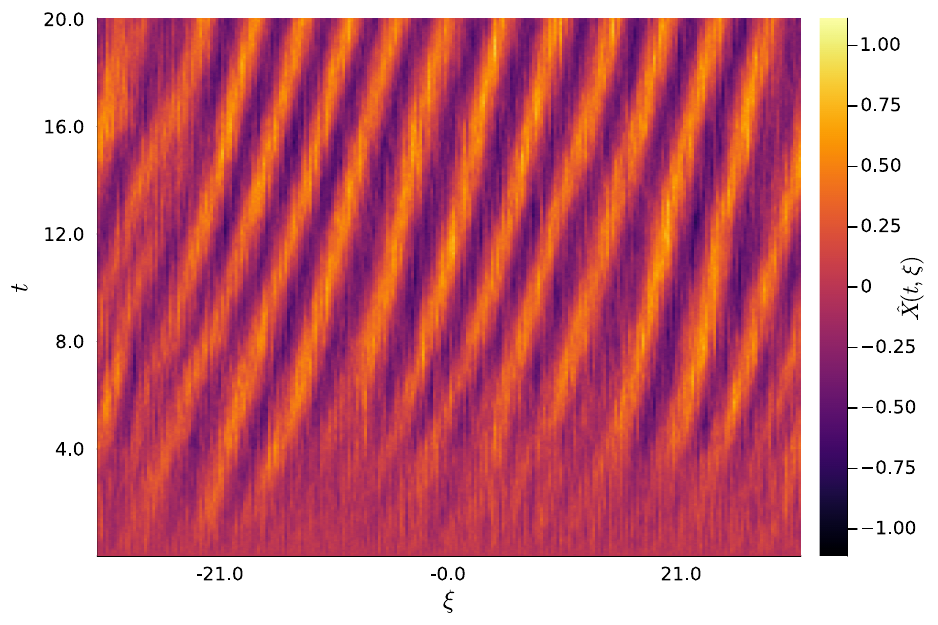}
    }
    \hspace{0.25cm}
    \subfigure[GPF-II, $n = 5$.]{%
        \includegraphics[width=0.45\textwidth]{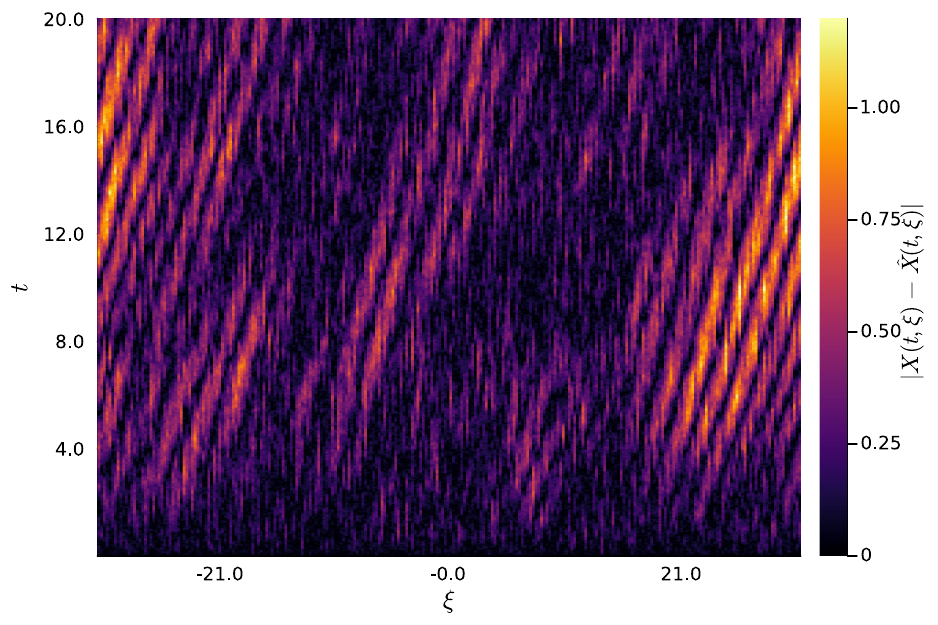}
    }
    \vspace{0.15cm}
	 \caption{Filtering estimates for the two guided particle filter methods based on the dataset in \ref{fig: obs_forward_filtering_A}. Left column: Heatmaps of the filtered path estimate. Right column: Heatmaps of the absolute estimation error. The first and third rows show GPF-I (our proposed method), while the second and fourth rows show GPF-II (method with guiding term in \eqref{eq: Gtld}) under the dense and sparse observation schemes, respectively.}
    \label{fig: filtering_heatmaps_B}
\end{figure}

\paragraph{\bfseries Methods}
For each of the two experiments and the three (downsampled) observation schemes, we run the proposed guided particle filter in Algorithm \ref{alg: guided_particle_filter}, including the tempering and moving steps of Algorithm \ref{alg: tempering_steps} with $J = 100$ particles, $N = 30$ MCMC move steps with pCN step size $\beta = 0.1$ and ESS threshold parameter $\alpha = 0.75$. 

Additionally, we implement the following two established methods as a baseline for comparison: 
\begin{itemize}
\item[1.] The guided particle filter as proposed in \cite{Llopis2018Particle}. This method differs to ours in the proposal distribution $\Q$ of Eq. \eqref{eq: guided_filtering_proposal}.
The proposals chosen in \cite{Llopis2018Particle} are defined by solving the guided process \eqref{eq: dXg2} with guiding term 
\begin{align}
\label{eq: Gtld}
	\tilde{G}_i(t,x) = L^* \left( \Sigma + (t_i - t) L Q L^* \right)^{-1} (y - L x), \quad t \in [\tim, \ti].
\end{align}
In the context of our work, this guiding term can be derived by taking $A = 0$ in the auxiliary process $Z$. To differentiate between the two guided particle filters in the experiments we denote by GPF-I our version and GPF-II the one based on the guiding term $\tilde{G}$. For each experiment, both algorithms are implemented with the same hyperparameter choices of $J, N, \beta$ and $\alpha$.
\item[2.] The unscented Kalman filter (UKF) introduced in \cite{Wan2000Unscented}. The UKF has been proposed for the filtering problem of an Amari model in \cite{Schiff2007Kalman}, where it has been shown to perform well in the setting of a discretisation of Equation \eqref{eq: dX_Amari}.
\end{itemize}

\vspace{\baselineskip}

\paragraph{\bfseries Results Experiment 1} 
Figure \ref{fig: filtering_errors_A} shows relative errors $|\hat{X}_{\ti} - X_{\ti}|/|X_{\ti}|$ of the filtering estimates $\hat{X}_{\ti}$ returned by the three respective methods at observation times $t_1, \ldots t_n$. In Figure \ref{fig: filtering_errors_A}(a), all three methods are compared in the dense observation scheme with $n = 20$ observations. 
In this setting, the UKF  outperforms both particle filtering methods. We hypothesise that this is due to two reasons. Firstly, the UKF relies on the temporal discretisation of the Amari equation between observation points. In this dense observation setting, the discretisation proves to be sufficient for the UKF to perform well on the filtering task. Secondly, in the chosen parametrisation with $\delta = 0$, the Amari equation converges to an equilibrium state at which the dynamics behave roughly as 
	\begin{align}
	\label{eq: a9p1723}
		\df X_t \approx \sqrt{Q} dW_t.
	\end{align} 
	Hence, once equilibrium is reached, the nonlinear part of the dynamical system vanishes, rendering a temporal discretisation more forgiving. 
	As can be observed in Figure \ref{fig: filtering_errors_A}(d), with increasing time lag between observations, the performance of the UKF worsens drastically to the point of diverging in the case of $n = 5$.  
\vspace{\baselineskip}

A comparison of the two particle filter methods in Figure \ref{fig: filtering_errors_A}(a) shows that GPF-II with proposals based on the guiding function in \eqref{eq: Gtld} slightly outperforms our proposed GPF-I in the dense observation setting. This may be explained yet again by the models convergence to an equilibrium state. Since the guiding function defined in \eqref{eq: Gtld} is derived based on the choice $\df Z_t = \sqrt{Q} \df W_t$ for the auxiliary process $Z$, once equilibrium is reached and $X$ behaves following \eqref{eq: a9p1723}, such proposals are almost locally optimal. Hence, GPF-II starts outperforming GPF-I around the observation times just before $t = 8.0$ when an equilibrium state is `fully' reached the first time. On the other hand, when observations are sparse, the quality of the proposals for the GPF-I at states away from equilibrium lead to a better performance in the cases of $n = 10$ and $n = 5$. Interestingly, as can be seen in Figure \ref{fig: filtering_errors_A}(b), the performance of GPF-I in this experiment increases as observations become more sparse. This coincides with findings in \cite{Llopis2018Particle} for GPF-II in the context of a different dynamical system. In our experiment, this increase in performance is drastic enough for GPF-I in the $n=5$ case to be the best performing method over all observation schemes.

Figure \ref{fig: filtering_heatmaps_A} shows heatmaps of the filtered path reconstructions returned by the guided particle filters as well as the corresponding error heatmaps. Note that we are only interested in the filtering distribution and hence do not resample the whole path history at each filtering step. This results in the discontinuity of the paths at observation times seen in Figure \ref{fig: filtering_heatmaps_A}.

It can clearly be seen that the filtering task at hand reduces to estimating the right equilibrium state early on. Once such a steady state is estimated, the particle filters struggle to correct their estimates. Only GPF-I with $n=5$ seems to have sufficient time for the guided proposals to evolve into the correct steady state.

\paragraph{\bfseries Results Experiment 2} 
For the second filtering experiment, the comparison between relative filtering errors of the three methods can be found in Figure \ref{fig: filtering_errors_B}. In this parametrisation with $\delta = 0.5$, the system does not converge to an equilibrium; instead, it exhibits traveling wave behaviour. Expectedly, the proposed GPF-I outperforms GPF-II over all three observation schemes. We also see the UKF performing worse than the particle filter methods, even in the dense observation scheme. Moreover, the large time steps between observations plus the nonlinear dynamics of the traveling waves states result in significantly more divergent behaviour in the case of $n=5$ observations. On the other hand, GPF-II performs only slightly worse as observation numbers decrease, whereas GPF-I remains robust to the sparsity of observations. 

This is confirmed by the heatmaps of the reconstructed paths in \ref{fig: filtering_heatmaps_B}. In the dense observation scheme, performance of the filtering methods is similar. However, for $n=5$, only the better informed proposals of GPF-I estimate the correct traveling waves at the first observation time $t_1 = 4$. Still, in contrast to the first experiment, both particle filters manage to overcome poor estimates instead of getting stuck in misestimated equilibria. 

\subsection{Smoothing and parameter estimation}

\paragraph{\bfseries Model setup and data}
The smoothing and parameter estimation experiments are run on the same data generating process and observation scheme as depicted in Figure \ref{fig: obs_forward_filtering_B}.
We assume the parameters $x_0, B, \sigma_0, \rho_0, \eta_0$ as defined in \eqref{eq: params} to be known.
On the other hand, the parameters $\eta, \zeta, A$ and $\delta$ are treated as unknown and therefore to be estimated. 

\paragraph{\bfseries Method}
We apply Algorithm \ref{alg: Gibbs_sampler2} to the given dataset with $N = 15 \, 000$ MCMC iterations and pCN step size $\beta = 0.1$.
As priors we choose 
\begin{align*}
	\eta \sim \Unif([0,15]), \, \zeta \sim \Unif([0,3]),\,  A \sim \Unif([0,8]), \, \delta \sim \Unif([0,1]).
\end{align*}

The parameters $\{\eta, \zeta, A, \delta \}$ are updated individually in a Gibbs sampling manner. Proposals for each parameter are chosen via an adaptive scaling Metropolis scheme (see e.g. \cite{Andrieu2008Tutorial}) as follows.
At iteration $j$, given a step size $S^{\eta}_j$, say we propose $\eta_{j+1}$ via 
\begin{align*}
	\eta_{j+1} = \eta_j + \exp(S^{\eta}_j) V
\end{align*}
with $V \sim \calN(0,1)$. Setting $\theta_{j+1} = [\eta_{j+1}, \zeta_{j}, A_{j}, \delta_{j}]$, we accept/reject $\theta_{j+1}$ with Metropolis acceptance probability 
\begin{align*}
	M_{j+1}= \min\left(1, \dfrac{\pi(\theta_{j+1})  g_{\theta_{j+1}}(0,x_0)  \Psi_{\theta_{j+1}}(\Gamma_{\theta_{j+1}}(x_0,V))}{\pi(\theta_j)   g_{\theta_j}(0,x_0) \Psi_{\theta_j}(\Gamma_{\theta_j}(x_0,V))} \right).
\end{align*}
Subsequently, the step size $S^{\eta}_j$ is updated via
\begin{align*}
	S^{\eta}_{j+1} = S^{\eta}_j + r_j (M_{j+1} - \alpha^*)
\end{align*}
where $r_j$ is a scaling function and $\alpha^*$ a target acceptance rate. 
The remaining parameters $\zeta_j, A_j, \delta_j$ are updated equivalently in the standard Gibbs sampling fashion, each with their own adaptive step sizes $S^{\zeta}_j, S^A_j, S^{\eta}_j$.
In our implementation, we set $r_j = j^{-2/3}$ and $\alpha^* = 0.234$. The step sizes are initiated at $S_0 = 1$ for all parameters.
To stress-test Algorithm \ref{alg: Gibbs_sampler2}, we run the experiment twice, initialising each parameter at either end of its prior interval. 

\paragraph{\bfseries Results}
Figure \ref{fig: smoothing_trace_plots} shows the trace-plots of the MCMC outputs returned by Algorithm \ref{alg: Gibbs_sampler2} for the two different initialisations. Given the initialisations on either end of the prior intervals, the chains appear convergent remarkably fast. The true parameters fall within the range of one standard deviation of the posterior mean for each individual parameter. The exact posterior mean and standard deviation, computed after a burn-in period of $5000$ iterations, are given in Table \ref{tab: param_estimation}. 

An estimate of the smoothed path, given by the mean sample path of the final $1000$ MCMC iterations, is shown in Figure \ref{fig: smoothing_heatmaps}.  It is noteworthy that the quality of the path reconstruction matches that of the filtering experiments, even in this setting, in which half of the model parameters are assumed to be unknown. Errors in the reconstruction also appear to be mostly due to local noise and the slight misestimation of parameters, rather than a mismatch in the structural patterns of the Amari equation.
Additionally, we plot a spatially localised path of every $100$th of the first $5000$ sample paths of the MCMC output. Each path is localised in space by applying the observation operator $L$ and taking the first and eighth component of the resulting vector. These correspond to the first and eighth measurement domains with centres at $\xi_1 \approx - 28.9$ and $\xi_8 = 0.0$ plotted on the right-hand side of Figure \ref{fig: obs_forward_filtering_B}. The localised paths as well as the true localised signal and measurements thereof are depicted in Figure \ref{fig: smoothing_localised}.
 The figures show a clear estimation quality of the posterior samples, with deviations to the true signal being stronger in areas where measurements deviate as well.

\begin{table}[]\begin{tabular}{l|l|l|l|l|l|}
\cline{2-6}
                               & True Value & \begin{tabular}[c]{@{}l@{}}Exp. 1 - \\ Mean (Std)\end{tabular} & \begin{tabular}[c]{@{}l@{}}Exp. 1 - \\ Acceptance Ratio\end{tabular} & \begin{tabular}[c]{@{}l@{}}Exp. 2 - \\ Mean (Std)\end{tabular} & \begin{tabular}[c]{@{}l@{}}Exp. 2 -\\ Acceptance Ratio\end{tabular} \\ \hline
\multicolumn{1}{|l|}{$\eta$}   & \textbf{10.0}      & \textbf{10.95} (1.72)                                                   & 0.239                                                                & \textbf{11.26} (2.2)                                                    & 0.237                                                               \\ \hline
\multicolumn{1}{|l|}{$\zeta$}  & \textbf{0.5}        & \textbf{0.31} (0.23)                                                    & 0.236                                                                & \textbf{0.33} (0.23)                                                    & 0.234                                                               \\ \hline
\multicolumn{1}{|l|}{$A$}      & \textbf{4.0}        & \textbf{3.89} (0.17)                                                    & 0.223                                                                & \textbf{3.89} (0.24)                                                    & 0.226                                                               \\ \hline
\multicolumn{1}{|l|}{$\delta$} & \textbf{0.5}        & \textbf{0.5} (0.006)                                                    & 0.211                                                                & \textbf{0.5} (0.007)                                                    & 0.208                                                               \\ \hline
\end{tabular}
\caption{Posterior means and standard deviations as well as acceptance ratios returned for the model parameters by Algorithm \ref{alg: Gibbs_sampler2}.}
\label{tab: param_estimation}
\end{table}

\begin{figure}[htbp]
    \centering
    \subfigure{%
        \includegraphics[width=0.45\textwidth]{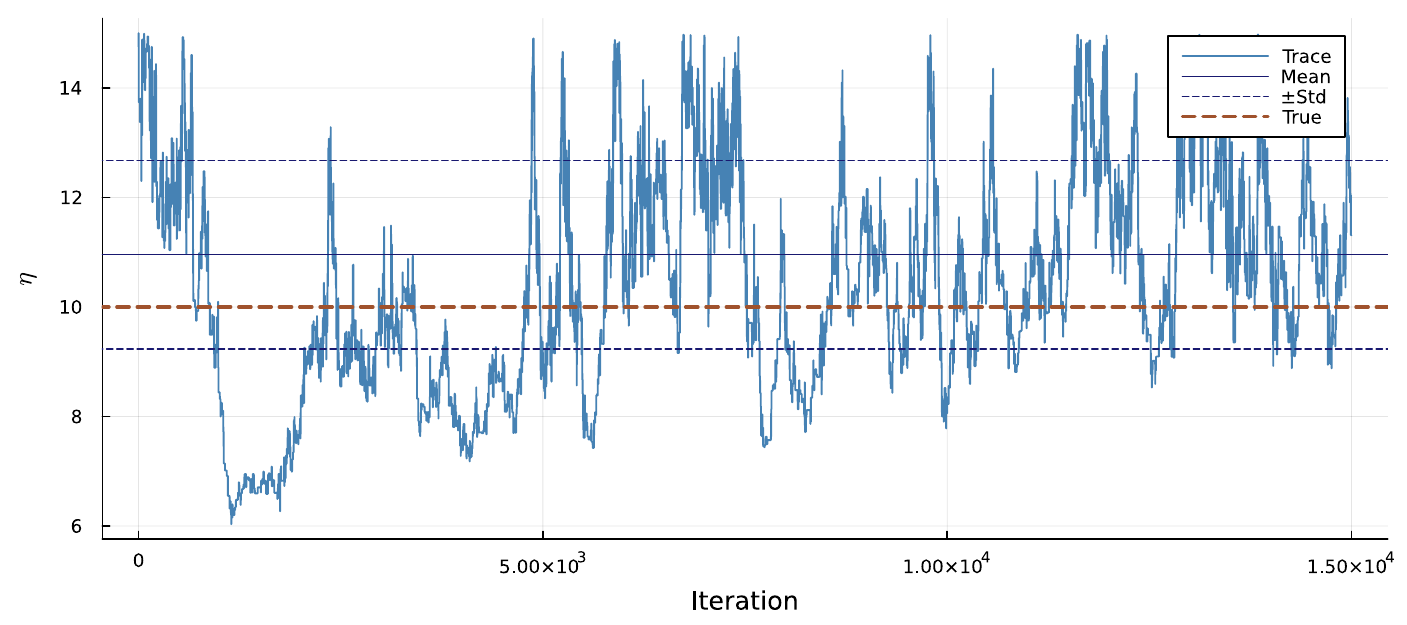}
    }
    \hspace{0.25cm}
    \subfigure{%
        \includegraphics[width=0.45\textwidth]{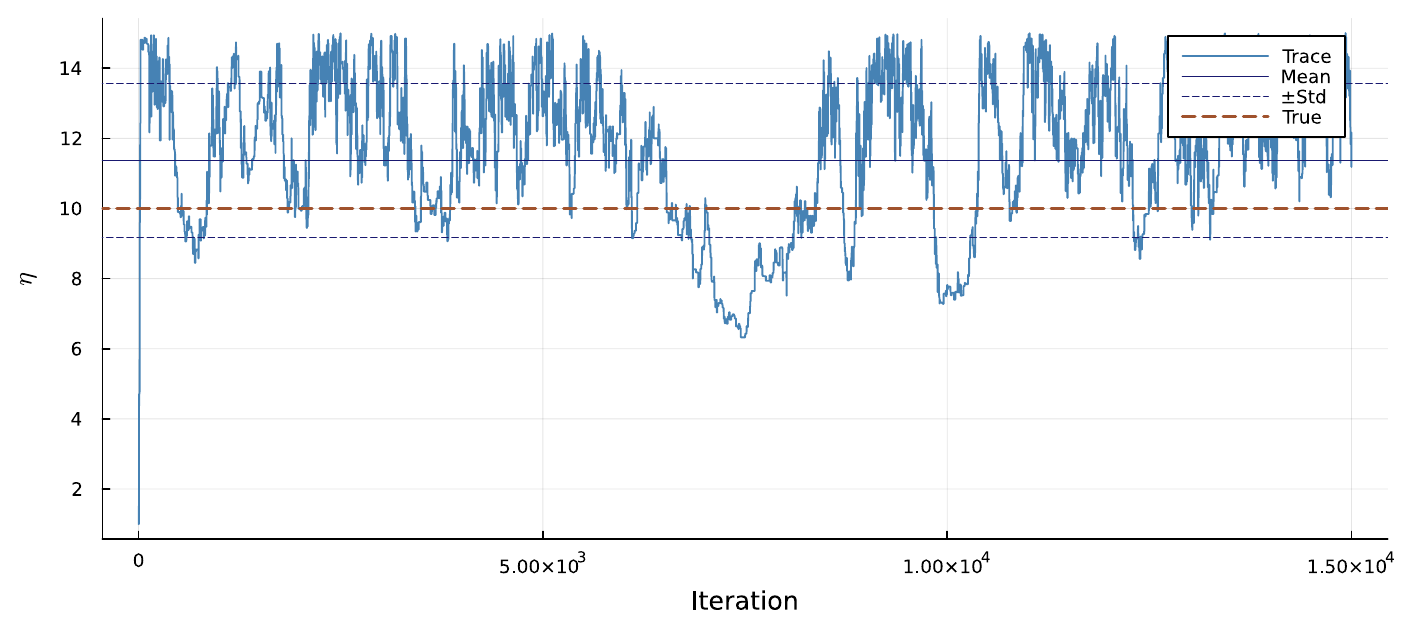}
    }
    \vspace{0.25cm}
        \subfigure{%
        \includegraphics[width=0.45\textwidth]{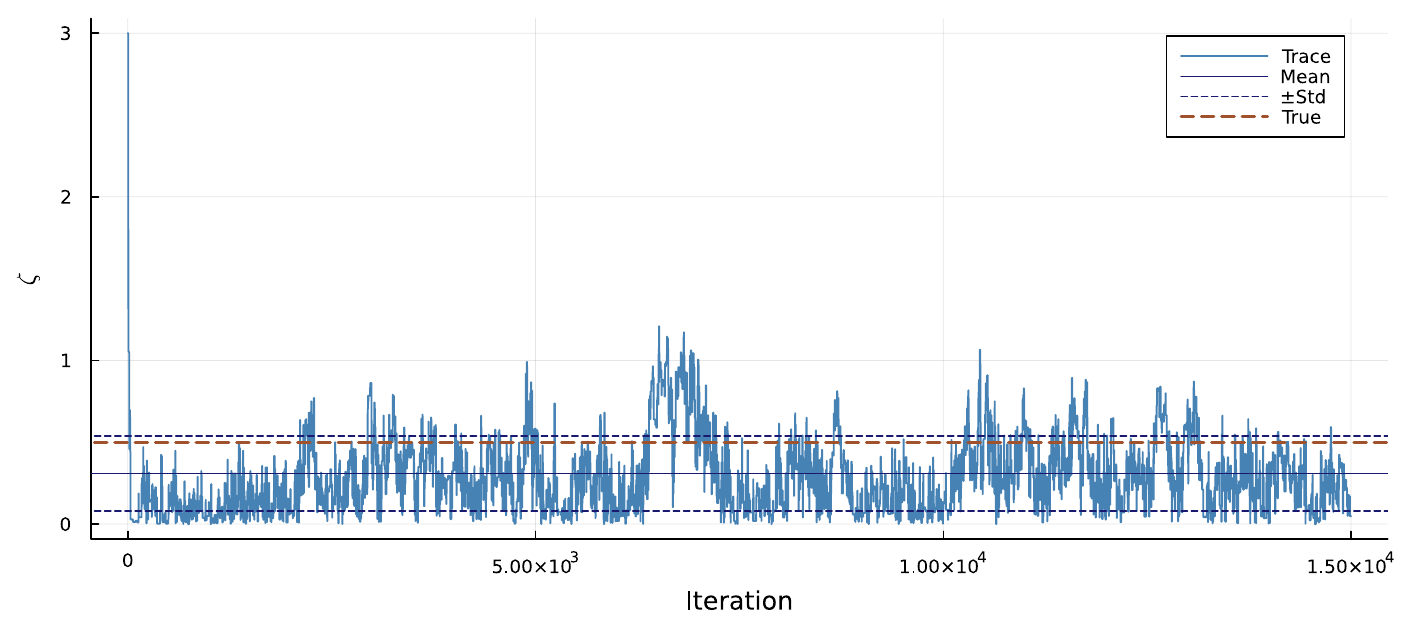}
    }
    \hspace{0.25cm}
    \subfigure{%
        \includegraphics[width=0.45\textwidth]{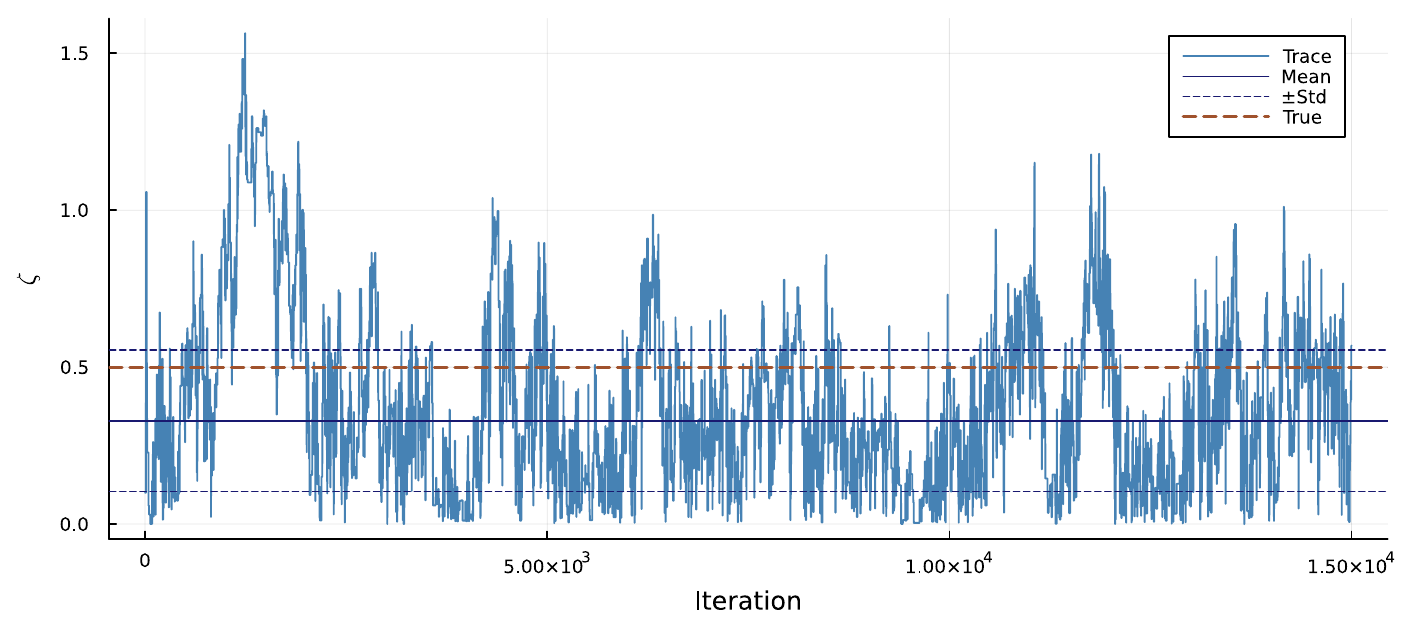}
    }
    \vspace{0.25cm}
        \subfigure{%
        \includegraphics[width=0.45\textwidth]{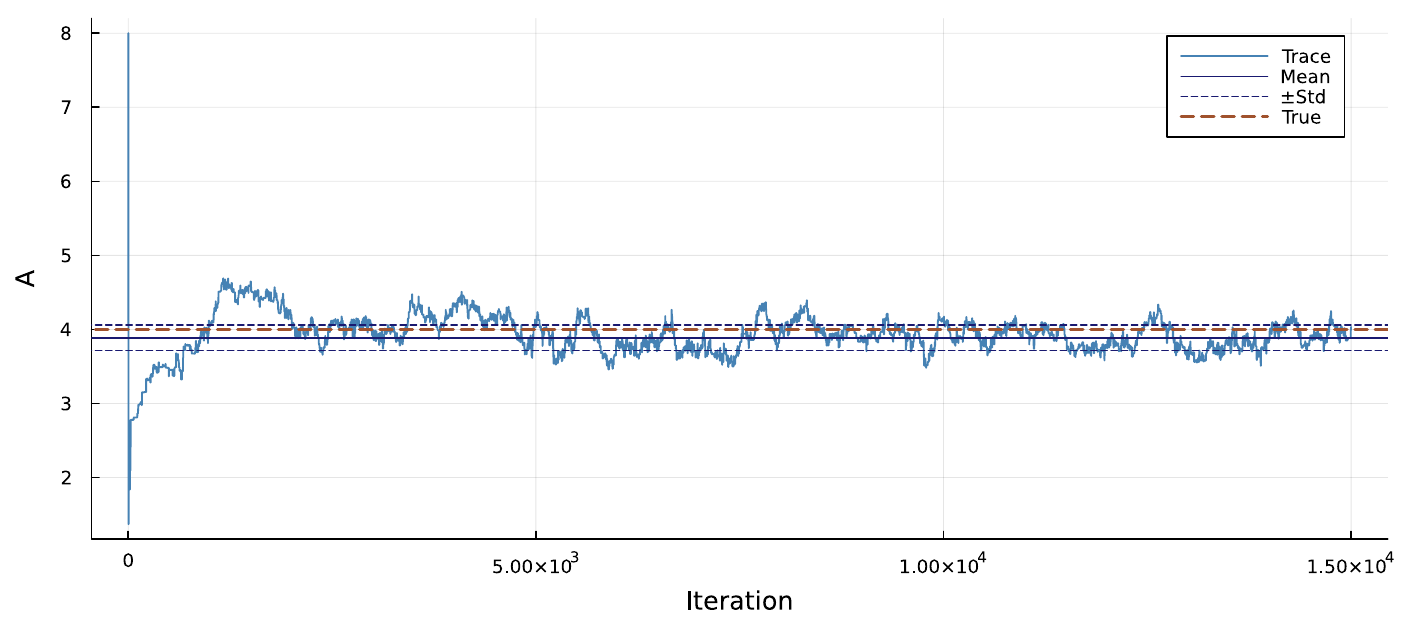}
    }
    \hspace{0.25cm}
    \subfigure{%
        \includegraphics[width=0.45\textwidth]{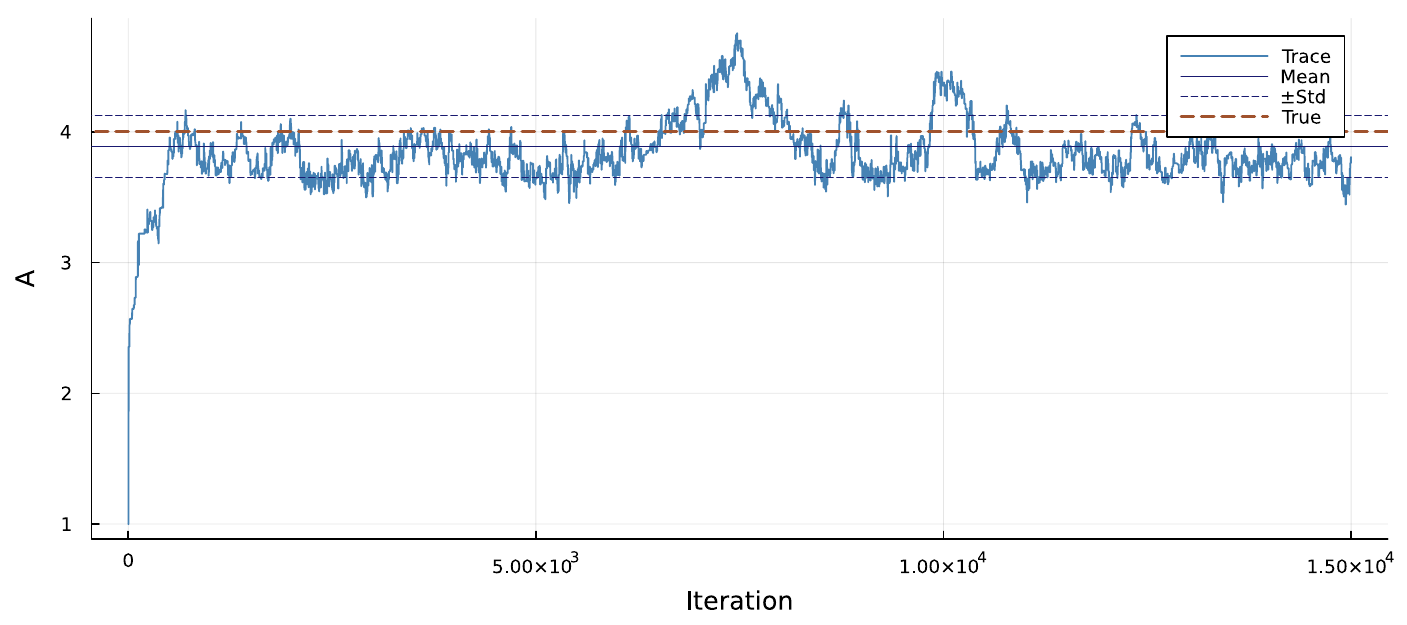}
    }
    \vspace{0.25cm}
        \subfigure{%
        \includegraphics[width=0.45\textwidth]{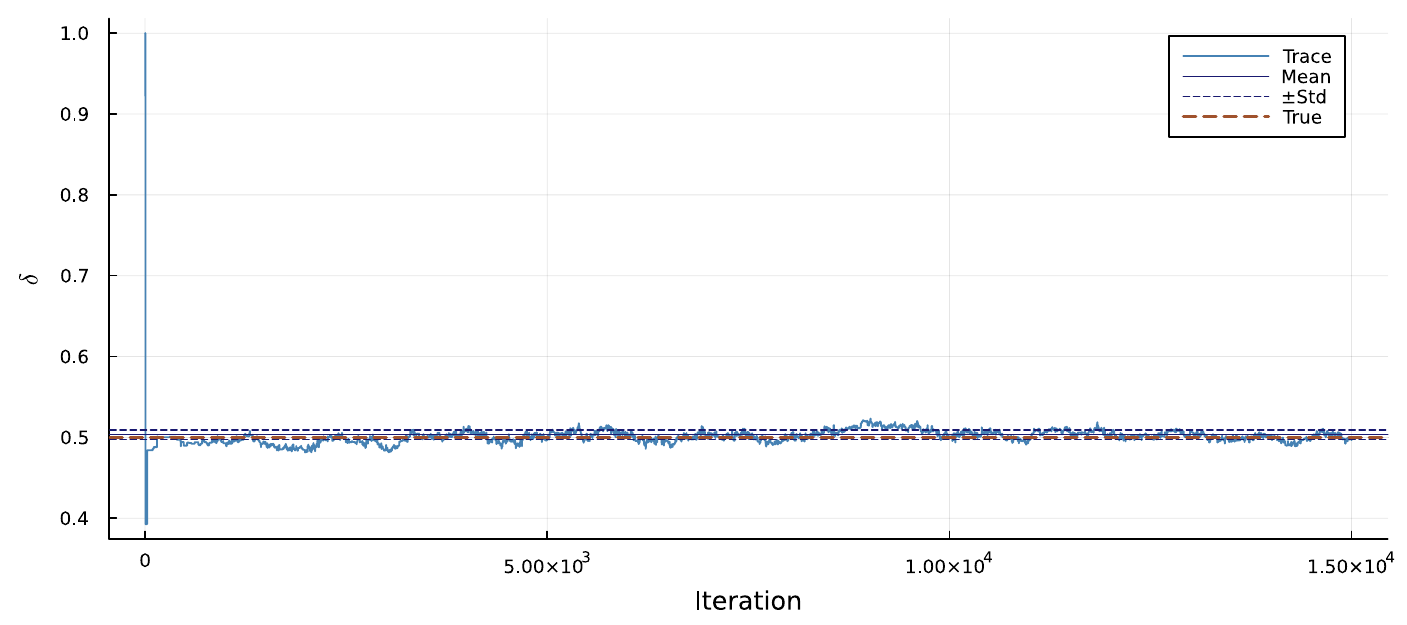}
    }
    \hspace{0.25cm}
    \subfigure{%
        \includegraphics[width=0.45\textwidth]{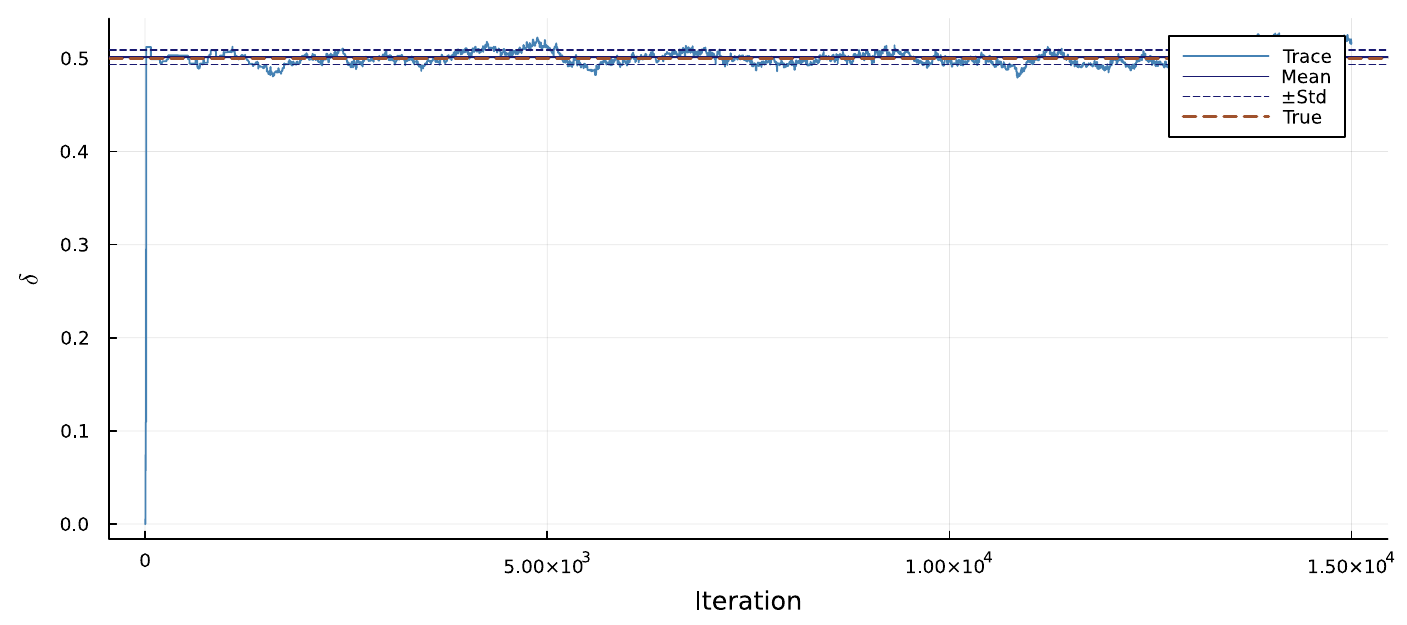}
    }
    \vspace{0.25cm}
	 \caption{Trace-plots of the MCMC outputs returned by Algorithm \ref{alg: Gibbs_sampler2} for various model parameters. The rows show, in order, the trace-plots of $\eta, \zeta, A$ and $\delta$. The dark blue solid line represents the sample mean after a burn-in period of $5000$ samples, whereas the dashed lines represent the interval spanning one empirical standard deviation above and below the sample mean. The orange lines represent the true value. Different columns relate to a different initialisation of the algorithm. }
    \label{fig: smoothing_trace_plots}
\end{figure}

\begin{figure}[htbp]
    \centering
    \subfigure{%
        \includegraphics[width=0.45\textwidth]{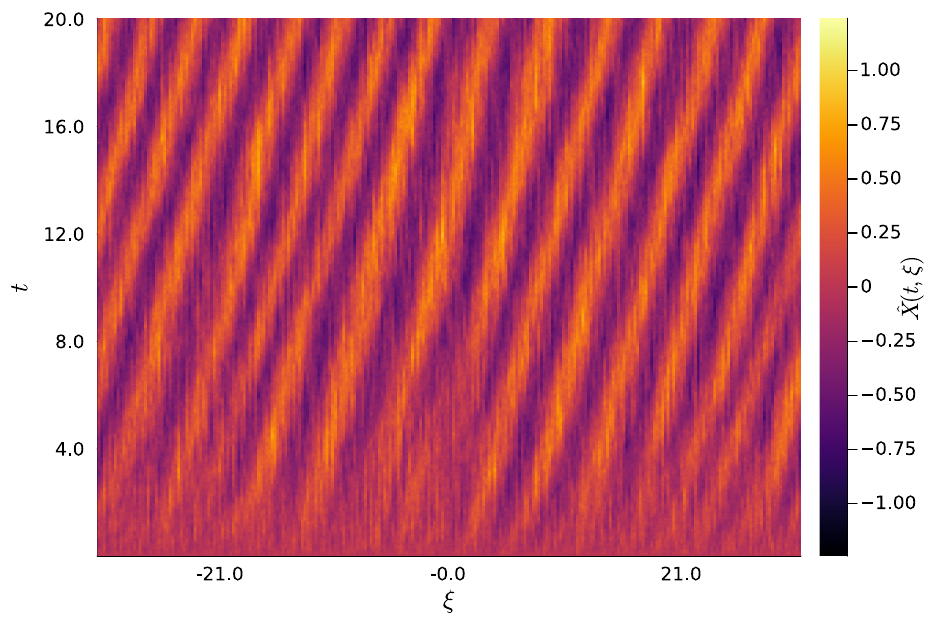}
    }
    \hspace{0.25cm}
    \subfigure{%
        \includegraphics[width=0.45\textwidth]{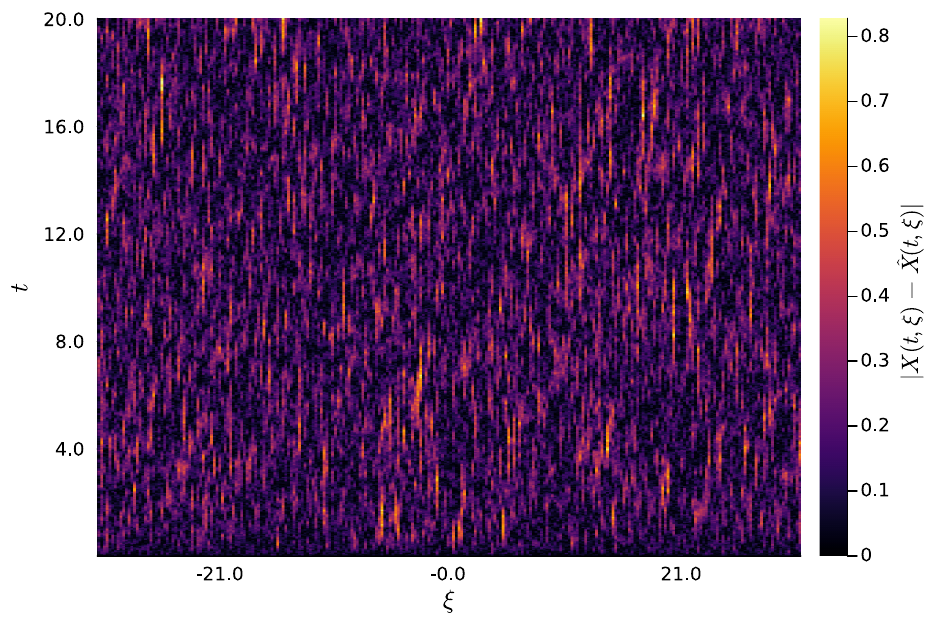}
    }
    \caption{Heatmap of the smoothing estimate of $X$. Left: Path estimate $\hat{X}$ given by the mean of the final $1000$ MCMC iterations returned by Algorithm \ref{alg: Gibbs_sampler2}. Right: Absolute error with respect to the true signal $X$.}
    \label{fig: smoothing_heatmaps}
\end{figure}

\begin{figure}[b]
    \centering
    \includegraphics[width=0.9\textwidth]{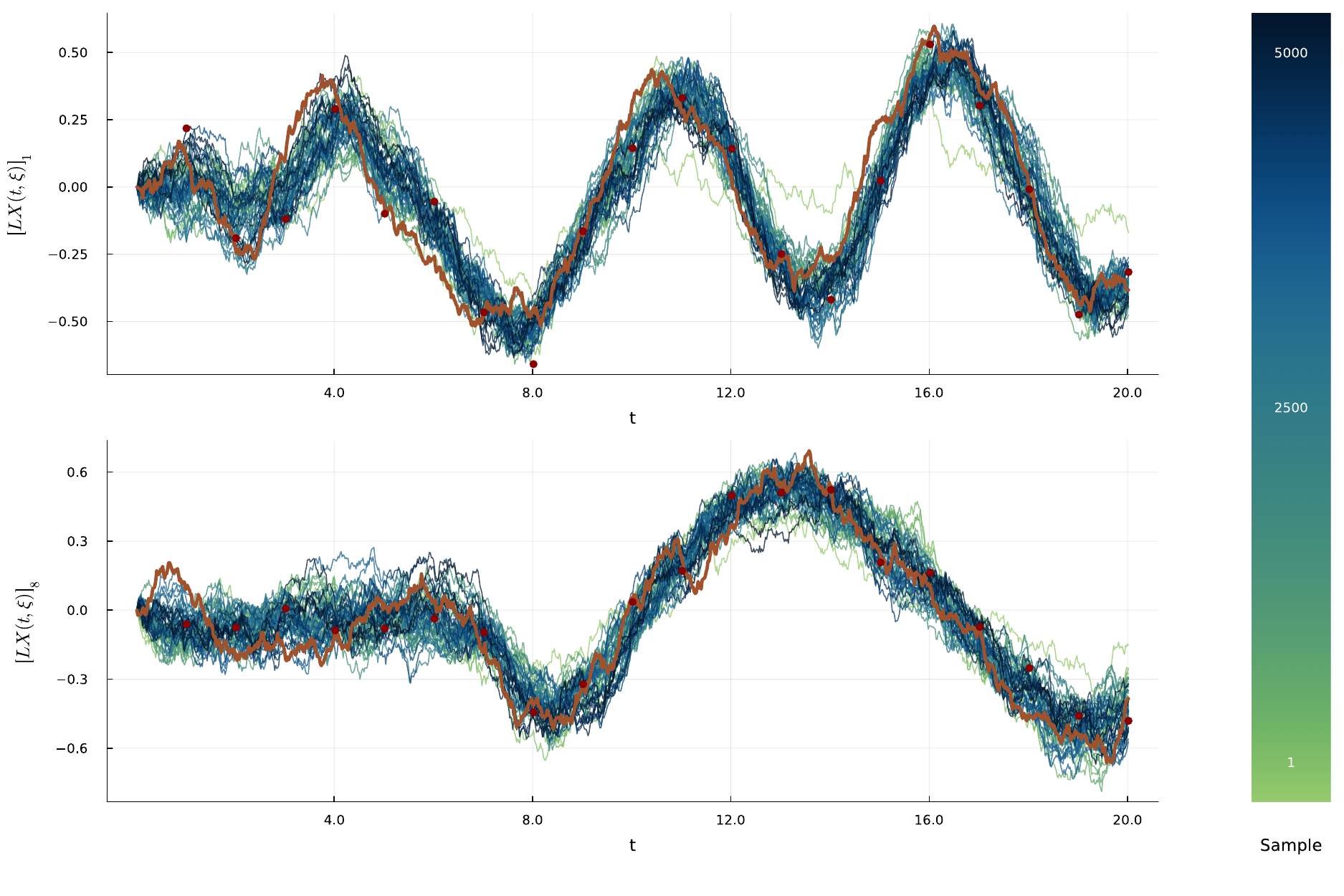}
    \caption{Spatially localised sample paths of the smoothing distribution of $X$. Spatial localisations are computed by applying the observation operator $L$. Green represents early samples in the Markov chain, whereas blue represents later samples. The orange line shows the true localised signal, with red dots symbolising the measurements thereof. Top: Local average over the first measurement domain around $\xi_1 \approx -28.9$. Bottom: Local average over the eighth measurement domain around $\xi_8 = 0.0$.}
    \label{fig: smoothing_localised}
\end{figure}

\bigskip
\noindent\textbf{Code availability:} \\
The source code for the experiments in Section \ref{sec: amari} is available at

\href{https://github.com/tpiepersethmacher/NeuralFieldsFiltering.jl}{\texttt{https://github.com/tpiepersethmacher/NeuralFieldsFiltering.jl}}.

\newpage 

\appendix  

\section{}
\label{app: A}
\subsection{Proofs of Section \ref{sec: smoothing_measure}}
\begin{proof}[Proof of Proposition \ref{prop: h_properties}]
    Both equalities follow from rewriting the definition of $h$ in \eqref{eq: def_h}, whereas the continuity claims are a consequence of the dominated convergence theorem and almost sure continuity of $X$.
\end{proof}

\begin{proof}[Proof of Proposition \ref{prop: Eht_martingale}]
We first show that the function $h$ defined in \eqref{eq: def_h} is space-time harmonic on any time interval $(\tim,\ti]$, i.e. 
\begin{align}
\label{eq: hspacetimeharmonic}
    \E[ h(t, X_{t}) \mid X_s = x] &= h(s,x)
\end{align}
for any $x \in H$ and $s,t \in (\tim,\ti]$ such that $ s < t.$

Indeed, from plugging in \eqref{eq: def_h} and using the Chapman-Kolmogorov equation, it follows that 
\begin{align*}
    \E[ h(t, X_{t}) \mid X_s = x] &= \int h(t,x_t) \mu_{s,t}(x, \df x_t) \\
    &= \int k(x_{\ti}, y_i) \left( \prod_{j=i}^{n-1} k(x_{\tjp},y_{j+1}) \, \mu_{\tj,\tjp}(x_{\tj}, \df x_{\tjp}) \right) \mu_{s,t}(x, \df x_t) \mu_{t, \ti}(x_t, \df x_{\ti}) \\
    &= \int k(x_{\ti}, y_i) \left( \prod_{j=i}^{n-1} k(x_{\tjp},y_{j+1}) \, \mu_{\tj,\tjp}(x_{\tj}, \df x_{\tjp}) \right) \mu_{s,\ti}(x,\df x_{\ti}) \\
    &= h(s,x).
\end{align*}
This gives \eqref{eq: hspacetimeharmonic}.
To show that $E^h$ is a $\P$-martingale, first consider the case that $\tim < s < t \leq t_i$.
Then, by the Markov property of $X$, the definition of $E^h$ and the preceding display we have
\begin{align*}
    \E[ E^h_t \mid \calF_s] &= \frac{1}{C^h} \E\left[\left( \prod_{j=1}^{i-1} k(X_{t_j},y_j) \right) h(t,X_t) \mid \calF_s\right] \\
    &= \frac{1}{C^h} \prod_{j=1}^{i-1} k(X_{t_j},y_j)  \E \left[ h(t,X_t) \mid X_s \right] \\
    &= \frac{1}{C^h} \prod_{j=1}^{i-1} k(X_{t_j},y_j) h(s,X_s) = E^h_s.
\end{align*}

On the other hand, if $t_{i-2} < s \leq \tim < t \leq t_i$, then it holds that
\begin{align*}
    \E[ E^h_t \mid \calF_s] &= \frac{1}{C^h} \E\left[\left( \prod_{j=1}^{i-1} k(X_{t_j},y_j) \right) h(t,X_t) \mid \calF_s\right] \\ 
    &= \frac{1}{C^h} \left(\prod_{j=1}^{i-2} k(X_{t_j},y_j) \right) \E\left[  k(X_{\tim},y_{i-1})  h(t,X_t) \mid X_s\right] \\
    &= \frac{1}{C^h} \left(\prod_{j=1}^{i-2} k(X_{t_j},y_j) \right) \int k(x_{\tim},y_{i-1}) h(t,x_t) \,\mu_{s,\tim}(X_s, \df x_{\tim}) \mu_{\tim, t}(x_{\tim},\df x_t) \\
    &= \frac{1}{C^h} \left(\prod_{j=1}^{i-2} k(X_{t_j},y_j) \right) h(s,X_s) = E^h_s.
\end{align*}
Here, we use the Markov property of $X$ as well as the definition of $E^h_t$ in the first step and the $\calF_s$-measurability of $X_{t_j}, j \leq i-2$ in the second step. The last step follows by plugging in the definition of $h(t,x_t)$ and integrating out the latent variable $X_t$. The general case $s < t$ follows by the same arguments.

\end{proof}

\begin{proof}[Proof of Theorem \ref{thm: dPh}]
We begin by showing that $\P^h$ satisfies \eqref{eq: smoothingEh}. For the proof of this first claim, let us write $\P^{h,y}$ and $E^{h,y}$ to emphasise the dependence of the measure $\P^h$ and process $E^h$ on the observation $Y = y$.
To prove \eqref{eq: smoothingEh} it suffices to show that 
\begin{align}
\label{eq: aap918273}
    \E \left[ \E^{h,Y} \left[ \varphi(X_t) \right] \mathbbm{1}_{\{Y \in A\}} \right] &= \E[\varphi(X_t) \mathbbm{1}_{\{Y \in A\}}]
\end{align}
for any $A \in (\calB(\R^{k}))^n$ and bounded, measurable function $\varphi$ and $t \in [0,T].$
Without loss of generality let $t \in (\tim, \ti]$ for some $i \in \{1,\dots,n\}$.
Then indeed we have that
\begin{align*}
    \E \left[ \E^{h,Y} \left[ \varphi(X_t) \right] \mathbbm{1}_{\{Y \in A\}} \right] &= \int_A \E^{h,y} \left[ \varphi(X_t)  \right] p(y) \df y \\
    &= \int_A \E \left[ \varphi(X_t)  \frac{1}{C^{h,y}} \left( \prod_{j=1}^{i-1} k(X_{t_j},y_j) \right) h(t,X_t) \right] p(y) \, \df y \\
    &= \int_A \E \left[ \varphi(X_t)  \left( \prod_{j=1}^{i-1} k(X_{t_j},y_j) \right) h(t,X_t) \right]  \, \df y \\
    &= \E \left[ \varphi(X_t)  \int_A  \left( \prod_{j=1}^{i-1} k(X_{t_j},y_j) \right) h(t,X_t) \, \df y\right] \\
    &= \E \left[ \varphi(X_t)  \E \left[\mathbbm{1}_{\{Y \in A\}} \mid X_{t_1}, \dots, X_{\tim}, X_t \right] \right] \\
    &= \E \left[ \varphi(X_t) \mathbbm{1}_{\{Y \in A\}} \right].
\end{align*}
Here, we used in the second step the definition of $\P^{h,y}$ and the martingale property of $E^{h,y}_t$ and in the third step the equality $C^{h,y} = p(y)$. Moreover, we used Fubini's theorem in the fourth equality as well as the identity of $p(y \mid x_{t_1}, \dots, x_{t}) = \left( \prod_{j=1}^{i-1} k(x_{t_j},y_j) \right) h(t,x_t)$ as derived in Remark \ref{rem: prodkh_likelihood} in the second to last step.

We continue to prove the second claim of Theorem \ref{thm: dPh}. 
To do so, we show that the martingale $E^h_t$ defined in \eqref{eq: def_Eht} is given by 
\begin{align}
\label{eq: a1p092873}
E^{h}_t = \dfrac{h(0,X_0)}{C^h} \calE(M^h)_t, \quad t \in [0,T],
\end{align}
where $\calE(M^h)$ denotes the Doléans-Dade exponential of 
\begin{align*}
	M^h_t := \int_0^t \langle Q^{\frac12} \Df_x \log h(s,X_s), \df W_s \rangle, \quad t \in [0,T].
\end{align*}
The claim then follows by an application of the Girsanov theorem. 

Let first $t \in [0,t_1]$ and denote by $K$ the infinitesimal space-time generator of the Markov process $X$. 
Following \eqref{eq: hspacetimeharmonic}, it holds that $K h = 0$. 
From this and the Fréchet differentiability of $h$ in $x$ it follows with Lemma 3.3 in \cite{Piepersethmacher2025Class} that 
\begin{align*}
    h(t,X_t) = h(0,X_0) + \int_0^t \langle Q^{\frac12} \Df_x h(s,X_s), \df W_s \rangle, \quad t \in [0, t_1).
\end{align*}
Hence the process $\bar{h}(t,X_t) := h(t,X_t)/h(0,X_0)$ is the unique solution to 
\begin{align*}
\begin{cases}
        \df \bar{h}(t,X_t) &= \bar{h}(t,X_t) \df M^{h}_t \\
    \bar{h}(0,X_0) &= 1.
\end{cases}
\end{align*}
From this and the left-continuity of $h$ it follows that
\begin{align*}
    h(t,X_t) = h(0,X_0) \calE(M^{h})_t, \quad t \in [0,t_1].
\end{align*}
This, jointly with $E^h_t = h(t,X_t)/C^h, ~t \in [0,t_1],$ shows \eqref{eq: a1p092873} on $[0,t_1]$.
Now, let $t \in (\tim, \ti], i \geq 2,$ and assume \eqref{eq: a1p092873} holds for all $[0,\tim].$

Again, by the definition of $E^{h}$, Equation \eqref{eq: hspacetimeharmonic} and Lemma 3.3 in \cite{Piepersethmacher2025Class} we have that
\begin{align*}
    \df E^{h}_t &=  \frac{1}{C^{h}} \left( \prod_{j=1}^{i-1} k(X_{t_j},y_j) \right) \df h(t,X_t) \\
    &= \frac{1}{C^{h}} \left( \prod_{j=1}^{i-1} k(X_{t_j},y_i) \right) \langle Q^{\frac12} \Df_x h(t,X_t), \df W_t \rangle \\
    &= E^{h}_t \df M^{h}_t, \quad t \in (\tim,\ti).
\end{align*}

From this it is apparent that $\bar{E}^{h}_t := E^{h}_t/E^{h}_{\tim}$ solves 
\begin{align}
    \begin{cases}
        \df \bar{E}^{h}_t &= \bar{E}^{h}_t \df M^{h}_t, \quad t \in (\tim,\ti),  \\
    \df \bar{E}^{h}_{\tim} &= 1.
\end{cases}
\end{align}
It follows that 
\begin{align*}
    \bar{E}^h_t = \exp \left(\int_{\tim}^t \langle Q^{\frac12} \Df_x \log h(s,X_s), \df W_s \rangle  - \frac{1}{2} \int_{\tim}^t |\langle Q^{\frac12} \Df_x \log h(s,X_s)|^2 \, \df s \right)
\end{align*}
for all $t \in (\tim,\ti]$ and hence
\begin{align*}
E^h_t &= E^h_{\tim} \exp \left(\int_{\tim}^t \langle Q^{\frac12} \Df_x \log h(s,X_s), \df W_s \rangle  - \frac{1}{2} \int_{\tim}^t |\langle Q^{\frac12} \Df_x \log h(s,X_s)|^2 \, \df s \right) \\
&= \dfrac{h(0,X_0)}{C^h} \calE(M^h)_{\tim} \exp \left(\int_{\tim}^t \langle Q^{\frac12} \Df_x \log h(s,X_s), \df W_s \rangle  - \frac{1}{2} \int_{\tim}^t |\langle Q^{\frac12} \Df_x \log h(s,X_s)|^2 \, \df s \right)\\
&= \dfrac{h(0,X_0)}{C^h} \calE(M^h)_t.
\end{align*}
The claim in \eqref{eq: a1p092873} then follows inductively, thereby finishing the proof.

\end{proof}

\subsection{Proofs of Section \ref{sec: guided_measure}}

\begin{proof}[Proof of Theorem \ref{thm: g_LtRtalphat}]

\textit{Step One.} We begin deriving $L_t, R_t$ and $\alpha_t$ on the interval $(\tnm,\tn]$.
By definition of $g$ in \eqref{eq: def_g}, we have $g(\tn,x) = k(x_n,y_n) = f(y_n ; L x, \Sigma)$.
Moreover, from the backwards recursion \eqref{eq: g_BIF2} it follows that 
\begin{align*}
        g(t,x) &= \int f(y_n ; L x_{\tn}, \Sigma)\, \nu_{t,\tn}(x, \df x_{\tn}) \\
        &= \int f(y_n ; \xi, \Sigma) f\left(\xi; L S_{t_n-t}x + L \int_t^{t_n} S_{t_n-s} a_s \, \df s, L Q_{t_n-t} L^*\right)\, \df \xi \\
        &= \int f(y_n ; \xi, \Sigma) f\left(\xi; L_t x + \alpha_t, L Q_{t_n-t} L^*\right)\, \df \xi.
\end{align*}
Here we use in the second equality that $\xi = Lx_{\tn}$ under $\nu_{t,\tn}(x, \df x_{\tn})$ is a Gaussian measure on $\R^m$ with mean 
\begin{align*}
    L S_{t_n-t}x + L \int_t^{t_n} S_{t_n-s} a_s \, \df s = L_t x + \alpha_t
\end{align*}
and covariance matrix $L Q_{t_n-t} L^*$. Now, integrating out the latent Gaussian variable $\xi$ gives
\begin{align*}
    g(t,x) = f(y_n; L_t x + \alpha_t, \Sigma + L Q_{t_n-t} L^*) = f(y_n; L_t x + \alpha_t, R_t),
\end{align*}
with $L_t$, $\alpha_t$ and $R_t$ as defined in \eqref{eq: def_Lt} and \eqref{eq: def_Rt_alphat}.

\textit{Step Two.} 
Consider $g(\tnm,x)$ evaluated at the observation time $\tnm$. From the second equality in the backwards recursion \eqref{eq: g_BIF2} it follows that
\begin{align*}
    g(\tnm,x) = k(y_{n-1},x) \, g^+(\tnm,x) = f(y_{n-1}; L x, \Sigma) \,f(y_n; L^+_{\tnm} + \alpha^+_{\tnm}, R^+_{\tnm}).
\end{align*}
Expanding the Gaussian densities on the right-hand side and simple algebra gives that
\begin{align*}
    g(\tnm,x) &= f(y_{n-1}^+; L_{\tnm} x + \alpha_{\tnm}, R_{\tnm}) 
\end{align*}
with 
\begin{align*}
    L_{\tnm} &= \begin{bmatrix}  L  \\  L^+_{\tnm}\end{bmatrix}, \,
    R_{\tnm} = \begin{bmatrix} 
\Sigma & 0 \\
0 & R^+_{\tnm} \\ 
\end{bmatrix}, \,
\alpha_{\tnm} = \begin{bmatrix}  0 \\ \alpha^+_{\tnm} \end{bmatrix}.
\end{align*}

\textit{Step Three.}
The general claim follows by repeating the backwards recursion of the first two steps to obtain $g(t,x)$ on the complete interval $[0,t_n].$
\end{proof}

\begin{proof}[Proof of Theorem \ref{thm: dPg}]
The proof of Theorem \ref{thm: dPg} follows roughly the same structure as the second part of the proof of Theorem \ref{thm: dPh}.
We show that $E^g$ as defined in Eq. \eqref{eq: def_Egt} equals the Doléans-Dade exponential
\begin{align}
\label{eq: ap912783987p123}
	E^g_t = \frac{g(0,X_0)}{C^g}\calE(M^g)_t, \quad t \in [0,T],
\end{align}
of the martingale $M^g_t := \int_0^t \langle Q^{\frac12} G(s,X_s), \df W_s \rangle$. If $\calE(M^g)_t$ is then a $\P$-martingale, Theorem \ref{thm: dPg} follows from an application of the Girsanov theorem. 

First note that, following Theorem \ref{thm: g_LtRtalphat}, $g(t,x)$ is bounded, continuous and Fréchet differentiable in $x$ on any interval $(\tim, \ti)$ with 
\begin{align*}
	\Df_x g(t,x) = g(t,x) \Df_x \log g(t,x) = g(t,x) G(t,x). 
\end{align*}
Hence, denoting by $K$ the infinitesimal space-time generator of $X$, it follows from \cite{Manca2009Fokker}, Theorem 4.1 that $K g(t,x) = \langle F(t,x), \Df_x g(t,x) \rangle$. Lemma 3.3 of \cite{Piepersethmacher2025Class} therefore gives that 
\begin{align}
\label{eq: ap918273123123}
g(t,X_t) = g(0,X_0) + \int_0^t \langle  F(s,X_s), \Df_x g(s,X_s) \rangle \df s + \int_0^t \langle Q^{\frac12} \Df_x g(s,X_s), \df W_s \rangle, \quad t \in [0,t_1). 
\end{align}
Consequently, plugging in $E^g$ as in \eqref{eq: def_Egt} and applying the integration by parts for semimartingales we have on $[0,t_1)$ that
\begin{align*}
\df E^g_t &= \dfrac{1}{C^g} \df \left( g(t,X_t)  \exp \left(- \int_0^t \langle F(s,X_s), G(s,X_s) \rangle \, \df s \right) \right) \\
&=\dfrac{1}{C^g}   \exp \left(- \int_0^t \langle F(s,X_s), G(s,X_s) \rangle \, \df s \right)\left[ \langle F(t,X_t), \Df_x g(t,X_t) \rangle \df t  + \langle Q^{\frac12} \Df_x g(t,X_t), \df W_t \rangle \right] \\
&\quad - \dfrac{1}{C^g} g(t,X_t) \exp \left(- \int_0^t \langle F(s,X_s), G(s,X_s) \rangle \, \df s \right)  \langle F(t,X_t), G(t,X_t) \rangle \, \df t \\
&= \dfrac{1}{C^g}   \exp \left(- \int_0^t \langle F(s,X_s), G(s,X_s) \rangle \, \df s \right) \langle Q^{\frac12} \Df_x g(t,X_t), \df W_t \rangle \\
&= \dfrac{g(t,X_t)}{C^g}   \exp \left(- \int_0^t \langle F(s,X_s), G(s,X_s) \rangle \, \df s \right) \langle Q^{\frac12}  G(t,X_t), \df W_t \rangle \\
&= E^g_t \df M^g_t.
\end{align*}
This shows that $\bar{E}^g_t := E^g_t / E^g_0 = \calE(M^g)_t$ and hence \eqref{eq: ap912783987p123} for all $t \in [0,t_1]$. The extension onto the complete time interval $[0,T]$ follows by the same inductive argument as in the proof of Theorem \ref{thm: dPh}.
Lastly, to show that $\calE(M^g)$ is a proper martingale, it suffices to note that $x \mapsto G(t,x)$ is Lipschitz continuous in $x$, uniformly on $[0,T]$. The martingale property of $\calE(M^g)$ then follows from \cite{Piepersethmacher2025Class}, Lemma C.1.
\end{proof}

\subsection{Proofs of Section \ref{sec: comp_efficient_guiding}}

\begin{proof}[Proof of Theorem \ref{thm: G_UtVt}]
 
We start by showing that $U_t$ is a solution to \eqref{eq: dU_t} in a mild sense, i.e. that $U_t$ satisfies
\begin{align}
\label{eq: dUt_mild_sol}
U_t = S^*_{\ti-t} U_{\ti} S_{\ti-t} - \int_t^{\ti} S^*_{s-t} U_s Q U_s S_{s-t} \, \df s, \quad t \in (\tim,\ti],
\end{align}
and $U_{\ti} = L^* \Sigma L + U^+_{\ti}$. The latter follows from plugging in $L_{\ti}$ and $R_{\ti}$ as given in Theorem \ref{thm: g_LtRtalphat} into the definition of $U_t$.
To show \eqref{eq: dUt_mild_sol}, note that $R_t$ as defined in Theorem \ref{thm: g_LtRtalphat} satisfies 
\begin{align}
\label{eq: a918o273}
    \dfrac{\df}{\df t} R_t= \dfrac{\df}{\df t} \left( L_{\ti} Q_{\ti-t} L^*_{\ti} \right)= - L_{\ti} S_{\ti-t} Q S^*_{\ti-t} L^*_{\ti} = -L_t Q L^*_t.
\end{align}
Here, we used that $\df Q_t = S_t Q S^*_t \df t$ following the definition of $Q_t$ in \eqref{eq: def_Qt}.
It follows
\begin{align}
\label{eq: a129o387}
    \dfrac{\df}{\df t} R^{-1}_t &= - R^{-1}_t \left(  \dfrac{\df}{\df t} R_t \right)R^{-1}_t = R^{-1}_t L_t Q L^*_t R^{-1}_t
\end{align}
and thus 
\begin{align}
\label{eq: ap9018273}
    R^{-1}_{t} &= R^{-1}_{\ti} - \int_t^{\ti} R^{-1}_s L_s Q L^*_s R^{-1}_s \, \df s.
\end{align}
Hence, plugging \eqref{eq: ap9018273} into \eqref{eq: def_Ut_Vt} and using that $L_t = L_{s} S_{s-t}$ for all $s \in [t, \ti]$, this gives
\begin{align*}
    L^*_t R^{-1}_t L_t &= L^*_t R^{-1}_{\ti} L_t -  L^*_t  \left( \int_t^{\ti} R^{-1}_s L_s Q L^*_s R^{-1}_s \, \df s \right) L_t \\
    &= S^*_{\ti-t} L^*_{\ti} R^{-1}_{\ti} L_{\ti} S_{\ti - t} -  \int_t^{\ti} S^*_{s-t} L^*_{s} R^{-1}_s L_s Q L^*_s R^{-1}_s L_s S_{s-t}\, \df s \\
    &= S^*_{\ti-t} U_{\ti} S_{\ti - t} - \int_t^{\ti} S^*_{s-t} U_s Q U_s S_{s-t}\, \df s.
\end{align*}
This shows \eqref{eq: dUt_mild_sol} and uniqueness of the mild solution follows from uniqueness of the infinite-dimensional forward Ricatti equation, see for example \cite{Burns2015Solutions}, Theorem 3.6 as well as \cite{Bensoussan2007Representation}, Chapter 2.2 and references within.

To show that $V_t$ satisfies \eqref{eq: dV_t} in a mild sense, first note that
\begin{align}
\label{eq: ap91782397123}
\begin{split}
    \dfrac{\df}{\df t} \left( R_t^{-1} (y_i^+ - \alpha_t) \right) &= \left( \dfrac{\df}{\df t} R_t^{-1} \right) (y_i^+ - \alpha_t) - R_t^{-1} \left( \dfrac{\df}{\df t} \alpha_t \right)  \\
    &= R^{-1}_t L_t Q L^*_t R^{-1}_t (y_i^+ - \alpha_t) + R_t^{-1} L_t a_t \\
    &= R^{-1}_t L_t Q V_t + R_t^{-1} L_t a_t, \quad t \in (\tim,\ti).
\end{split}
\end{align}
Here we used \eqref{eq: a129o387} and the definition of $\alpha_t$ in the second equality and the definition of $V_t$ in the last.
This gives 
\begin{align*}
    R_t^{-1} (y_i^+ - \alpha_t) = R_{\ti}^{-1} (y_i^+ - \alpha_{\ti}) - \int_t^{\ti} R^{-1}_s L_s Q V_s + R_s^{-1} L_s a_s \, \df s, \quad t \in (\tim,\ti),
\end{align*}
and consequently, using that $L^*_t = S^*_{s-t} L^*_s$ for all $s \in [t,\ti]$ and the definition of $U_s$,
\begin{align*}
    V_t &= L^*_t R_t^{-1} (y_i^+ - \alpha_t) \\
    &= L^*_t R_{\ti}^{-1} (y_i^+ - \alpha_{\ti}) - L_t^* \left( \int_t^{\ti} R^{-1}_s L_s Q V_s + R_s^{-1} L_s a_s \, \df s \right) \\
    &= S^*_{\ti-t} L^*_{\ti} R_{\ti}^{-1} (y_i^+ - \alpha_{\ti}) -  \left( \int_t^{\ti} S^*_{s-t} L^*_sR^{-1}_s L_s Q V_s + S^*_{s-t} L^*_sR_s^{-1} L_s a_s \, \df s \right) \\
    &= S^*_{\ti-t} V_{\ti} - \left( \int_t^{\ti} S^*_{s-t} U_s Q V_s + S^*_{s-t} U_s a_s \, \df s \right), \quad t \in (\tim,\ti).
\end{align*}
This, together with the fact that $V_{\ti} = L^* \Sigma^{-1} y_i + V^+_{\ti}$ follows from plugging in $L_{\ti}, R_{\ti}$ and $\alpha_{\ti}$ as given by Theorem \ref{thm: g_LtRtalphat} into the definition of $V_t$, shows that $V_t$ is a mild solution to \eqref{eq: dV_t}. Moreover, uniqueness of $V$ is a direct consequence of classical results concerning the uniqueness of mild solutions to the abstract Cauchy problem.
\end{proof}

\begin{proof}[Proof of Proposition \ref{prop: dct}]
    The first claim follows from taking the log of $g(t,x)$ in \eqref{eq: g_BIF_computed}. 
    To derive the ODE \eqref{eq: dct}, first note that 
    \begin{align*}
        \dfrac{\df}{\df t} \det(R_t) &= \det(R_t) \tr \left[ R^{-1}_t \left( \dfrac{\df}{\df t} R_t \right) \right] = - \det(R_t) \tr \left[ R^{-1}_t L_t Q L^*_t \right] = - \det(R_t) \tr \left[ U_t Q  \right].
    \end{align*}
    on any $(\tim, \ti).$ Here, we use \eqref{eq: a918o273} in the second step and the definition of $U_t$ in the last. It follows that 
    \begin{align}
    \label{eq: ap9182732398}
     \dfrac{\df}{\df t}  \log(\det(R_t)) = - \tr \left[ U_t Q  \right].
    \end{align}
    Furthermore, plugging in $\df \alpha_t = -L_t a_t \df t$ and \eqref{eq: a129o387} it holds that 
    \begin{align}
    \label{eq: ao1987263123}
    \begin{split}
        \dfrac{\df}{\df t} \left \langle y^+_i - \alpha_t, R^{-1}_t (y^+_i - \alpha_t)\right \rangle  &= \left \langle  \dfrac{\df}{\df t} (y^+_i - \alpha_t), R^{-1}_t (y^+_i - \alpha_t)\right \rangle  \\
        &\quad+ \left \langle  y^+_i - \alpha_t, \dfrac{\df}{\df t} \left( R^{-1}_t (y^+_i - \alpha_t)\right) \right \rangle  \\
        &= \left \langle   L_t a_t , R^{-1}_t (y^+_i - \alpha_t)\right \rangle  \\
        &\quad+ \left \langle  y^+_i - \alpha_t, R^{-1}_t L_t Q L^*_t R^{-1}_t (y^+_i - \alpha_t) + R^{-1}_t L_t a_t \right \rangle \\
        &= 2 \left \langle   a_t ,  L^*_t R^{-1}_t (y^+_i - \alpha_t)\right \rangle + \left \langle  y^+_i - \alpha_t, R^{-1}_t L_t Q L^*_t R^{-1}_t (y^+_i - \alpha_t) \right \rangle \\
        &= 2 \left \langle   a_t ,  V_t \right \rangle + \left \langle V_t, Q V_t \right \rangle.
    \end{split}
    \end{align}
    Hence, combining \eqref{eq: ap9182732398} and \eqref{eq: ao1987263123} gives 
    \begin{align*}
         \dfrac{\df}{\df t} c_t &= \dfrac{1}{2} \tr \left[ U_t Q  \right] - \left \langle   a_t ,  V_t \right \rangle -\dfrac{1}{2} \left \langle V_t, Q V_t \right \rangle, \quad t \in (\tim, \ti).
    \end{align*}
    Lastly, the expression for the terminal condition $c_{\ti}$ follows from plugging $R_{\ti}$ and $\alpha_{\ti}$ as given in Theorem \ref{thm: g_LtRtalphat} into the definition of $c_t$.
\end{proof}

\section{}
\label{app: B}

\subsection{Proofs of Section \ref{sec: amari}}

\begin{proof}[Proof of Proposition \ref{prop: amari_G}]
As noted in Remark \ref{rem: filtering_G}, the function $G_i$ for the one-step-ahead guiding distribution is given by
\begin{align*}
	G_i(t,x) = L_t^* R_t^{-1} \left( y_i - L_t x \right), \quad t \in [\tim,\ti],
\end{align*}
with $L_t = L S_{\ti-t}$ and $R_t = \Sigma + L Q_{\ti-t} L^*$.
Let us define, with slight abuse of notation, $\Delta_i := t_i - t$.

Given $A x = -x$, it holds that $(S_t)_t$ is the semigroup given by $S_{\Delta_i} = \exp(- \Delta_i)$.
Hence, following the definition of $Q_t$ given in \eqref{eq: def_Qt}, we have
\begin{align*}
	Q_{\Delta_i} &= \int_0^{\Delta_i} \exp (- 2 s ) Q \, \df s = \frac{1 - \exp(- 2 \Delta_i)}{2} Q. 
\end{align*}
Plugging this into the expression of $G_i$ above gives the first claim of Proposition \ref{prop: amari_G}.
The matrix representation of $L Q L^*$ in the second claim can be deduced from the fact that, for any $y \in \R^m$,
\begin{align*}
(L Q L^*) y &= LQ \left(\sum_{i=1}^m y_i \om_i \right) \\
&= L \left( \sum_{i=1}^m y_i \sum_{l=1}^{\infty} q_l \langle \om_i, e_l \rangle e_l \right) \\
&= \left( \sum_{i=1}^m y_i \sum_{l=1}^{\infty} q_l \langle \om_i, e_l \rangle \langle \om_j, e_l \rangle \right)_{j=1}^m \in \R^m.
\end{align*}

\end{proof}

\section{}
\label{app: C}

The idea of guiding bears similarities to a data assimilation technique known as nudging, where a dynamical system described by a partial differential equation is adjusted by superimposing a data-dependent term. This ``nudging'' term is designed to ``nudge'' trajectories of the system  towards  observations. The underlying idea  dates back to the 1960s and 1970s and is also known as Newtonian relaxation (\cite{hoke1976initialization}) and   Luenberger observer  in control theory for linear systems (\cite{luenberger1964observing}).  The proposed nudging term is rather ad-hoc and consequently does not satisfy any minimum error criterion. Recently, \cite{conti2025model} (see also \cite{conti2022physical}) derived  nudging terms in a more principled way. In particular, the two proposed nudging terms proposed in \cite{conti2025model} --PNDRQF (Physical Nudging Data-Retrieval Quadratic Form) and PNDRQ (Physical Nudging Data-Retrieval Quadrature)-- are special instances of guiding terms inherited from a Gaussian system. In \cite{Singh2025Dataassimilationusingglobal}, a particle filter for SPDEs is introduced in which the nudging term is optimised to increase the effective sample size of the ensemble of particles.

	The nudging approach differs from ours in several ways. Firstly, nudging is designed to tackle the filtering problem, and nudging terms are thereby restricted to one future observation. 
	In contrast, the guiding terms we propose can be generalised to take {\it all} future observations into account, thereby enabling methodology to efficiently address the smoothing problem.  
	Moreover, contrary to our approach, no Radon-Nikodym derivative for the law of the nudged process with respect to the true conditioned process is derived, precluding exact inference. 
	Lastly, current approaches derive the nudging terms \textit{after} discretising infinite-dimensional SDEs into a high-dimensional SDE. Guided processes, however, are derived in the fully infinite-dimensional Hilbert space setting, hence ensuring methodology that remains valid under mesh refinement.

\bibliographystyle{plainnat}  
\bibliography{references}  

\end{document}